\documentclass[12pt]{amsart}
\usepackage[margin=1in]{geometry}
\usepackage[T1]{fontenc}
\usepackage{enumerate}		
\usepackage{amssymb,amsfonts}
\usepackage{hyperref}
\hypersetup{
    colorlinks=true,
    linkcolor=blue,
    filecolor=magenta,      
    urlcolor=cyan,
    citecolor=magenta
}
\usepackage{cleveref}
\usepackage{amsthm}
\usepackage{amsmath,amscd}
\usepackage{graphicx}
\usepackage{lmodern}
\usepackage[english]{babel}
\usepackage{MnSymbol}
\usepackage{yfonts}
\usepackage{xfrac}
\usepackage{mathrsfs}

\newtheorem{thm}{Theorem}[section]
\newtheorem{cor}[thm]{Corollary}
\newtheorem{lem}[thm]{Lemma}
\newtheorem{prop}[thm]{Proposition}
\newtheorem{claim}[thm]{Claim}
\theoremstyle{definition}
\newtheorem*{defn}{Definition}
\newtheorem{conj}[thm]{Conjecture}
\newtheorem{ex}[thm]{Example}

\theoremstyle{remark}
\newtheorem{remark}[thm]{Remark}

\numberwithin{equation}{section}

\title{Generic solutions of equations involving the modular $j$ function}
\author{Sebastian Eterovi\'c}
\address{School of Mathematics, University of Leeds, Leeds, UK} 
\email{s.eterovic@leeds.ac.uk}

\date{\today}

\thanks{Supported by NSF RTG grant DMS-1646385 and EPSRC fellowship EP/T018461/1. 
I would like to thank Vahagn Aslanyan, Sebasti\'an Herrero, Vincenzo Mantova, Adele Padgett, Thomas Scanlon and Roy Zhao for helpful discussions on some of the details presented here}

\thanks{ORCiD: 0000-0001-6724-5887}



\begin{document}

\begin{abstract}
    Assuming a modular version of Schanuel's conjecture and the modular Zilber--Pink conjecture, we show that the existence of generic solutions of certain families of equations involving the modular $j$ function can be reduced to the problem of finding a Zariski dense set of solutions. By imposing some conditions on the field of definition of the variety, we are also able to obtain versions of this result without relying on these conjectures, and even a result including the derivatives of $j$. 
\end{abstract}

\maketitle

\section{Introduction}
In this paper we study the strong part of the \emph{Existential Closedness Problem} (\emph{strong EC} for short) for the modular $j$ function. The strong EC problem asks to find minimal geometric conditions that an algebraic variety $V\subset\mathbb{C}^{2n}$ should satisfy to ensure that for every finitely generated field $K$ over which $V$ is defined, there exists a point $(z_1,\ldots,z_n)$ in $\mathbb{H}^{n}$ such that $(z_1,\ldots,z_n,j(z_1),\ldots,j(z_n))$ is a point of $V$ which is generic over $K$. The results of \cite{aek-differentialEC} and \cite{aslanyan-kirby} inform what the conditions of strong EC should be: in technical terms, it is expected that \emph{broadness} and \emph{freeness} is the minimal set of conditions (see \textsection\ref{subsec:broadfree} for definitions and Conjecture \ref{conj:ec} for a precise statement). 

When approaching the strong EC problem, the first immediate obstacle is the \emph{EC problem}, which simply asks to find geometric conditions that an algebraic variety $V\subset\mathbb{C}^{2n}$ should satisfy to ensure that $V$ has a Zariski dense set of points in the graph of the $j$ function. When $V$ has such a Zariski dense set of points, we say that \emph{$V$ satisfies (EC)}. Following from the previous paragraph, it is expected that if $V$ is free and broad, then $V$ satisfies (EC) (see Conjecture \ref{conj:ec}). It was proven in \cite[Theorem 1.1]{paper1} that if $V\subset\mathbb{C}^{2n}$ is a variety such that the projection $\pi:V\rightarrow\mathbb{C}^{n}$ onto the first $n$ coordinates is dominant, then $V$ satisfies (EC). This is a partial solution to the EC problem, since asking that the projection $\pi$ is dominant is a stronger hypothesis than the notion of broadness.  

Regarding strong EC, it was shown in \cite[Theorem 1.2]{paper1} that if one assumes a modular version of Schanuel's conjecture, then any plane irreducible curve $V\subset\mathbb{C}^{2}$ which is not a horizontal or a vertical line has a generic point over any given finitely generated field over which it is defined. 
On the other hand, \cite[Theorem 1.1]{aek-closureoperator} provides a version of this without assuming the modular Schanuel conjecture, but instead assuming that $V$ is in some precise sense ``generic'' with respect to the $j$ function. 

The main results of this paper extend \cite[Theorem 1.2]{paper1} and \cite[Theorem 1.1]{aek-closureoperator} to higher dimensions. 
Our first main result shows that under a modular version of Schanuel's conjecture (which we call MSCD, see Conjecture \ref{conj:mscd}) and the modular Zilber--Pink conjecture (MZP, see Conjecture \ref{conj:mzp}) the strong EC problem can be reduced to the EC problem.
\begin{thm}
\label{thm:generic}
Let $K\subseteq\mathbb{C}$ be a finitely generated field, and let $V\subseteq\mathbb{C}^{2n}$ be a broad and free variety defined over $K$ satisfying (EC). 
Then MSCD and MZP imply that $V$ has a point of the form $(z_1,\ldots,z_n,j(z_1),\ldots,j(z_n))$, with $(z_1,\ldots,z_n)\in\mathbb{H}^{n}$, which is generic over $K$.
\end{thm}
This addresses the second question posed in \cite[\textsection 1]{paper1}, and as such, this paper can be seen as a continuation of the work done there. 
We will also show some general cases in which we can remove the dependence on MZP, see Theorem \ref{thm:domproj2}. 
In particular, this has consequences on the dynamical behaviour of the $j$ function, see Corollary \ref{cor:iterates1}.

Using that the Ax--Schanuel theorem for $j$ (\cite{pila-tsimerman}) implies weak forms of both MSCD and MZP, we are able to remove the dependency on these conjectures from Theorem \ref{thm:generic} by instead imposing conditions on the field of definition of $V$, conditions we are calling ``having no $C_{j}$-factors'' (see \S\ref{subsec:unconditional} for the definition). 
Here $C_j$ is a specific countable algebraically closed subfield of $\mathbb{C}$ which is built by solving systems of equations involving the $j$ function, and the condition of ``having no $C_j$-factors'' ensures that $V$ is sufficiently generic so that the weak forms of MSCD and MZP suffice. 
A version of this condition already appeared in the hypotheses of \cite[Theorem 1.1]{aek-closureoperator}.
We now state our second main result.
\begin{thm}
\label{thm:unconditional}
Let $V\subseteq\mathbb{C}^{2n}$ be a broad and free variety with no $C_{j}$-factors and satisfying (EC). 
Then for every finitely generated field $K$ containing the field of definition of $V$, there is a point in $V$ of the form $(z_1,\ldots,z_n,j(z_1),\ldots,j(z_n))$, with $(z_1,\ldots,z_n)\in\mathbb{H}^{n}$, which is generic over $K$.
\end{thm}
We will also prove a version of Theorem \ref{thm:unconditional} which includes the derivatives of $j$, see Theorem \ref{thm:unconditionalwithderivatives}.

The EC problem for $j$ is still open, but certain partial results have been achieved (we already mentioned \cite[Theorem 1.1]{paper1}). 
One of the main results of \cite{aslanyan-kirby} proves an approximate solution to the EC problem. 
Their approach is referred to as a \emph{blurring} of the $j$-function (definition can be found in \S\ref{subsec:blurring}). 
Using this, we will deduce the following in \S\ref{subsec:thmblurred}. 
\begin{thm}
\label{thm:blurred}
Let $V\subseteq\mathbb{C}^{2n}$ be a broad and free variety with no $C_{j}$-factors. 
Then for every finitely generated field $K\subset\mathbb{C}$ containing the field of definition of $V$, there are matrices $g_{1},\ldots,g_{n}$ in $ \mathrm{GL}_2^+(\mathbb{Q})$ such that $V$ has a point of the form  
\[(z_{1},\ldots,z_{n},j(g_{1}z_{1}),\ldots,j(g_{n}z_{n})),\]
with $(z_1,\ldots,z_n)\in \mathbb{H}^n$, which is generic over $K$.
\end{thm}

\subsection{Summary of the proof of Theorem \ref{thm:generic}}
\label{subsec:keyingredients}
\label{subsec:keys}
In order to reduce Strong EC to EC under MSCD and MZP, there are two main technical steps.
\begin{enumerate}[(a)]
    \item MSCD gives a lower bound for transcendence degree, but this inequality is only measured over $\mathbb{Q}$, whereas strong EC requires one to measure transcendence degree over an arbitrary finitely generated field, that is, we need a modular Schanuel-inequality \emph{with parameters}. 
    This is resolved by using the results on the existence of ``convenient generators'' of \cite[\textsection 5]{aek-closureoperator}, which in turn is based upon the results of \cite{aek-differentialEC} on differential existential closedness. See \textsection\ref{subsec:transineq} for details.
    \item Under MSCD, a point in $V$ of the form $(z_{1},\ldots,z_{n},j(z_{1}),\ldots,j(z_{n}))$ which is not generic in $V$ will produce what is known as an \emph{atypical intersection}. 
    MZP speaks precisely about such atypical components, giving us sufficient control over their behaviour. 
\end{enumerate}

\subsection{Strong EC for \texorpdfstring{$\exp$}{exp}}
The motivation for the EC and strong EC problems for $j$ originates in the study of analogous problems stated for the complex exponential function $\exp:\mathbb{C}\to\mathbb{C}^\times$, which where first considered in Zilber's work on pseudo-exponentiation \cite{zilberexp}, with further details and results in \cite{kirby-zilber} and \cite{bays-kirby}. 
Zilber's work gives a model-theoretic approach to the study of the algebraic properties of the complex exponential function, and his ideas have since been expanded to many other settings.

With this in mind, the motivation for Theorem \ref{thm:generic} is not simply a restating of Zilber's conjectures on $\exp$ for the case of $j$, but our aim is also to give a general strategy for reducing strong EC problems to EC problems, \emph{even though} there is no nice model-theoretic ``pseudo-$j$'' structure like pseudo-exponentiation (yet). 
The methods presented here are expected to work in more general situations, such as in the case of Shimura varieties for which some cases of (EC) have been shown in \cite{eterovic-zhao}. 

The original name of the EC problem in the case of $\exp$ was \emph{Exponential Algebraic Closedness} (EAC). 
Instead of requiring the variety $V$ to be broad and free, EAC requires the varieties  $V\subseteq\mathbb{C}^{n}\times\left(\mathbb{C}^{\times}\right)^{n}$ to be \emph{rotund}, and additively and multiplicatively \emph{free}, see \cite[\textsection 3.7]{kirby-zilber} for definitions. 
The original name for the Zilber--Pink conjecture in the context of the exponential function was the \emph{conjecture on the intersection with tori} (CIT). 

The result below is a direct analogue of Theorem \ref{thm:generic} for $\exp$, and we reference \cite[Theorem 1.5]{kirby-zilber} as a source.
We point out however that the original formulation of \cite[Theorem 1.5]{kirby-zilber} does not require the variety $V$ to be additively and multiplicatively free, although these conditions are necessary for the theorem to hold. 
Without them, we cannot even expect the variety to intersect the graph of $\exp$. 
For example, if we define a variety $V\subset \mathbb{C}^2\times\left(\mathbb{C}^\times\right)^2$ by the following two equations: $X_1 - X_2 = 0$ on $\mathbb{C}^2$ (which prevents $V$ from being additively free) and $Y_1 - Y_2 = 1$ on $\left(\mathbb{C}^\times\right)^2$, then $V$ cannot have point in the graph of any function. 
Not only that, if we slightly modify $V$ to be defined by the equations $X_1 - X_2 = 0$ and $X_1=Y_1$, it can be checked that $V$ has an infinite intersection with the graph of $\exp$ (one for every fixed point of $\exp$), but every point in this intersection has transcendence degree at most 1, so they are not generic in $V$. 
\begin{thm}[{{\cite[Theorem 1.5]{kirby-zilber}}}]
\label{thm:expgeneric}
Let $K\subseteq\mathbb{C}$ be a finitely generated field, and let $V\subseteq\mathbb{C}^{n}\times\left(\mathbb{C}^{\times}\right)^{n}$ be an algebraic variety which is rotund, additively free, multiplicatively free, defined over $K$, and satisfying (EAC). 
Then Schanuel's conjecture and CIT imply that $V$ has a point of the form $(\mathbf{z},\exp(\mathbf{z}))$ which is generic over $K$.
\end{thm}
Since our methods for proving Theorem \ref{thm:generic} are expected to easily generalise to give a proof of Theorem \ref{thm:expgeneric}, and many aspects of the such a proof can already be found in \cite{bays-kirby} and \cite{kirby-zilber}, we will not present a proof of this result here. 
We remark that Theorem \ref{thm:expgeneric} only considers complex algebraic varieties, not a general model of the first-order theory of pseudo-exponentiation, and as such it falls short of the full ambition of \cite[Theorem 1.5]{kirby-zilber}.
 
On the other hand, we expect that by imposing conditions on the base field of $V$ (analogous to our notion of ``having no $C_{j}$-factors'') one can proceed like in the proof of Theorem \ref{thm:unconditional} to remove the dependence on Schanuel's conjecture and CIT from Theorem \ref{thm:expgeneric}. 
Although not phrased in this way, many aspects of such a result can be recovered from \cite[Proposition 11.5]{bays-kirby}, with some extra details provided by \cite[Theorem 5.6]{aek-closureoperator}.

Naturally, given the similarities between the results for $\exp$ and $j$ one should ask: when can one expect to obtain analogous results with other functions? 
The key ingredient in obtaining the existence of the ``convenient generators'' mentioned in \S\ref{subsec:keyingredients} is the Ax--Schanuel theorem. 
As we will see, this theorem is also essential for obtaining uniform weak forms of Zilber--Pink, which are then used in proving Theorem \ref{thm:unconditional}. 
We therefore expect that the methods used here can be extended to other situations where one has such a result. 
Ax--Schanuel theorems have been been obtained in many situations: the exponential map of a (semi-) abelian variety in \cite{ax2} and \cite{kirby-semiab}, uniformizers of any Fuchsian group of the first kind in \cite{bcfn} (and more), the uniformisation map of a Shimura variety in \cite{mpt}, variations of mixed Hodge structures in \cite{gao-klingler,chiu-mixedaxschanuelwithder}, among many others.

\subsection{Structure of the paper}
\begin{enumerate}
    \item[\textsection 2:] We set up the basic notation and give some background on the algebraic aspects of the modular $j$ function.
    \item[\textsection 3:] We go over the technical details of broadness, freeness, the (EC) condition, and the uniform and weak versions of Zilber--Pink.
    \item[\textsection 4:] We first lay down the necessary groundwork for handling transcendence inequalities over different fields (using the convenient generators alluded to earlier), we then deal with the possible presence of special solutions (using uniform Andr\'e--Oort), and we prove Theorem \ref{thm:generic}. 
    After that, we give a few cases of Theorem \ref{thm:generic} were we can remove the dependence on MZP, without imposing conditions on the field of definition of $V$. 
    We also prove Theorem \ref{thm:genericblurred} in \S\ref{subsec:blurring}, which is a version of Theorem \ref{thm:generic} for a blurring of the $j$ function.
    \item[\textsection 5:] We prove Theorem \ref{thm:unconditional}. 
    This proof, combined with the technicalities explained in the proof of Theorem \ref{thm:genericblurred}, produce Theorem \ref{thm:blurred} in \S\ref{subsec:thmblurred}.
    \item[\textsection 6:] We give an analogue of Theorem \ref{thm:unconditional} including the derivatives of $j$.
\end{enumerate}

\section{Background}
\label{sec:background}

\subsection{Basic notation}
\begin{itemize}
    \item If $L$ is any subfield of $\mathbb{C}$, then $\overline{L}$ denotes the algebraic closure of $L$ in $\mathbb{C}$. 
    \item Given sets $A,B$ we define $A\setminus B:=\{a\in A:a\not \in B\}$. 
    \item Tuples of elements will be denoted with boldface letters; that is, if $x_{1},\ldots,x_{m}$ are elements of a set $X$, then we write $\mathbf{x}:=(x_{1},\ldots,x_{m})$ for the ordered tuple. We will also sometimes use $\mathbf{x}$ to denote the (unordered) set $\left\{x_{1},\ldots,x_{m}\right\}$, which should not lead to confusion. 
    \item Suppose $X$ is a non-empty subset of $\mathbb{C}^{m}$. If $f$ denotes a function defined on $X$ and $\mathbf{x}$ is an element of $X$, then we write $f(\mathbf{x})$ to mean $(f(x_{1}),\ldots,f(x_{m}))$. 
    \item Given a subfield $L$ of $\mathbb{C}$ and a subset $A\subseteq\mathbb{C}^{n}$, we say that a subset $Z\subseteq\mathbb{C}^{n}$ is the \emph{$L$-Zariski closure of $A$} if $Z$ is the smallest Zariski closed set containing $A$ that is defined over $L$. 
    \item The term \emph{algebraic variety} for us will just mean a Zariski closed set, not necessarily irreducible. We also identify complex algebraic varieties $V$ with the set of their $\mathbb{C}$-points $V(\mathbb{C})$. 
\end{itemize}

\subsection{The \texorpdfstring{$j$}{j}-function}
We denote by $\mathbb{H}$ the complex upper-half plane $\{z\in \mathbb{C}:\mathrm{Im}(z)>0\}$. 
The group $\mathrm{GL}_2^+(\mathbb{R})$ of $2$ by $2$ matrices with coefficients in $\mathbb{R}$ and positive determinant, acts on $\mathbb{H}$ via the formula 
\[gz:=\frac{az+b}{cz+d} \ \text{ for } \ g=\left(\begin{array}{cc}a & b\\ c & d\end{array}\right) .\]
This action can be extended to a continuous action of $\mathrm{GL}_2^+(\mathbb{R})$ on the Riemann sphere $\widehat{\mathbb{C}}:=\mathbb{C}\cup \{\infty\}$. 
Given a subring $R$ of $ \mathbb{R}$ we define $M_2^+(R)$ as the set of $2$ by $2$ matrices with  coefficients in $R$ and positive determinant. 
We put 
\[G:=\mathrm{GL}_2^+(\mathbb{Q})= M_2^+(\mathbb{Q}),\]
which is a subgroup of $\mathrm{GL}_2^+(\mathbb{R})$. 
The modular group is defined as 
\[\Gamma:=\mathrm{SL}_2(\mathbb{Z})=\{g\in M_2^+(\mathbb{Z}):\det(g)=1\}.\]
The modular $j$ function is defined as the unique holomorphic function $j:\mathbb{H}\to \mathbb{C}$ that satisfies  
\[j(gz)=j(z) \text{ for every }g \text{ in } \Gamma \text{ and every }z\text{ in }\mathbb{H},\]
and has a Fourier expansion of the form 
\begin{equation}\label{eq:j-fourier-expansion}
j(z)=q^{-1}+744+\sum_{k=1}^{\infty}a_kq^k \text{ with }q:=\exp(2\pi i z)  \text{ and }a_k\in \mathbb{C}.
\end{equation} 
This function allows us to identify $\Gamma \backslash \mathbb{H}\simeq \mathbb{C}$. 
The quotient space $Y(1):=\Gamma\backslash \mathbb{H}$ is known to be a 
(coarse) moduli space for one-dimensional complex tori, or equivalently, elliptic curves over $\mathbb{C}$. 
If $\Gamma z$ is a point in $Y(1)$ and $E_z$ denotes an elliptic curve in the corresponding isomorphism class, then $j(z)$ is simply the $j$-invariant of the curve $E_z$. 

It is well-known (\cite{mahler}) that $j$ satisfies the following algebraic differential equation (and none of lower order):
\begin{equation}
    \label{eq:j}
    0 = \frac{j'''}{j'} - \frac{3}{2}\left(\frac{j''}{j'}\right)^{2} + \frac{j^{2} -1968j + 2654208}{j^{2}(j-1728)^{2}}\left(j'\right)^{2}.
\end{equation}

\subsection{Modular polynomials}
\label{sec:modular_polynomials} 
Let $\left\{\Phi_{N}(X,Y)\right\}_{N=1}^{\infty}\subseteq\mathbb{Z}[X,Y]$ denote the family of \emph{modular polynomials} associated with $j$ (see \cite[Chapter 5, Section 2]{lang} for the definition and main properties of this family). 
We recall that $\Phi_{N}(X,Y)$ is irreducible in $\mathbb{C}[X,Y]$, $\Phi_{1}(X,Y) = X-Y$, and for $N\geq 2$, $\Phi_{N}(X,Y)$ is symmetric of total degree $\geq 2N$. 
Also, the action of $G$ on $\mathbb{H}$ can be traced by using modular polynomials in the following way: for every $g$ in $G$ we define $\widetilde{g}$ as the unique matrix of the form $rg$ with $r\in \mathbb{Q}$ and $r>0$, so that the entries of $\widetilde{g}$ are all integers and relatively prime. 
Then, for every $x$ and $y$ in $\mathbb{H}$ the following statements are equivalent: 
\begin{itemize}
    \item[(M1):] $\Phi_{N}(j(x),j(y)) = 0$;
    \item[(M2):] There exists $g$ in $G$ with $gx=y$ and $\det\left(\widetilde{g}\right) = N$.
\end{itemize}
\begin{defn}
A finite set $A\subset \mathbb{C}$ is said to be \emph{modularly independent} if for every pair of distinct numbers $a,b$ in $A$ and every positive integer $N$, we have that $\Phi_{N}(a,b)\neq 0$. 
Otherwise, we say that $A$ is \emph{modularly dependent}.

An element $w$ is said to be \emph{modularly dependent over $A$} if there is $a\in A$ such that the set $\{a,w\}$ is modularly dependent.
\end{defn}

\begin{defn}
Given a subset $A\subseteq\mathbb{C}$, the \emph{Hecke orbit} of $A$ is defined as
$$\mathrm{He}(A):=\left\{z\in\mathbb{C} : \exists a\in A\exists N\in \mathbb{N}(\Phi_N(z,a) = 0)\right\}.$$
\end{defn}

\begin{remark}
\label{rem:heckeorbit}
Let $d$ be a positive integer. 
As explained in \cite[\textsection 7.3]{paper1}, combining isogeny estimates of Masser--W\"ustholz and Pellarin, with gonality estimates for modular curves, we obtain that for every $z\in\mathbb{C}$ we have that the set 
\[\mathrm{He}_{d}(z):=\{w\in\mathrm{He}(z) : [\mathbb{Q}(z,w):\mathbb{Q}(z)]\leq d\}\]
is finite. 
This immediately gives that if $A\subset\mathbb{C}$ is a finite set, then
\[\mathrm{He}_{d}(A):=\{w\in\mathrm{He}(A) : [\mathbb{Q}(A,w):\mathbb{Q}(A)]\leq d\}\]
is also finite.
\end{remark}

\subsection{Special points}
\label{subsec:specialpts}
A point $w$ in $\mathbb{C}$ is said to be \emph{special} (also known as a \emph{singular modulus}) if there is $z$ in $\mathbb{H}$ such that $[\mathbb{Q}(z):\mathbb{Q}] = 2$ and $j(z)=w$. 
Under the moduli interpretation of $j$, special points are those that correspond to elliptic curves endowed with complex multiplication. 
We set  
\[\Sigma:=\left\{z\in\mathbb{H} : [\mathbb{Q}(z):\mathbb{Q}] = 2\right\}.\]
A theorem of Schneider \cite{schneider} says that $\mathrm{tr.deg.}_{\mathbb{Q}}\mathbb{Q}(z,j(z)) = 0$ if and only if $z$ is in $\Sigma$. 
We say that a point $\mathbf{w}$ in $\mathbb{C}^{n}$ is \emph{special} if every coordinate of $\mathbf{w}$ is special.

\subsection{The modular Schanuel conjecture}
\label{subsec:msc}
Given a subset  $A$ of $\mathbb{H}$, we define $\dim_{G}(A)$ as the number of distinct $G$-orbits in  
\[G\cdot A=\{g a:g\in G,a\in A\}\]
(which may be infinite). 
Equivalently, $\dim_G(A)$ is the cardinality of the quotient set $G\backslash (G\cdot A)$. 
Given another subset $C$ of $\mathbb{C}$, we define $\dim_{G}(A|C)$ as the number of distinct $G$-orbits in $(G\cdot A)\setminus (G\cdot C)$. 
In plain words, $\dim_{G}(A|C)$ counts the number of orbits generated by elements of $A$ that do not contain elements of $C$. 

We now present two modular versions of Schanuel's conjecture. 
Both follow from the Grothendieck--Andr\'e generalised period conjecture applied to powers of the modular curve (see \cite[\textsection 6.3]{aek-closureoperator} and references therein). 

\begin{conj}[Modular Schanuel conjecture (MSC)]
\label{conj:msc}
For every $z_{1},\ldots,z_{n}$ in $\mathbb{H}$:
\[\mathrm{tr.deg.}_{\mathbb{Q}}\mathbb{Q}\left(\mathbf{z},j(\mathbf{z})\right)\geq \dim_{G}(\mathbf{z}|\Sigma).\]    
\end{conj}

\begin{conj}[Modular Schanuel conjecture with derivatives (MSCD)]
\label{conj:mscd}
For every $z_{1},\ldots,z_{n}$ in $\mathbb{H}$:
\[\mathrm{tr.deg.}_{\mathbb{Q}}\mathbb{Q}\left(\mathbf{z},j(\mathbf{z}),j'(\mathbf{z}),j''(\mathbf{z})\right)\geq 3\dim_{G}(\mathbf{z}|\Sigma).\]    
\end{conj}
Clearly MSC follows from MSCD. 
Also, in view of the differential equation \ref{eq:j}, MSCD only needs to involve up to the second derivative of $j$.
We also point out that MSCD is somewhat incomplete as it does not say anything about the transcendence degree of $\mathbb{Q}(z,j(z),j'(z),j''(z))$ over $\mathbb{Q}$ when $z\in\Sigma$. 
In fact, a stronger conjecture is implied by the generalised period conjecture (see \cite[Conjecture 6.13]{aek-closureoperator}) which does account for points in $\Sigma$. 
However, in the proofs of our main results we will show that we can move away from points that have coordinates in $\Sigma$, and for this reason, MSCD will suffice for us. 

\begin{remark}
\label{remark:geometricmsc}
Using the equivalence between (M1) and (M2) in \textsection\ref{sec:modular_polynomials}, it is easy to see that MSC can be restated equivalently using a more geometric language (see also \cite[\textsection 15]{pila2}): For every algebraic variety $V\subseteq\mathbb{C}^{2n}$ defined over $\overline{\mathbb{Q}}$ for which there exists $(\mathbf{z},j(\mathbf{z}))\in V$, if $\dim V<n$ then $\dim_{G}(\mathbf{z}|\Sigma) < n$. 

This way of thinking about MSC leads one naturally to a counterpart problem for MSC, that is, the question of determining which subvarieties $V\subseteq\mathbb{C}^{2n}$ intersect the graph of the $j$-function. 
This gives rise to the existential closedness problem, which we discuss next.
\end{remark}

\section{Existential Closedness and the Zilber--Pink Conjecture}
\label{sec:zilber-pink}

In this section we define the notions of \emph{broad} and \emph{free}, we define the existential closedness conjecture for the $j$ function, we recall the modular Zilber--Pink conjecture, and we give some partial results. 

To do all this, we need to distinguish a family of subvarieties that will be called \emph{special}. 
For this, it will be important to differentiate the first $n$ coordinates of $\mathbb{C}^{2n}$ from the last $n$ coordinates because we want to think of $\mathbb{C}^{2n}$ as a product $\mathbb{C}^{n}\times \mathbb{C}^{n}$ where the left factor has special subvarieties coming from the action of $G$, whereas on the second factor the special subvarieties are defined by modular polynomials. 
So as to avoid confusion, we prefer to change the name of the second factor to $\mathrm{Y}(1)^{n}$, as in this case we are thinking of $\mathbb{C}$ as the modular curve $\mathrm{Y}(1)$. 
This way, we view $\mathbb{C}^{2n}$ as $\mathbb{C}^{n}\times \mathrm{Y}(1)^{n}$. This is purely done for purposes of notation.

We will always think of the subvarieties of $\mathbb{C}^{n}\times \mathrm{Y}(1)^{n}$ as being defined over the ring of polynomials $\mathbb{C}[X_{1},\ldots,X_{n},Y_{1},\ldots,Y_{n}]$. 

We will denote by $\pi_{\mathbb{C}}:\mathbb{C}^{n}\times \mathrm{Y}(1)^{n}\to\mathbb{C}^{n}$ and $\pi_{\mathrm{Y}}:\mathbb{C}^{n}\times \mathrm{Y}(1)^{n}\to\mathrm{Y}(1)^{n}$ the corresponding projections.

\subsection{Broad and free varieties}
\label{subsec:broadfree}
Given a matrix $g=\left(\begin{array}{cc}
    a & b \\
    c & d
\end{array}\right)$ in $\mathrm{GL}_2(\mathbb{C})$, we define the polynomial $M_g(X,Y):=Y(cX+d)-(aX+b)$ and the rational function $gX:=\frac{aX+b}{cX+d}$. 
Note that $M_g(X,gX)=0$.

\begin{defn}
A \emph{M\"obius subvariety} of $\mathbb{C}^n$ is a variety defined by finitely many equations of the form $M_{g_{i,k}}(X_i,X_k)$ with $g_{i,k}$ in $\mathrm{GL}_2(\mathbb{C})$ non-scalar and $i,k$ in $\{1,\ldots,n\}$ not necessarily different.
\end{defn}

\begin{ex}\footnote{We thanks Sebasti\'an Herrero for providing this example.}
The variety $V=\{(x_1,x_2,x_3)\in \mathbb{C}^3:x_1-2x_2+3x_3=0\}$ is not a M\"obius subvariety of $\mathbb{C}^3$ but it contains infinitely many M\"obius subvarieties. 
Indeed, for every integer $m$ define $g_m=\left(\begin{array}{cc}
    3m-1 & 0 \\
    0 & 1
\end{array}\right)$ and $h_m=\left(\begin{array}{cc}
    2m-1 & 0 \\
    0 & 1
\end{array}\right)$. 
Then $V$ contains $M_{g_m}(X_1,X_2)\cap M_{h_m}(X_1,X_3)$, which is a M\"obius subvariety when $m\geq 2$.
\end{ex}

\begin{defn}
A \emph{special subvariety} of $\mathrm{Y}(1)^{n}$ is an irreducible component of an algebraic set defined by equations of the following forms: 
\begin{enumerate}[(a)]
    \item  $\Phi_{N}(Y_{i},Y_{k}) = 0$, for some $N\in\mathbb{N}$, and
    \item $Y_{i} = \tau$, where $\tau\in\mathbb{C}$ is a special point. 
\end{enumerate}
We allow the set of equations to be empty, so $\mathrm{Y}(1)^{n}$ is itself a special variety. 
A special subvariety without constant coordinates is called a \emph{basic special subvariety}.

Since every special point can be obtained as the solutions of $\Phi_N(X,X)=0$ for some $N\in\mathbb{N}$, the condition $Y_{i} = \tau$ in the previous definition follows from the condition $\Phi_{N}(Y_{i},Y_{k}) = 0$ by choosing $i=k$, but we have opted for this presentation as it makes it easy to compare it with the definition of weakly special variety in the following paragraph. 

A subvariety of $\mathrm{Y}(1)^{n}$ is called \emph{weakly special} if it is an irreducible component of an algebraic set defined by equations of the following forms: 
\begin{enumerate}[(i)]
    \item $\Phi_{N}(Y_{i},Y_{k}) = 0$, for some $N\in\mathbb{N}$,
    \item $Y_{\ell} = d$, for some constant $d\in \mathbb{C}$.
\end{enumerate}
As it turns out, every special variety has a Zariski dense set of special points (see e.g.~\cite[1.4 Aside]{pila:andre-oort}), and so a weakly special subvariety is special if and only if it contains a special point.

If $S$ is a special subvariety of $\mathrm{Y}(1)^{n}$, then a \emph{(weakly) special subvariety of $S$} is a (weakly) special subvariety of $\mathrm{Y}(1)^{n}$ that is contained in $S$. 

If $T$ is a proper positive dimensional weakly special subvariety of $\mathrm{Y}(1)^{n}$, there are $m\in\{1,\ldots,n\}$, a basic special subvariety $S\subseteq\mathrm{Y}(1)^m$, and a point $p\in\mathrm{Y}(1)^{n-m}$ such that (up to re-indexing of the coordinates) $T$ can be written as $S\times\{p\}$.   
We define the \emph{basic complexity} of $T$ to be the maximal positive integer $N$ for which the modular polynomial $\Phi_{N}$ is required to define $S$. 
We denote this number by $\Delta_b(T)$ (which represents the complexity of the basic special part of $T)$.
This definition differs slightly from \cite[Definition 3.8]{habegger-pila2}, where the complexity is only defined for special varieties, and the complexity also depends on the constant special coordinates that the variety may have. 
\end{defn}

We remark that if $S$ and $T$ are (weakly) special subvarieties of $\mathrm{Y}(1)^{n}$, then the irreducible components of $S\cap T$ (if there are any) are again (weakly) special suvbarieties of $\mathrm{Y}(1)^{n}$. 
In fact, by Hilbert's basis theorem the irreducible components of any non-empty intersection of (weakly) special subvarieties is again a (weakly) special subvariety. 

\begin{defn}
Given an irreducible constructible set $X$, we denote by $\mathrm{spcl}(X)$ the \emph{special closure} of $X$, that is, the smallest special subvariety containing $X$. 
Similarly, we denote by $\mathrm{wspcl}(X)$ the \emph{weakly special closure} of $X$.
\end{defn}

Following \cite{vahagn3}, we also make the following definition.

\begin{defn}
Given a tuple $\mathbf{c} = (c_{1},\ldots,c_m)$ of complex numbers, we say that a subvariety $S$ of $\mathrm{Y}(1)^n$ is $\mathbf{c}$\emph{-special} if $S$ is weakly special and the values of each of the constant coordinates of $S$ (if there are any) is either an element of $\mathrm{He}(\mathbf{c})$, or a special point.
\end{defn}

\begin{defn}
We will say that an irreducible constructible  set $V\subseteq \mathbb{C}^{n}\times\mathrm{Y}(1)^{n}$ is \emph{modularly free} if $\pi_{\mathrm{Y}}(V)$ is not contained in a proper special subvariety of $\mathrm{Y}(1)^n$. 

We will also say that $V$ is \emph{free} if $V$ is modularly free, no coordinate of $V$ is constant, and $\pi_{\mathbb{C}}(V)$ is not contained in any M\"obius subvariety of $\mathbb{C}^{n}$ which is only defined by elements of $G$. 

We say that a constructible subset $V\subseteq \mathbb{C}^{n}\times\mathrm{Y}(1)^{n}$ is \emph{free} if every irreducible component of $V$ is free.
\end{defn}

Now we introduce some notation for coordinate projections. 
Let $n$ and $\ell$ denote positive integers with $\ell \leq n$, and let $\mathbf{i}=(i_1,\ldots,i_\ell)$ denote a point in $\mathbb{N}^{\ell}$ with $1\leq i_1 < \ldots < i_\ell \leq n$. 
Define the projection map $\mathrm{pr}_{\mathbf{i}}:\mathbb{C}^{n} \rightarrow \mathbb{C}^{\ell}$ by
\[\mathrm{pr}_{\mathbf{i}}:(x_1,\ldots,x_n)\mapsto (x_{i_1},\ldots,x_{i_\ell}).\] 
In particular we distinguish between the natural number $n$ and the tuple $\mathbf{n}=(1,\ldots,n)$. 
We also use the notation $\mathbf{n}\setminus\mathbf{i}$ to denote the tuple of entries of $\mathbf{n}$ that do not appear in $\mathbf{i}$.

\begin{remark}
If $T$ is a (weakly) special subvariety of $\mathrm{Y}(1)^{n}$, then for any choice of indices $1\leq i_{1}<\cdots<i_{\ell}\leq n$ we have that $\mathrm{pr}_{\mathbf{i}}(T)$ is a (weakly) special subvariety of $\mathrm{Y}(1)^{\ell}$. 
\end{remark}

Define $\mathrm{Pr}_{\mathbf{i}}:\mathbb{C}^{n}\times\mathrm{Y}(1)^{n}\rightarrow \mathbb{C}^{\ell}\times\mathrm{Y}(1)^{\ell}$ by
\[\mathrm{Pr}_{\mathbf{i}}(\mathbf{x},\mathbf{y}):= (\mathrm{pr}_{\mathbf{i}}(\mathbf{x}),\mathrm{pr}_{\mathbf{i}}(\mathbf{y})).\]

\begin{defn}
An algebraic set $V \subseteq \mathbb{C}^{n}\times\mathrm{Y}(1)^{n}$ is said to be \emph{broad} if for any $\mathbf{i}=(i_1,\ldots,i_{\ell})$ in $\mathbb{N}^{\ell}$ with $1\leq i_1 < \ldots < i_{\ell} \leq n$ we have $\dim \mathrm{Pr}_{\mathbf{i}} (V) \geq \ell$. 
In particular, if $V$ is broad then $\dim V\geq n$. 

We say $V$ is \emph{strongly broad} if the strict inequality $\dim \mathrm{Pr}_{\mathbf{i}} (V) > \ell$ holds for every $\mathbf{i}$.
\end{defn}

For example, if $\pi_{\mathbb{C}}(V)$ is Zariski dense in $\mathbb{C}^{n}$, then $V$ is broad.

\subsection{Existential Closedness}
\label{subsec:ec}
Here we will define the (EC) condition, give a conjecture for the EC problem for $j$ (see Conjecture \ref{conj:ec}), and review some known results.

Given an integer $n\geq 1$ we denote the graph of $j:\mathbb{H}^{n}\rightarrow\mathbb{C}^{n}$ as
\[\mathrm{E}_j^n:=\{(z_1,\ldots,z_n,j(z_1),\ldots,j(z_n)):z_1,\ldots,z_n \in \mathbb{H}\}.\]

\begin{defn}
We say that an algebraic variety $V\subseteq \mathbb{C}^{n}\times\mathrm{Y}(1)^{n}$ satisfies the \emph{Existential Closedness condition} (EC) if the set $V\cap \mathrm{E}_{j}^{n}$ is Zariski dense in $V$. 
\end{defn}

We recall the following result which gives examples of varieties satisfying (EC). 

\begin{thm}[{{\cite[Theorem 1.1]{paper1}}}]
\label{th:main1}
Let $V\subseteq \mathbb{C}^{n}\times\mathrm{Y}(1)^{n}$ be an irreducible algebraic variety. 
If $\pi_{\mathbb{C}}(V)$ is Zariski dense in $\mathbb{C}^{n}$, then $V$ satisfies (EC).
\end{thm}

We also recall the following statement (although it falls short from giving examples of varieties satisfying (EC) as it does not prove Zariski density). 

\begin{thm}[{{\cite[Theorem 3.31]{gallinaro}}}]
\label{thm:gallinaro}
Let $L\subseteq\mathbb{C}^{n}$ be a subvariety defined by equations of the form $M_{g}(X,Y)=0$, with $g\in\mathrm{GL}_{2}(\mathbb{R})$. 
Let $W\subseteq\mathrm{Y}(1)^{n}$ be a subvariety, and assume that $L\times W$ is a free and broad subvariety of $\mathbb{C}^{n}\times\mathrm{Y}(1)^{n}$. 
Then $j(L)$ is Euclidean dense in $W$. 
\end{thm}

These theorems and the main results of \cite{aek-differentialEC} give evidence for the following conjecture (cf \cite[Conjecture 1.2]{aslanyan-kirby}).

\begin{conj}
\label{conj:ec}
Let $n$ be a positive integer and let $V\subseteq \mathbb{C}^{n}\times\mathrm{Y}(1)^{n}$ be a broad and free variety such that $V\cap \left(\mathbb{H}^{n}\times\mathrm{Y}(1)^{n}\right)$ is Zariski dense in $V$. 
Then $V\cap \mathrm{E}_{j}^{n}\neq\emptyset$. 
\end{conj}
The reader may be wondering why, in light of Theorem \ref{th:main1}, Conjecture \ref{conj:ec} does not ask for the stronger condition of $V\cap \mathrm{E}_{j}^{n}$ being Zariski dense in $V$. 
In fact, as we will see in Lemma \ref{lem:usefultrick}, Conjecture \ref{conj:ec} already implies the this stronger condition.

Conjecture \ref{conj:ec} can be thought of as a version of Hilbert's Nullstellensatz for systems of algebraic equations involving the $j$-function (cf \cite{daquino-macintyre-terzo} for the discussion in the case of the exponential function). 
The examples of \cite[\textsection 3]{paper1} show that if $V$ has constant coordinates or is contained in a proper subvariety of the form $M\times\mathrm{Y}(1)^{n}$, with $M$ a M\"obius subvariety, then it can happen that $V\cap\mathrm{E}_{j}^{n}$ is empty. 

We also point out that the condition of $V\cap \left(\mathbb{H}^{n}\times\mathrm{Y}(1)^{n}\right)$ being Zariski dense in $V$ is needed, as otherwise the subvariety of $\mathbb{C}^2\times\mathrm{Y}(1)^2$ defined by the single equation $X_1 = iX_2$ is free and broad, but cannot have points in the graph of the $j$ function since there is no point $z$ in $\mathbb{H}$ for which $iz$ is also in $\mathbb{H}$. 
One could alternatively resolve this issue by extending the domain of $j$ to the lower half-plane using Schwarz reflection. 

\begin{remark}
Let $V\subseteq \mathbb{C}^{n}\times\mathrm{Y}(1)^{n}$ be a broad and free variety. 
If $\dim V=n$, then by \cite[Theorem 1.1]{pila-tsimerman} if $U$ is a component of $V\cap\mathrm{E}_{j}^{n}$ (assuming there are any), then $U$ has dimension zero unless $\pi_{\mathrm{Y}}(U) \subseteq T$, for some proper special subvariety $T\subset\mathrm{Y}(1)^{n}$. 
So for a generic $V$, we can only expect the intersection $V\cap\mathrm{E}_{j}^{n}$ to be at most countable.
\end{remark}

The following result is an approximate version of Conjecture \ref{conj:ec} using what is known as a \emph{blurring} of the $j$ function.

\begin{thm}[{{\cite[Theorem 1.9]{aslanyan-kirby}}}]
\label{thm:blurredj}
Let $V\subseteq \mathbb{C}^{n}\times\mathrm{Y}(1)^{n}$ be an irreducible variety which is broad and free. 
Then the set of points in $V$ of the form
\[(z_{1},\ldots,z_{n},j(g_{1}z_{1}),\ldots,j(g_{n}z_{n})),\]
such that $z_1,\ldots,z_n\in\mathbb{H}$ and $g_{1},\ldots,g_{n}\in G$, is Euclidean dense in $V\cap\left(\mathbb{H}^{n}\times\mathrm{Y}(1)^{n}\right)$. 
\end{thm}

We finish this subsection by recalling a useful trick. 

\begin{lem}[cf {{\cite[Proposition 4.34]{vahagn2}}}]
\label{lem:usefultrick}
Suppose Conjecture \ref{conj:ec} holds. 
Let $V\subseteq \mathbb{C}^{n}\times\mathrm{Y}(1)^{n}$ be a broad and free irreducible
variety. 
Then $V$ satisfies (EC). 
\end{lem}
\begin{proof} 
Let $p_{1},\ldots,p_{m}$ be polynomials in $\mathbb{C}[X_{1},\ldots,X_{n},Y_{1},\ldots,Y_{n}]$ defining $V$. 
Let $f$ be a polynomial in $\mathbb{C}[X_{1},\ldots,X_{n},Y_{1},\ldots,Y_{n}]$ that does not vanish identically on $V$ and let 
\[W = \left\{(\mathbf{x},\mathbf{y})\in V : f(\mathbf{x},\mathbf{y})=0\right\}.\] 
To prove the lemma it suffices to show that $\mathrm{E}_j^n\cap (V\setminus W)\neq \emptyset$. 
By Conjecture \ref{conj:ec} this is clear if $W=\emptyset$, hence we can assume that $f$ vanishes at least at one point of $V$.  
We now define the following subvariety\footnote{This is the standard Rabinowitsch trick used in the proof of Nullstellensatz.} of $\mathbb{C}^{n+1}\times\mathrm{Y}(1)^{n+1}$:
\begin{equation*}
    V':=\left\{ \begin{array}{ccc}
        p_{1}(x_{1},\ldots,x_{n},y_{1},\ldots,y_{n}) &=& 0  \\
         & \vdots & \\
         p_{m}(x_{1},\ldots,x_{n},y_{1},\ldots,y_{n}) &=& 0\\
         y_{n+1}f(x_{1},\ldots,x_{n},y_{1},\ldots,y_{n}) -1 &=& 0
    \end{array}\right\}.
\end{equation*}
Choose $1\leq i_{1}<\cdots<i_{\ell}\leq n+1$ and set $\mathbf{i}=(i_1,\ldots,i_\ell)$. Then
\[\dim\mathrm{Pr}_{\mathbf{i}}(V')=\left\{
\begin{array}{ll}
\dim\mathrm{Pr}_{\mathbf{i}}(V\setminus W)   &  \text{ if }i_{\ell}\neq n+1, \\
 \dim\mathrm{Pr}_{(i_1,\ldots,i_{\ell-1})}(V\setminus W) +1    &  \text{ if } \ell\geq 2 \text{ and }i_{\ell}= n+1,\\
 2 &  \text{ if } \ell=1 \text{ and }i_{1}= n+1.
\end{array}
\right.\]
Now, since $V$ is irreducible, the set $V\setminus W$ is dense in $V$. 
Hence, for any subtuple $\mathbf{k}$ of $\mathbf{n}$ we have, by continuity of $\mathrm{Pr}_{\mathbf{i}}$, that $\dim \mathrm{Pr}_{\mathbf{k}}(V\setminus W)=\dim \mathrm{Pr}_{\mathbf{k}}(V)$. 
Since $V$ is broad, we conclude that $V'$ is also broad.  

We will now prove that $V'$ is free. 
Since $V$ has no constant coordinates, and $f$ is non-constant, we see that no coordinate is constant on $V'$. 
Also, since $V$ is free, it is clear that $V'$ is not contained in a variety of the form $M\times \mathrm{Y}(1)^{n+1}$ where $M$ is a proper M\"obius subvariety of $\mathbb{C}^{n+1}$. 
Moreover, if $V$ is contained in a variety of the form $\mathbb{C}^{n+1}\times T$ where $T$ is a proper special subvariety of $\mathrm{Y}(1)^{n+1}$, then at least one of the polynomials defining $T$ must be of the from $\Phi_N(y_i,y_{n+1})$ with $i\in \{1,\ldots,n\}$ and $N\geq 1$.  
This implies that $\Phi_N(y_i,1/f)=0$ on $V\setminus W$, hence $V\setminus W$ is contained in the variety $Z$ defined by the polynomial $f^d\Phi_N(y_i,1/f)$ where $d$ is the degree of $\Phi_N$ in the $Y$ variable. 
This implies that $V\subseteq Z$. 
Since $\Phi_N(X,Y)$ has leading term $\pm 1$ as a polynomial in $Y$, it follows that $f$ has no zeroes on $Z$. 
But this implies that $f$ has no zeroes on $V$, which is a contradiction. 
This proves that $V'$ is free.

Since $V'$ is broad and free, Conjecture \ref{conj:ec} implies that there exists a point of the form $(\mathbf{z},j(\mathbf{z}))$ in $V'$ with $\mathbf{z}\in \mathbb{H}^{n+1}$, and so $\mathrm{Pr}_{\mathbf{n}}(\mathbf{z},j(\mathbf{z})) \in  (V\setminus W)$. 
This completes the proof of the lemma.
\end{proof}

\subsection{Atypical intersections and Zilber--Pink}
\label{subsec:atypical}

\begin{defn}
Suppose that $V$ and $W$ are subvarieties of a smooth algebraic variety $Z$. 
Let $X$ be an irreducible component of the intersection $V\cap W$. 
We say that $X$ is an \emph{atypical component of $V\cap W$ (in $Z$)} if
$$\dim X > \dim V + \dim W - \dim Z.$$
We say that the intersection \emph{$V\cap W$ is atypical (in $Z$)} if it has an atypical component. 
Otherwise, we say that \emph{$V\cap W$ is typical (in $Z$)}, i.e.~$V\cap W$ is typical in $Z$ if $\dim V\cap W = \dim V + \dim W - \dim Z$.

If $Z = \mathrm{Y}(1)^{n}$, we say that $X$ is an \emph{atypical component of $V$} if there exists a special subvariety $T$ of $\mathrm{Y}(1)^{n}$ such that $X$ is an atypical component of $V\cap T$. 
We remark that in this case, since $\dim T\geq \dim\mathrm{spcl}(X)$, it is also true that $X$ is an atypical component of $V\cap \mathrm{spcl}(X)$.

We say that $X$ is a \emph{strongly atypical component of $V$} if $X$ is an atypical component of $V$ and no coordinate is constant on $X$. 

An atypical (resp.~strongly atypical) component of $V$ is said to be \emph{maximal} (in $V$) if it is not properly contained in another atypical (resp.~strongly atypical) component of $V$.  

Given a tuple $\mathbf{c} = (c_1,\ldots,c_m)$ of complex numbers, an atypical component $X$ of $V$ is said to be $\mathbf{c}$\emph{-atypical} is $X$ is an atypical component of the intersection $V\cap T$, where $T$ is a $\mathbf{c}$-special subvariety.
\end{defn}

\begin{ex}
Let $T$ be a proper special subvariety of $\mathrm{Y}(1)^{n}$. 
Then $T$ is an atypical component of itself, since
\[\dim T = \dim T\cap T > \dim T + \dim T - n.\]
On the other hand, although $\mathrm{Y}(1)^{n}$ is a special variety, it is not atypical in itself.
\end{ex}

\begin{conj}[Modular Zilber--Pink]
\label{conj:mzp}
For every positive integer $n$, any subvariety of $\mathrm{Y}(1)^{n}$ has only finitely many maximal atypical components. 
\end{conj}

From now on, we will abbreviate this conjecture as MZP.
This conjecture is sometimes presented in terms of optimal varieties, which we discuss next. 

\begin{defn}
Let $V$ be a subvariety of $\mathrm{Y}(1)^{n}$. 
Given a subvariety $X\subseteq V$, we define the \emph{defect of $X$} to be
\[\mathrm{def}(X):= \dim\mathrm{spcl}(X) - \dim X.\]
We say that $X$ is \emph{optimal in $V$} is for every subvariety $W\subseteq V$ satisfying $X\subsetneq W$ we have that $\mathrm{def}(X)<\mathrm{def}(W)$. 
We let $\mathrm{Opt}(V)$ denote the set of all optimal subvarieties of $V$. 
Observe that always $V\in\mathrm{Opt}(V)$. 
We think of $\mathrm{Opt}(V)$ as a cycle in $\mathrm{Y}(1)^{n}$.
\end{defn}

\begin{remark}
\label{rem:maxatypisopt}
A maximal atypical component of $V$ is optimal in $V$. 
However, optimal subvarieties need not be maximal atypical.

On the other hand, if $X$ is a proper subvariety of $V$ which is optimal in $V$, then $\mathrm{def}(X) < \mathrm{def}(V)$ which implies that $\dim V\cap\mathrm{spcl}(X)\geq \dim X > \dim V + \dim\mathrm{spcl}(X) - \dim\mathrm{spcl}(V)$, so the intersection $V\cap\mathrm{spcl}(X)$ is atypical. 
\end{remark}

As shown in \cite[Lemma 2.7]{habegger-pila2}, MZP is equivalent to the statement that any subvariety of $\mathrm{Y}(1)^{n}$ contains only finitely many optimal subvarieties, i.e.~$\mathrm{Opt}(V)$ is a finite set.  

\begin{defn}
Let $S$ be a constructible set (resp.~an algebraic variety) in $\mathbb{C}^{N}$, where $N$ is some positive integer. 
A \emph{parametric family of constructible subsets (resp.~subvarieties) of $S$} is a constructible set $V\subseteq S\times Q$, where $Q\subseteq\mathbb{C}^{m}$ is another constructible set, which we denote as an indexed collection $V = (V_\mathbf{q})_{\mathbf{q}\in Q}$, where for each $q\in Q$ the set
\[V_{\mathbf{q}}:=\left\{\mathbf{s}\in S : (\mathbf{s},\mathbf{q})\in V\right\}\]
is a constructible subset (resp.~subvariety) of $S$.\footnote{In the terminology of model-theory, we can equivalently say that $V\subseteq S\times Q$ is a definable family of definable subsets of $S$, in the language of rings.}
\end{defn}

\begin{ex}
\label{ex:parametric family}
An important example for us of a parametric family is given by the following construction. 
Let $W$ be an algebraic subvariety of $\mathbb{C}^{n}\times\mathrm{Y}(1)^{n}$; we want to define the family of subvarieties of $W$ obtained by intersecting $W$ with all M\"obius subvareities of $\mathbb{C}^{n}$ defined by elements of $\mathrm{GL}_{2}(\mathbb{C})$. 

Given a function $f:D\to \mathrm{GL}_2(\mathbb{C})$ defined on a non-empty subset $D$ of $\{1,\ldots,n\}\times \{1,\ldots,n\}$, set
\[W_f:=\{(x_1,\ldots,x_{n},y_{1},\ldots,y_{n})\in W :  f(i,j)x_i=x_j \text{ for all }(i,j) \in D\}.\]
Then, the collection of all such $W_f$ forms a parametic family of subvarieties of $W$. 
Indeed, put
\[Q:=\bigsqcup_{\emptyset \neq D\subseteq \{1,\ldots,n\}^2} \mathrm{GL}_2(\mathbb{C})^{D}.\]
Every function $f:D\to \mathrm{GL}_2(\mathbb{C})$ can be represented as an element of $\mathrm{GL}_2(\mathbb{C})^{D}$. 
Since for every finite non-empty set $A$ we have that $\mathrm{GL}_2(\mathbb{C})^{A}$ is a constructible subset of $\mathbb{C}^{4m}$, where $m=\#A$, we have that $Q$ is a finite union of constructible sets, hence it is constructible. 
Then choosing $S=W$ and
\[V=\{((\mathbf{x},\mathbf{y}),f)\in W\times Q:(\mathbf{x},\mathbf{y})\in W_f\}\]
we have that $(W_f)_f$ is just the parametric family $(V_f)_{f\in Q}$ associated to $V$.
\end{ex}

Pila showed that MZP implies the following uniform version of itself.

\begin{thm}[Uniform MZP, see {{\cite[\S24.2]{pila_2022}}}]
\label{thm:UMZP}
Suppose that for every positive integer $n$, MZP holds for all subvarieties of $\mathrm{Y}(1)^{n}$. 
Let $(V_\mathbf{q})_{\mathbf{q}\in Q}$ be a parametric family of subvarieties of $\mathrm{Y}(1)^{n}$. 
Then there is a parametric family $(W_{\mathbf{p}})_{\mathbf{p}\in P}$ of closed algebraic subsets of $\mathrm{Y}(1)^{n}$ such that for every $\mathbf{q}\in Q$ there is $\mathbf{p}\in P$ such that $\mathrm{Opt}(V_{\mathbf{q}}) = W_{\mathbf{p}}$.
\end{thm}

\begin{cor}
\label{cor:UMZP}
Let $(V_{\mathbf{q}})_{\mathbf{q}\in Q}$ be a parametric family of constructible subsets of $\mathrm{Y}(1)^{n}$. 
Then MZP implies that there is a finite collection $\mathscr{S}$ of proper special subvarieties of $\mathrm{Y}(1)^{n}$ such that for all $\mathbf{q}\in Q$ and all special subvarieties $T$, if $X$ is an atypical component of $V_\mathbf{q}\cap T$, then there is $T_0\in\mathscr{S}$ such that $X\subseteq T_0$. 
\end{cor}
\begin{proof}
Let $(W_\mathbf{p})_{\mathbf{p}\in P}$ be the parametric family given by applying Theorem \ref{thm:UMZP} to $(V_{\mathbf{q}})_{\mathbf{q}\in Q}$. 
Without loss of generality we may assume that for all $\mathbf{p}\in P$ there is $\mathbf{q}\in Q$ such that $W_{\mathbf{p}}\subseteq V_\mathbf{q}$. 
Since $V_{\mathbf{q}}$ is always in $\mathrm{Opt}(V_{\mathbf{q}})$, we may remove this trivial optimal subvariety, and so we can assume that for every $\mathbf{p}\in P$, if $W_{\mathbf{p}}\subseteq V_\mathbf{q}$, then $V_{\mathbf{q}}$ is not in $W_{\mathbf{q}}$.

Let $T$ be a special subvariety of $\mathrm{Y}(1)^{n}$ and suppose that $\mathbf{q}\in Q$ is such that $V_\mathbf{q}\cap T$ contains an atypical component $X$. 
Then $T$ must be a proper subvariety of $\mathrm{Y}(1)^{n}$. $X$ is contained in a maximal atypical component of $V_\mathbf{q}$ which, by Remark \ref{rem:maxatypisopt}, is contained in $W_\mathbf{p}$ for some $\mathbf{p}\in P$. 

For every $\mathbf{p}\in P$, if $Z$ is in $W_\mathbf{p}$, then $Z$ is an optimal proper subvariety of $V_\mathbf{q}$, for some $\mathbf{q}\in Q$. 
By Remark \ref{rem:maxatypisopt} we know that 
\[\dim V\cap\mathrm{spcl}(Z) > \dim V + \dim\mathrm{spcl}(Z) - n,\]
and so in particular $\mathrm{spcl}(Z)$ is a proper special subvariety of $\mathrm{Y}(1)^{n}$. 
We conclude then that for every $p\in P$ there are finitely many proper special subvarieties $T_{1,\mathbf{p}},\ldots,T_{m,\mathbf{p}}$ such that
\begin{equation}
\label{eq:finitespecials}
    W_{\mathbf{p}}\subseteq\bigcup_{i=1}^{m}T_{i,\mathbf{p}}.
\end{equation}

We now use the compactness theorem from model theory. 
Let $\left\{T_{i}\right\}_{i\in\mathbb{N}}$ be an enumeration of all the proper special subvarieties of $\mathrm{Y}(1)^{n}$. 
If the conclusion of the corollary were not true, then the following set of formulas (in the variables $\mathbf{p}$ and $x$)
\begin{equation}
    \label{eq:type}
    \left\{ \mathbf{p}\in P\wedge x\in W_{\mathbf{p}}\wedge x\notin\bigcup_{i=0}^{m}T_{i}\right\}_{m\in\mathbb{N}}
\end{equation}
would be finitely satisfiable and hence it would form a type in the language of rings with some extra constant symbols. 
As the family $(W_\mathbf{p})_{\mathbf{p}\in P}$ only requires finitely many parameters to be defined and every $T_i$ in definable over $\overline{\mathbb{Q}}$, then (\ref{eq:type}) is a type in the language of rings with only countably many added constant symbols. 
As $\mathbb{C}$ is $\aleph_{0}$-saturated, this type must be realised over $\mathbb{C}$, but that would mean that there is $\mathbf{p}^{\star}\in P$ such that $W_{\mathbf{p}^{\star}}$ is not contained in the union of all proper special subvarieties, which contradicts (\ref{eq:finitespecials}). 
\end{proof}

We will need the following ``two-sorted'' version of this result.

\begin{cor}
\label{cor:horizontalUMZP}
Let $(U_{\mathbf{q}})_{\mathbf{q}\in Q}$ be a parametric family of subvarieties of $\mathbb{C}^{n}\times\mathrm{Y}(1)^{n}$. 
Then MZP implies that there is a finite collection $\mathscr{S}$ of proper special subvarieties of $\mathrm{Y}(1)^{n}$ such that for all $\mathbf{q}\in Q$ and all special subvarieties $T$, if $X$ is an atypical component of $V_\mathbf{q}\cap (\mathbb{C}^{n}\times T)$, then there is $T_0\in\mathscr{S}$ such that $X\subseteq \mathbb{C}^{n}\times T_0$. 
\end{cor}
\begin{proof}
We follow the proof of \cite[Theorem 11.4]{bays-kirby}. 
Let $U\subseteq Q\times\mathbb{C}^{n}\times\mathrm{Y}(1)^{n}$ be the definable set such that for each $\mathbf{q}\in Q$, $U_{\mathbf{q}} = \left\{(\mathbf{x},\mathbf{y})\in\mathbb{C}^{n}\times\mathrm{Y}(1)^{n} : (\mathbf{q},\mathbf{x},\mathbf{y})\in U\right\}$. 
Given $k\in\left\{0,\ldots,\dim U\right\}$ define
\[U^{(k)}:=\left\{(\mathbf{q},\mathbf{x},\mathbf{y})\in U : \dim(U_{\mathbf{q}}\cap\pi_{\mathrm{Y}}^{-1}(\mathbf{y}))=k\right\}.\]
By definability of dimensions, the $U^{(k)}$ are all constructible subsets of $U$. 
Furthermore, define
\[V^{(k)}:=U^{(k)}\cup U^{(k+1)}\cup\cdots\cup U^{(\dim U)}.\]
Given $\mathbf{q}\in Q$ let $V_{\mathbf{q}}^{(k)}:=\left\{(\mathbf{x},\mathbf{y})\in\mathbb{C}^{n}\times\mathrm{Y}(1)^{n} : (\mathbf{q},\mathbf{x},\mathbf{y})\in V^{(k)}\right\}$. 

For each $k\in\left\{0,\ldots,\dim U\right\}$ let $\mathscr{S}_{k}$ be the finite collection of special subvarieties obtained by applying Corollary \ref{cor:UMZP} to the family $\left(\pi_{\mathrm{Y}}\left(V_{\mathbf{q}}^{(k)}\right)\right)_{\mathbf{q}\in Q}$. 
Set 
\[\mathscr{S} := \bigcup_{k=0}^{\dim U}\mathscr{S}_{k}.\] 

Let $T\subset\mathrm{Y(1)}^{n}$ be a special subvariety, and suppose that $X$ is an atypical component of $U_{\mathbf{q}}\cap (\mathbb{C}^{n}\times T)$. 
Let $k_{0} = \dim X - \dim\pi_{\mathrm{Y}}(X)$, by the fibre-dimension theorem there is a Zariski open subset $X'\subset X$ such that for every $(\mathbf{x},\mathbf{y})\in X'$ we have that $\dim X'\cap\pi_{\mathrm{Y}}^{-1}(\mathbf{y}) = k_{0}$. 
In particular $\dim X' = \dim X$. 
Also $V_{\mathbf{q}}^{(k_{0})}\subseteq U_{\mathbf{q}}$, so
\[\dim X > \dim U_{\mathbf{q}} + \dim T - n \geq \dim V_{\mathbf{q}}^{(k_{0})} + \dim T - n.\]
Let $\mathbf{y}\in\pi_{\mathrm{Y}}(X')$, then $\dim X'\cap\pi_{\mathrm{Y}}^{-1}(\mathbf{y}) = k_{0}$. 
By construction $\dim V_{\mathbf{q}}^{(k_{0})}\cap\pi_{\mathrm{Y}}^{-1}(\mathbf{y})\geq k_{0}$, so by the fibre-dimension theorem we get
\[\dim\pi_{\mathrm{Y}}(X) = \dim \pi_{\mathrm{Y}}(X') > \dim\pi_{\mathrm{Y}}\left(V_{\mathbf{q}}^{(k_{0})}\right) + \dim T - n.\]
Then there is a special subvariety $T_{0}\subset\mathrm{Y}(1)^{n}$ with $\Delta(T_{0})\leq N$ such that $\pi_{\mathrm{Y}}(X')\subseteq T_{0}$ and
\[\dim\pi_{\mathrm{Y}}(X')\leq\dim\pi_{\mathrm{Y}}\left(V_{\mathbf{q}}^{(k_{0})}\cap T_{0}\right) + \dim T\cap T_{0} -\dim T_{0}.\] 
This shows that $X'\subset \mathbb{C}^{n}\times T_{0}$, and since  $\mathbb{C}^{n}\times T_{0}$ is Zariski closed, then $X\subset \mathbb{C}^{n}\times T_{0}$. 
By the definition of $V_{\mathbf{q}}^{(k_{0})}$, the dimension of the fibres of the restriction of $\pi_{\mathrm{Y}}$ to $V_{\mathbf{q}}^{(k_{0})}\cap(\mathbb{C}^{n}\times T_{0})$ is at least $k_{0}$. 
By fibre-dimension theorem we get:
\[\dim X - k_{0}\leq \dim V_{\mathbf{q}}^{(k_{0})}\cap(\mathbb{C}^{n}\times T_{0}) - k_{0} + \dim T\cap T_{0} - \dim T_{0}.\]
Since $V_{\mathbf{q}}^{(k_{0})}\subseteq U_{\mathbf{q}}$, this completes the proof.
\end{proof}

\subsection{Weak MZP}
\label{subsec:weakmzp}
Although MZP is open, Pila and Tsimerman (\cite[\textsection 7]{pila-tsimerman}) showed that, as a consequence of the Ax--Schanuel theorem for $j$, one can obtain a weak form of MZP which states that the atypical components of an algebraic subvariety of $\mathrm{Y}(1)^n$, are contained in finitely many parametric families of proper weakly special subvarieties. 
For the proofs of our main results, we will need the following version of the weak form of MZP which allows for parametric families of algebraic varieties and is ``two-sorted'' like Corollary \ref{cor:horizontalUMZP}. 
\begin{prop}
\label{prop:horizontalboundedcomplexity}
Given a parametric family $(U_{\mathbf{q}})_{\mathbf{q}\in Q}$ of constructible subsets of $\mathbb{C}^{n}\times\mathrm{Y}(1)^{n}$, there is a positive integer $N$ such that for every $q\in Q$, for every weakly special subvariety $T\subset\mathrm{Y}(1)^{n}$ and for every atypical component $X$ of $U_{\mathbf{q}}\cap (\mathbb{C}^{n}\times T)$, there is a proper weakly special subvariety $T_{0}\subset\mathrm{Y}(1)^{n}$ with $\Delta_b(T_{0})\leq N$ such that $X\subseteq \mathbb{C}^{n}\times T_{0}$ and
$$\dim X\leq \dim U_{\mathbf{q}}\cap(\mathbb{C}^{n}\times T_{0}) + \dim T\cap T_{0} - \dim T_{0}.$$ 
\end{prop}
In this section we will present a few technical results centred around the weak form of MZP, which build up to prove Proposition \ref{prop:horizontalboundedcomplexity}. 
We begin with a result of Aslanyan showing the following uniform version of weak MZP.\footnote{That Ax--Schanuel implies a uniform version of a weak form of the Zilber--Pink conjecture holds in very general contexts, see 
\cite[Propositon 2.20]{eterovic-scanlon} and \cite{pila-scanlon}.}

\begin{thm}[Weak MZP, {{\cite[Theorem 5.2]{vahagn}}}]
\label{thm:weakzp}
Let $S$ be a special subvariety of $\mathrm{Y}(1)^{n}$. 
Given a parametric family $(U_{\mathbf{q}})_{\mathbf{q}\in Q}$ of constructible subsets of $S$, there is a finite collection $\mathscr{S}$ of proper special subvarieties of $S$ such that for every $\mathbf{q}$ in $Q$ and for every strongly atypical component $X$ of $U_{\mathbf{q}}$ in $S$, there is $T\in\mathscr{S}$ such that $X\subseteq T$.
\end{thm}

\begin{cor}
\label{cor:weakmzp}
Let $S$ be a special subvariety of $\mathrm{Y}(1)^{n}$. 
Given a parametric family $(U_{\mathbf{q}})_{\mathbf{q}\in Q}$ of constructible subsets of $S$, there is a finite collection $\mathscr{S}$ of proper special subvarieties of $S$ such that for every $\mathbf{q}$ in $Q$ and for every strongly atypical component $X$ of $U_{\mathbf{q}}$, there is $T_{0}\in\mathscr{S}$ satisfying the following conditions:
\begin{enumerate}[(a)]
    \item $X\subseteq T_{0}$,
    \item $X$ is a typical component of $U\cap\mathrm{spcl}(X) = (U_{\mathbf{q}}\cap T_{0}) \cap (\mathrm{spcl}(X)\cap T_{0})$ in $T_{0}$, that is:
    \begin{equation*}
      \dim X \leq \dim U_{\mathbf{q}}\cap T_{0} + \dim \mathrm{spcl}(X) - \dim T_{0},  
    \end{equation*}
    and 
    \item the intersection $U_{\mathbf{q}}\cap T_{0}$ is atypical in $S$:
    \begin{equation*}
        \dim U_{\mathbf{q}}\cap T_{0} > \dim U_{q} + \dim T_{0} - \dim S.
    \end{equation*}
\end{enumerate}
\end{cor}
\begin{proof}
We will first show that there is a family $\mathscr{S}$ satisfying conditions (a) and (b). 
We proceed by induction on the dimension of $S$. 
When $\dim S=0$, then $S$ is a just a point and there is nothing to prove as $U_{\mathbf{q}}=S$ for all $\mathbf{q}\in Q$, so no $U_{\mathbf{q}}$ contains an atypical component. 
Now assume that $\dim S>0$. Let $\mathscr{S}_{1}$ be the finite collection of proper special subvarieties of $S$ obtained by applying Theorem \ref{thm:weakzp} to $(U_{\mathbf{q}})_{\mathbf{q}\in Q}$. 

Suppose that $X$ is a strongly atypical component of $U_{\mathbf{q}}$, then
\[\dim X > \dim U_{\mathbf{q}} + \dim \mathrm{spcl}(X) - \dim S.\]
Choose $T_{1}\in\mathscr{S}_{1}$ such that $X\subseteq T_{1}$. 
If
\[\dim X > \dim U_{\mathbf{q}}\cap T_{1} + \dim \mathrm{spcl}(X)\cap T_{1} - \dim T_{1},\]
then $X$ is a strongly atypical component of $(U_{\mathbf{q}}\cap T_{1})\cap(T\cap T_{1})$ in $T_{1}$. 
Since $\dim T_{1} < \dim S$ we can apply the induction hypothesis on $T_{1}$ to the family $\left(U_{\mathbf{q}}\cap T_{1}\right)_{\mathbf{q}\in Q}$ to obtain a finite collection $\mathscr{S}_{2}$ of proper special subvarieties of $T_{1}$ (which in turn are special subvarieties of $S$) such that there is $T_{0}\in\mathscr{S}_{2}$ satisfying that $X\subseteq T_{0}$ and the intersection $(U_{\mathbf{q}}\cap T_{0}) \cap (\mathrm{spcl}(X)\cap T_{0})$ is typical in $T_{0}$. 

Therefore, the collection $\mathscr{S}$ obtained by taking the union of $\mathscr{S}_{1}$ with the finite collections obtained by the induction hypothesis applied to $\left(U_{\mathbf{q}}\cap T_{0}\right)_{\mathbf{q}\in Q}$ for every $T_{0}\in\mathscr{S}$ satisfies conditions (a) and (b).

We now check that $\mathscr{S}$ satisfies (c). 
Suppose that $X$ is a strongly atypical component of $U_{\mathbf{q}}$, and let $T_{0}\in\mathscr{S}$ be an element satisfying (a) and (b). 
Observe that $\mathrm{spcl}(X)\cap T_{0} = \mathrm{spcl}(X)$. 
From (b) we get that 
    \begin{equation*}
        \dim X - \dim \mathrm{spcl}(X)\leq  \dim U_{\mathbf{q}}\cap T_{0}  - \dim T_{0}.
    \end{equation*}
    Combining this with the fact that $X$ is an atypical component of $U_{\mathbf{q}}$ gives
    \begin{equation*}
        \dim U_{\mathbf{q}} - \dim S < \dim X - \dim \mathrm{spcl}(X) \leq \dim U_{\mathbf{q}}\cap T_{0}  - \dim T_{0},
    \end{equation*}
    from which we get
    \[\dim U_{\mathbf{q}}\cap T_{0} > \dim U_{\mathbf{q}} + \dim T_{0} - \dim S,\]
    thus confirming (c).
\end{proof}

\begin{cor}
\label{cor:basicspecials}
Let $(S_{\mathbf{q}})_{\mathbf{q}\in Q}$ be a parametric family of proper weakly special subvarieties of a special subvariety $S$ of $\mathrm{Y}(1)^{n}$. 
Then there are only finitely many members of this family which are basic special subvarieties. 
\end{cor}
\begin{proof}
Let $\mathscr{S}$ be the finite family of proper special subvarieties of $S$ obtained by applying Corollary \ref{cor:weakmzp} to the family $(S_{\mathbf{q}})_{\mathbf{q}\in Q}$. 
Suppose $\mathbf{q}\in Q$ is such that $S_{\mathbf{q}}$ is a proper basic special subvariety of $S$. 
Then 
\[\dim S_{\mathbf{q}}\cap S_{\mathbf{q}} = \dim S_{\mathbf{q}} > \dim S_{\mathbf{q}} + \dim S_{\mathbf{q}} - \dim S,\]
which shows that $S_{\mathbf{q}}$ is a strongly atypical component of itself in $S$. 
There is $T_{0}\in\mathscr{S}$ such that $S_{q}\subseteq T_{0}$ and
\[\dim S_{\mathbf{q}}\leq \dim S_{\mathbf{q}}\cap T_{0} + \dim S_{\mathbf{q}} - \dim T_{0}.\]
From this we obtain that $0 \leq \dim S_{\mathbf{q}}\cap T_{0} - \dim T_{0}$, and so we obtain that $T_{0} = S_{\mathbf{q}}$. Since $\mathscr{S}$ is a finite set, the result is proven.
\end{proof}

We can now get the following result with the same proof used in Corollary \ref{cor:horizontalUMZP}.

\begin{thm}[Two-sorted weak MZP]
\label{thm:horizontalweakzp}
Given a parametric family $(U_{\mathbf{q}})_{\mathbf{q}\in Q}$ of algebraic subvarieties of $\mathbb{C}^{n}\times\mathrm{Y}(1)^n$, there is a finite collection $\mathscr{S}$ of proper special subvarieties of $\mathrm{Y}(1) ^{n}$ such that for every $\mathbf{q}\in Q$, every proper special subvariety $T$ of $\mathrm{Y}(1)^n$ and every atypical irreducible component $X$ of $U_{\mathbf{q}}\cap (\mathbb{C}^n \times T)$ satisfying that $\pi_{\mathrm{Y}}(X)$ has no constant coordinates, there is $T_0$ in $\mathscr{S}$ such that $X\subseteq\mathbb{C}^{n}\times T_0$. 
\end{thm}

Although Theorem \ref{thm:weakzp} and Corollaries \ref{cor:weakmzp} and \ref{cor:basicspecials} only speak about strongly atypical intersections, one can get the following result about general atypical intersections.

\begin{prop}[{{\cite[Proposition 3.4]{vahagn3}}}]
\label{prop:boundedcomplexity}
Let $S$ be a special subvariety of $\mathrm{Y}(1)^{n}$. 
Given a parametric family $(U_{\mathbf{q}})_{\mathbf{q}\in Q}$ of constructible subsets of $S$, there is a positive integer $N$ such that for every $\mathbf{q}\in Q$ and for every weakly special subvariety $T$ of $S$, if $X$ is an atypical component of $U_{\mathbf{q}}\cap T$ in $S$, then there is a proper weakly special subvariety $T_{0}$ of $S$ satisfying the following conditions: 
\begin{enumerate}[(a)]
    \item $\Delta_b(T_{0})\leq N$,
    \item $X\subseteq T_{0}$,
    \item $X$ is a typical component of $(U_{\mathbf{q}}\cap T_{0})\cap (T\cap T_{0})$ in $T_{0}$:
    \begin{equation*}
        \dim X \leq \dim U_{\mathbf{q}}\cap T_{0} + \dim \mathrm{wspcl}(X) - \dim T_{0},
    \end{equation*}
    and 
    \item the intersection $U_{\mathbf{q}}\cap T_{0}$ is atypical in $S$:
    \begin{equation*}
        \dim U_{\mathbf{q}}\cap T_{0} > \dim U_{\mathbf{q}} + \dim T_{0} - \dim S.
    \end{equation*}
\end{enumerate}
\end{prop}
\begin{proof}
The proof requires a few more calculations than the proof given in \cite{vahagn3}. 
We proceed by induction on $n$, the case $n=0$ being trivial. 
Let $\mathscr{S}$ be the finite family of special subvarieties given by Corollary \ref{cor:weakmzp} applied to $(U_{\mathbf{q}})_{\mathbf{q}\in Q}$. 
Let $N_{\mathscr{S}}$ be the maximal complexity of the elements of $\mathscr{S}$.

Suppose that for some $\mathbf{q}\in Q$ and some weakly special subvariety $T\subset S$, the intersection $U_{\mathbf{q}}\cap T$ contains an atypical component $X$. 
We may assume that $T = \mathrm{wspcl}(X)$. 

If $X$ is strongly atypical (i.e.~has no constant coordinates), then $\mathrm{spcl}(X) = \mathrm{wspcl}(X)$ and so we know that there is $T_{0}\in\mathscr{S}$ satisfying (b), (c) and (d). 
By construction, $\Delta_b(T_{0})\leq N_{\mathscr{S}}$, which verifies (a). 

Assume now that $X$ is not strongly atypical. Let $\mathbf{i}=(i_{1},\ldots,i_{m})$ be the tuple of all of the coordinates which are constant on $X$, let $\mathbf{c}\in\mathrm{Y}(1)^{m}$ be such that $\mathrm{pr}_{\mathbf{i}}(X) = \left\{\mathbf{c}\right\}$, and define
\[S_{\mathbf{c}}:= S\cap \mathrm{pr}_{\mathbf{i}}^{-1}(\mathbf{c}),\] 
that is, $S_{\mathbf{c}}$ is the fibre in $S$ over $\mathbf{c}$. 
As explained in \cite[Proposition 3.4]{vahagn3}, $S_{\mathbf{c}}$ is an irreducible variety, so it is weakly special. 
Observe that as $T$ is the weakly special closure of $X$, then $T\subset S_{\mathbf{c}}$. 
When $\mathbf{i}$ is not all of $\mathbf{n}$, let $\mathbf{k} = \mathbf{n}\setminus\mathbf{i}$. 
We consider now the possible cases.

\textbf{Suppose that} 
\[\dim U_{\mathbf{q}}\cap S_{\mathbf{c}} > \dim U_{\mathbf{q}} + \dim S_{\mathbf{c}} -\dim S.\] 
If $\dim X\leq \dim U_{\mathbf{q}}\cap S_{\mathbf{c}} + \dim T - \dim S_{\mathbf{c}}$,  then we can let $T_{0}$ be $S_{\mathbf{c}}$, since $\Delta_b(S_{\mathbf{c}})=0$. 
In particular, this would be the case if $X$ happens to be a single point, in which case $\mathbf{i} = \mathbf{n}$. 

If instead $\dim X> \dim U_{\mathbf{q}}\cap S_{\mathbf{c}} + \dim T - \dim S_{\mathbf{c}}$, then $\mathbf{i} \neq \mathbf{n}$. 
Consider the projection $\mathrm{pr}_{\mathbf{k}}:\mathrm{Y}(1)^{n}\to\mathrm{Y}(1)^{n-m}$. 
Since we have that $X\subseteq T\subseteq S_{\mathbf{c}}$, then $\dim X = \dim\mathrm{pr}_{\mathbf{k}}(X)$, $\dim T = \dim\mathrm{pr}_{\mathbf{k}}(T)$ and $\dim U_{\mathbf{q}}\cap S_{\mathbf{c}} = \dim\mathrm{pr}_{\mathbf{k}}(U_{\mathbf{q}}\cap S_{\mathbf{c}})$. 
Also, since $S_{\mathbf{c}}\subseteq S$, then $\dim S_{\mathbf{c}} = \dim \mathrm{pr}_{\mathbf{k}}(S_{\mathbf{c}})\leq\dim\mathrm{pr}_{\mathbf{k}}(S)$. 
Let $X_{1}$ be the irreducible component of $\mathrm{pr}_{\mathbf{k}}(U_{\mathbf{q}}\cap S_{\mathbf{c}})\cap\mathrm{pr}_{\mathbf{k}}(T)$ containing $\mathrm{pr}_{\mathbf{k}}(X)$, and observe that $\mathrm{pr}_{\mathbf{k}}(T)$ is a weakly special subvariety of $\mathrm{pr}_{\mathbf{k}}(S)$. 
Thus
\begin{align*}
    \dim X_{1}\geq \dim\mathrm{pr}_{\mathbf{k}}(X) &= \dim X\\
    &> \dim U_{\mathbf{q}}\cap S_{\mathbf{c}} + \dim T - \dim S_{\mathbf{c}}\\
    & = \dim\mathrm{pr}_{\mathbf{k}}(U_{\mathbf{q}}\cap S_{\mathbf{c}}) + \dim\mathrm{pr}_{\mathbf{k}}(T) -\dim\mathrm{pr}_{\mathbf{k}}(S_{\mathbf{c}}),
\end{align*}
showing that $X_{1}$ is atypical in $\mathrm{pr}_{\mathbf{k}}(U_{\mathbf{q}}\cap S_{\mathbf{c}})\cap\mathrm{pr}_{\mathbf{k}}(T)$. 
We also remark that by construction $X_{1}$ is in fact strongly atypical, and so $\mathrm{pr}_{\mathbf{k}}(S_{\mathbf{c}})$ must in fact be a basic special subvariety of $\mathrm{pr}_{\mathbf{k}}(S)$. 
The collection $\left(\mathrm{pr}_{\mathbf{k}}(S_{\mathbf{c}})\right)_{\mathbf{c}\in\mathrm{Y}(1)^{m}}$ is a parametric family of weakly special subvarieties of $\mathrm{pr}_{\mathbf{k}}(S)$. By definability of dimensions, we may consider the parametric subfamily of those elements which are properly contained in $\mathrm{pr}_{\mathbf{k}}(S)$ (if there are any), and so by Corollary \ref{cor:basicspecials} among the whole family $\left(\mathrm{pr}_{\mathbf{k}}(S_{\mathbf{c}})\right)_{\mathbf{c}\in\mathrm{Y}(1)^{m}}$ we will find only finitely many basic special subvarieties. 
Specifically we can find a finite set $C\subset\mathrm{Y}(1)^{m}$ such that every basic special subvariety found in $\left(\mathrm{pr}_{\mathbf{k}}(S_{\mathbf{c}})\right)_{\mathbf{c}\in\mathrm{Y}(1)^{m}}$ can be realised as $S_{\mathbf{c}}$ for some $\mathbf{c}\in C$. 

We can apply the induction hypothesis to the family $\left(\mathrm{pr}_{\mathbf{k}}(U_{\mathbf{q}}\cap S_{\mathbf{c}})\right)_{\mathbf{q}\in Q}$ of constructible subsets of $\mathrm{pr}_{\mathbf{k}}(S_{\mathbf{c}})$. 
In this way we find a natural number $N_{\mathbf{k},\mathbf{c}}$ and a proper weakly special subvariety $T_{1}$ of $\mathrm{pr}_{\mathbf{k}}(S_{\mathbf{c}})$ such that
\begin{enumerate}[(i)]
    \item $\Delta_b(T_{1})\leq N_{\mathbf{k}}$,
    \item $X_{1}\subseteq T_{1}$,
    \item $\dim X_{1} \leq \dim \mathrm{pr}_{\mathbf{k}}(U_{\mathbf{q}}\cap S_{\mathbf{c}})\cap T_{1} + \dim \mathrm{wspcl}(X_{1}) - \dim T_{1}$, and 
    \item $\dim \mathrm{pr}_{\mathbf{k}}(U_{\mathbf{q}}\cap S_{\mathbf{c}})\cap T_{1} > \dim \mathrm{pr}_{\mathbf{k}}(U_{\mathbf{q}}\cap S_{\mathbf{c}}) + \dim T_{1} - \dim \mathrm{pr}_{\mathbf{k}}(S_{\mathbf{c}})$.
\end{enumerate}
Let $T_{0}:=\mathrm{pr}_{\mathbf{k}}^{-1}(T_{1})\cap S_{\mathbf{c}}$. 
Then $T_{0}$ is a weakly special subvariety of $S$ with the property that $\Delta_b(T_{0}) = \Delta_b(T_{1})\leq N_{\mathbf{k}}$, $\dim T_{0} = \dim T_{1}$ and $X\subseteq T_{0}$. 
Observe that since $X_{1}\subseteq\mathrm{pr}_{\mathbf{k}}(T)$ and $\mathrm{pr}_{\mathbf{k}}(T)$ is weakly special, then $\mathrm{wspcl}(X_{1})\subseteq\mathrm{pr}_{\mathbf{k}}(T)$. 
Also, one can readily check that $\mathrm{pr}_{\mathbf{k}}(U_{\mathbf{q}}\cap T_{0}) = \mathrm{pr}_{\mathbf{k}}(U_{\mathbf{q}}\cap S_{\mathbf{c}})\cap T_{1}$. So using (iii) we get:
\begin{align*}
    \dim X = \dim\mathrm{pr}_{\mathbf{k}}(X) &\leq \dim X_{1}\\
    &\leq \dim \mathrm{pr}_{\mathbf{k}}(U_{\mathbf{q}}\cap S_{\mathbf{c}})\cap T_{1} + \dim \mathrm{wspcl}(X_{1}) - \dim T_{1}\\
    &\leq\dim\mathrm{pr}_{\mathbf{k}}(U_{\mathbf{q}}\cap T_{0}) + \dim T - \dim T_{0}\\
    &= \dim U_{\mathbf{q}}\cap T_{0} + \dim T - \dim T_{0},
\end{align*}
and using (iv) we get
\begin{align*}
    \dim U_{\mathbf{q}}\cap T_{0} &=  \dim\mathrm{pr}_{\mathbf{k}}(U_{\mathbf{q}}\cap T_{0})\\
    &= \dim \mathrm{pr}_{\mathbf{k}}(U_{\mathbf{q}}\cap S_{\mathbf{c}})\cap T_{1}\\
    &> \dim \mathrm{pr}_{\mathbf{k}}(U_{\mathbf{q}}\cap S_{\mathbf{c}}) + \dim T_{1} - \dim \mathrm{pr}_{\mathbf{k}}(S_{\mathbf{c}})\\
    &\geq \dim U_{\mathbf{q}}\cap S_{\mathbf{c}} + \dim T_{0} - \dim \mathrm{pr}_{\mathbf{k}}(S_{\mathbf{c}})\\
    &> \dim U_{\mathbf{q}} + \dim T_{0} - \dim S.
\end{align*}
Thus, we can set $N_{\mathbf{k}} = \max\left\{N_{\mathbf{k},\mathbf{c}} \mid \mathbf{c}\in C\right\}$, and in this case $T_{0}$ satisfies the conditions of the proposition.

\textbf{Suppose now that}
\[\dim (U_{\mathbf{q}}\cap S_{\mathbf{c}}) = \dim U_{\mathbf{q}} + \dim S_{\mathbf{c}} - \dim S.\]
As $T\subseteq S_{\mathbf{c}}$, $\dim T = \dim\mathrm{pr}_{\mathbf{k}}(T)$ and $\dim (U_{\mathbf{q}}\cap S_{\mathbf{c}}) = \dim \mathrm{pr}_{\mathbf{k}}(U_{\mathbf{q}}\cap S_{\mathbf{c}})$. Then 
\begin{align*}
\dim \mathrm{pr}_{\mathbf{k}}(X) = \dim X &> \dim U_{\mathbf{q}} + \dim T - \dim S\\
&=\dim \mathrm{pr}_{\mathbf{k}}(U_{\mathbf{q}}\cap S_{\mathbf{c}}) + \dim\mathrm{pr}_{\mathbf{k}}(T) - \dim\mathrm{pr}_{\mathbf{k}}(S_{\mathbf{c}}).
\end{align*}
This seems to puts us right back in the previous case, and the argument can still be carried through with a few subtleties, so we will just focus on those. 
Let $X_{1}$ be the irreducible component of $\mathrm{pr}_{\mathbf{k}}(U_{\mathbf{q}}\cap S_{\mathbf{c}})\cap\mathrm{pr}_{\mathbf{k}}(T)$ containing $\mathrm{pr}_{\mathbf{k}}(X)$, and observe that $X_{1}$ is strongly atypical. 
As before, we get a finite set $C\subset\mathrm{Y}(1)^{m}$ such that every basic special subvariety found in $\left(\mathrm{pr}_{\mathbf{k}}(S_{\mathbf{c}})\right)_{\mathbf{c}\in\mathrm{Y}(1)^{m}}$ can be realised as $S_{\mathbf{c}}$ for some $\mathbf{c}\in C$. 
Let $\mathscr{S}_{\mathbf{k},\mathbf{c}}$ be the finite family of special subvarieties of $\mathrm{pr}_{\mathbf{k}}(S)$ given by Corollary \ref{cor:weakmzp} applied to  $\left(\mathrm{pr}_{\mathbf{k}}(U_{\mathbf{q}}\cap S_{\mathbf{c}})\right)_{\mathbf{q}\in Q}$ as a family in $ S_{\mathbf{c}}$. 
Then there is $T_{1}\in\mathscr{S}_{\mathbf{k},\mathbf{c}}$ such that
\begin{enumerate}[(i)]
    \item $X_{1}\subseteq T_{1}$
    \item $\dim X_{1}\leq \dim \mathrm{pr}_{\mathbf{k}}(U_{\mathbf{q}}\cap S_{\mathbf{c}})\cap T_{1} + \dim \mathrm{wspcl}(X_{1}) - \dim T_{1}$, and
    \item $\dim\mathrm{pr}_{\mathbf{k}}(U_{\mathbf{q}}\cap S_{\mathbf{c}})\cap T_{1} >\dim\mathrm{pr}_{\mathbf{k}}(U_{\mathbf{q}}\cap S_{\mathbf{c}}) + \dim T_{1} -\dim\mathrm{pr}_{\mathbf{k}}( S_{\mathbf{c}})$.
\end{enumerate}
Let $N_{\mathscr{S}_{\mathbf{k}}}$ be the maximal basic complexity of elements of $\mathscr{S}_{\mathbf{k},\mathbf{c}}$ for all $\mathbf{c}\in C$.

Define $T_{0}:=\mathrm{pr}_{\mathbf{k}}^{-1}(T_{1})\cap S_{\mathbf{c}}$, which is a proper weakly special subvariety of $S$ satisfying $\Delta_b\left(T_{0}\right) = \Delta_b(T_{1})\leq N_{\mathscr{S}_{\mathbf{k}}}$, $X\subseteq T_{0}$, 
\begin{align*}
    \dim X = \dim\mathrm{pr}_{\mathbf{k}}(X) &\leq \dim X_{1}\\
    & \leq \dim \mathrm{pr}_{\mathbf{k}}(U_{\mathbf{q}}\cap T_{0}) + \dim \mathrm{pr}_{\mathbf{k}}(T)  - \dim T_{0}\\
    &=\dim U_{\mathbf{q}}\cap T_{0} + \dim T - \dim T_{0},
\end{align*}
and 
\begin{align*}
\dim U_{\mathbf{q}}\cap T_{0} &= \dim U_{\mathbf{q}}\cap S_{\mathbf{c}}\cap T_{0}\\
&= \dim\mathrm{pr}_{\mathbf{k}}(U_{\mathbf{q}}\cap S_{\mathbf{c}})\cap T_{1}\\
&> \dim \mathrm{pr}_{\mathbf{k}}(U_{\mathbf{q}}\cap S_{\mathbf{c}}) + \dim T_{1} - \dim\mathrm{pr}_{\mathbf{k}}( S_{\mathbf{c}})\\
&= \dim U_{\mathbf{q}}\cap S_{\mathbf{c}} + \dim T_{0} - \dim\mathrm{pr}_{\mathbf{k}}( S_{\mathbf{c}})\\
&\geq \dim U_{\mathbf{q}} + \dim T_{0} - \dim S.
\end{align*}

Since there are finitely many tuples $\mathbf{k}$ to consider, by taking the maximum of the $N_{\mathscr{S}_{\mathbf{k}}}$, the $N_{\mathbf{k}}$ and $N_{\mathscr{S}}$, we get the desired $N$. 
This completes the proof.
\end{proof}

Now one can prove Proposition \ref{prop:horizontalboundedcomplexity} using the same method of proof as in Corollary \ref{cor:horizontalUMZP}.
We leave the details to the reader.

\section{Strong Existential Closedness}
\label{sec:sec}

Let $V\subseteq\mathbb{C}^{m}$ be an irreducible constructible set and $K\subset\mathbb{C}$ a subfield over which $V$ is defined. We say that a point $\mathbf{v}\in V$ is \emph{generic in $V$ over $K$} if 
\[\mathrm{tr.deg.}_{K}K(\mathbf{v}) = \dim V.\]

\begin{lem}
\label{lem:genericpoint}
Let $V\subseteq\mathbb{C}^{m}$ be an irreducible subvariety, and let $K\subset\mathbb{C}$ be a subfield over which $V$ is defined. 
Suppose that no coordinate is constant on $V$. 
If $\mathbf{v}\in V$ is generic in $V$ over $K$, then every coordinate of $\mathbf{v}$ is transcendental over $K$.
\end{lem}
\begin{proof}
Say that $V$ defined in the ring of polynomials $\mathbb{C}[X_{1},\ldots,X_{m}]$. 
Since no coordinate is constant on $V$, then for every $c\in\mathbb{C}$ and for every $i\in\left\{1,\ldots,m\right\}$ we have that $V\cap\left\{X_{i}=c\right\}$ is a proper subvariety of $V$. 
So, given $\mathbf{v}\in V$, if $c\in\overline{K}$ and $v_{i}=c$, for some $i\in\left\{1,\ldots,m\right\}$, then we have that
\[\mathrm{tr.deg.}_{K}K(\mathbf{v})\leq\dim V\cap\left\{X_{i}=c\right\} < \dim V,\]
which prevents $\mathbf{v}$ from being generic in $V$ over $K$.
\end{proof}

\begin{defn}
We say that a variety $V\subseteq\mathbb{C}^{n}\times\mathrm{Y}(1)^{n}$ satisfies the \emph{strong existential closedness property} (SEC for short) if for every finitely generated field $K\subset\mathbb{C}$ over which $V$ can be defined, there exists $(\mathbf{z},j(\mathbf{z}))\in V$ such that $(\mathbf{z},j(\mathbf{z}))$ is generic in $V$ over $K$. 
\end{defn}

\begin{conj}
\label{conj:sec}
For every positive integer $n$, every algebraic variety $V$ contained in $\mathbb{C}^{n}\times\mathrm{Y}(1)^{n}$ which is broad and free, if  $V\cap\left(\mathbb{H}^n\times\mathrm{Y}(1)^n\right)$ is Zariski dense in $V$, then $V$ satisfies (SEC).
\end{conj}

The condition in Theorem \ref{th:main1} that $\pi_{\mathbb{C}}(V)$ be Zariski dense in $\mathrm{Y}(1)^{n}$ implies broadness but not freeness, and yet, (EC) still holds on $V$. 
That is, even if $V$ is contained in a subvariety of the form $\mathbb{C}^{n}\times T$, with $T$ a special subvariety of $\mathrm{Y}(1)^{n}$, $V$ will satisfy (EC) as long as $\pi_{\mathbb{C}}(V)$ is Zariski dense in $\mathbb{C}^{n}$. 
However, if $V$ is not free then it may not have points in $V\cap\mathrm{E}_{j}^{n}$ which are generic, as shown in the next example. 

\begin{ex}
Let $D$ denote the diagonal of $\mathrm{Y}(1)^{2}$, and let $V = \mathbb{C}^{2}\times D$ (so $V$ is defined over $\mathbb{Q}$). 
By Theorem \ref{th:main1} $V\cap\mathrm{E}_{j}^{2}$ is Zariski dense in $V$. 
However, every point $(z_{1},z_{2},j(z_{1}),j(z_{2}))\in V$ satisfies that $j(z_{1}) = j(z_{2})$, and so for this point there exists $\gamma\in\mathrm{SL}_{2}(\mathbb{Z})$ such that $\gamma z_{1} = z_{2}$. Therefore
\[\mathrm{tr.deg.}_{\mathbb{Q}}\mathbb{Q}(z_{1},z_{2},j(z_{1}),j(z_{2}))\leq 2<3 = \dim V.\]
\end{ex}

The next proposition shows that in Conjecture \ref{conj:sec} we may assume $\dim V=n$.

\begin{prop}[cf {{\cite[Lemma 4.30]{vahagn2}}}]
\label{prop:secindimn}
Fix a positive integer $n$. 
Assume that every subvariety of $\mathbb{C}^{n}\times\mathrm{Y}(1)^{n}$ which is free, broad and has dimension $n$ satisfies (SEC). 
Then every subvariety $V\subseteq\mathbb{C}^{n}\times\mathrm{Y}(1)^{n}$ which is free and broad satisfies (SEC).
\end{prop}
\begin{proof}
Let $V\subseteq \mathbb{C}^{n}\times\mathrm{Y}(1)^{n}$ be an subvariety which is free and broad. 
Let $k:=\dim V - n$. 
We proceed by induction on $k$. 
The case $k=0$ is given by the assumption, so we assume that $k>0$. 

Let $K$ be a finitely generated field over which $V$ is definable. 
Choose $(\mathbf{a},\mathbf{b})\in V$ to be generic over $K$. 
Let $p_{1},\ldots,p_{n},q_{1},\ldots,q_{n-1}\in \mathbb{C}$ be algebraically independent over $K(\mathbf{a},\mathbf{b})$ and choose $q_{n}\in\mathbb{C}$ such that
\begin{equation}
    \label{eq:hyperplane}
    \sum_{i=1}^{n}p_{i}a_{i} + \sum_{i=1}^{n}q_{i}b_{i} = 1.
\end{equation}
Set $K_{1}:=K(\mathbf{p},\mathbf{q})$ (so $K_{1}$ is also finitely generated) and let $V_{1}$ be the $\overline{K_{1}}$-Zariski closure of $(\mathbf{a},\mathbf{b})$. 
Clearly $V_{1}$ is and $\dim V_{1} < \dim V$. 
We now check the remaining conditions to apply the induction hypothesis on $V_{1}$. 
\begin{enumerate}[(a)]
    \item We first show that $\dim V_{1} = \dim V-1$. 
    Observe that
\begin{align*}
\mathrm{tr.deg.}_{K}K(\mathbf{a},\mathbf{b}) + \mathrm{tr.deg.}_{K(\mathbf{a},\mathbf{b})}K(\mathbf{a},\mathbf{b},\mathbf{p},\mathbf{q}) &= \mathrm{tr.deg.}_{K}K(\mathbf{a},\mathbf{b},\mathbf{p},\mathbf{q})\\
&= \mathrm{tr.deg.}_{K}K_{1} + \mathrm{tr.deg.}_{K_{1}}K_{1}(\mathbf{a},\mathbf{b}),
\end{align*}
and since by construction
\begin{align*}
\mathrm{tr.deg.}_{K}K(\mathbf{a},\mathbf{b}) &= \dim V\\
\mathrm{tr.deg.}_{K(\mathbf{a},\mathbf{b})}K(\mathbf{a},\mathbf{b},\mathbf{p},\mathbf{q}) &= 2n-1\\
\mathrm{tr.deg.}_{K}K_{1} &\geq 2n-1
\end{align*}
we conclude that $\dim V_{1} = \mathrm{tr.deg.}_{K_{1}}K_{1}(\mathbf{a},\mathbf{b}) = \dim V-1$. 
\item We now show that $V_{1}$ is free. 
Since $V$ is free, we know that $(\mathbf{a},\mathbf{b})$ is not contained in any variety of the form $M\times T$, where $M$ is a M\"obius subvariety of $\mathbb{C}^{n}$ defined over $\mathbb{Q}$ and $T$ is a proper special subvariety of $\mathrm{Y}(1)^{n}$. 
If $V_{1}$ had a constant coordinate, then that would imply that some coordinate of $(\mathbf{a},\mathbf{b})$ is in $\overline{K_{1}}$, but since $(\mathbf{a},\mathbf{b})$ already satisfies (\ref{eq:hyperplane}) and is generic in $V_{1}$ over $K_{1}$, then we would have that $\dim V_{1} < \dim V-1$. 
Therefore $V_{1}$ is free. 
\item Next we show that $V_{1}$ is broad. 
Let $\mathbf{i} = (i_{1},\ldots,i_{\ell})$ be such that $1\leq i_{1}<\cdots<i_{\ell}\leq n$ and $\ell<n$. 
Consider the projection $\mathrm{Pr}_{\mathbf{i}}(V_{1})$. 
We have that  
\[\dim \mathrm{Pr}_{\mathbf{i}}(V_{1}) = \mathrm{tr.deg.}_{K_{1}}K_{1}(\mathrm{pr}_{\mathbf{i}}(\mathbf{a}),\mathrm{pr}_{\mathbf{i}}(\mathbf{b})).\]
Proceeding as we did in (a), we have that
\begin{multline*}
\mathrm{tr.deg.}_{K}K(\mathrm{pr}_{\mathbf{i}}(\mathbf{a}),\mathrm{pr}_{\mathbf{i}}(\mathbf{b})) + \mathrm{tr.deg.}_{K(\mathrm{pr}_{\mathbf{i}}(\mathbf{a}),\mathrm{pr}_{\mathbf{i}}(\mathbf{b}))}K(\mathrm{pr}_{\mathbf{i}}(\mathbf{a}),\mathrm{pr}_{\mathbf{i}}(\mathbf{b}),\mathbf{p},\mathbf{q}) =\\ 
 \mathrm{tr.deg.}_{K}K_{1} + \mathrm{tr.deg.}_{K_{1}}K_{1}(\mathrm{pr}_{\mathbf{i}}(\mathbf{a}),\mathrm{pr}_{\mathbf{i}}(\mathbf{b})),
\end{multline*}
and 
\begin{align*}
\mathrm{tr.deg.}_{K}K(\mathrm{pr}_{\mathbf{i}}(\mathbf{a}),\mathrm{pr}_{\mathbf{i}}(\mathbf{b})) &= \dim\mathrm{Pr}_{\mathbf{i}}(V)\\
\mathrm{tr.deg.}_{K(\mathrm{pr}_{\mathbf{i}}(\mathbf{a}),\mathrm{pr}_{\mathbf{i}}(\mathbf{b}))}K(\mathrm{pr}_{\mathbf{i}}(\mathbf{a}),\mathrm{pr}_{\mathbf{i}}(\mathbf{b}),\mathbf{p},\mathbf{q}) &\geq 2n-1\\
\mathrm{tr.deg.}_{K}K_{1} &= 2n.
\end{align*}
 By \cite[Lemma 4.31]{vahagn2} we have that one of the two following cases must hold:
\begin{enumerate}[(i)]
    \item $a_{1},\ldots,a_{n},b_{1},\ldots,b_{n}\in \overline{K(\mathrm{pr}_{\mathbf{i}}(\mathbf{a}),\mathrm{pr}_{\mathbf{i}}(\mathbf{b}))}$. 
    Then $\dim\mathrm{Pr}_{\mathbf{i}}(V) = \dim V$, so 
    \[\dim\mathrm{Pr}_{\mathbf{i}}(V_{1})\geq n-1\geq\ell.\] 
    \item $\mathrm{tr.deg.}_{K_{1}}K_{1}(\mathrm{pr}_{\mathbf{i}}(\mathbf{a}),\mathrm{pr}_{\mathbf{i}}(\mathbf{b})) = \mathrm{tr.deg.}_{K}K(\mathrm{pr}_{\mathbf{i}}(\mathbf{a}),\mathrm{pr}_{\mathbf{i}}(\mathbf{b}))$. 
    Since $V$ is broad, 
    \[\dim\mathrm{Pr}_{\mathbf{i}}(V_{1}) = \dim\mathrm{Pr}_{\mathbf{i}}(V)\geq\ell.\]
\end{enumerate}
\end{enumerate}

We can now apply the induction hypothesis and deduce that there is $(\mathbf{z},j(\mathbf{z}))\in V_{1}\cap \mathrm{E}_{j}^{n}$ such that $(\mathbf{z},j(\mathbf{z}))$ is generic in $V_{1}$ over $K_{1}$. 
In particular $(\mathbf{z},j(\mathbf{z}))$ satisfies (\ref{eq:hyperplane}), but since the coefficients of this equation are transcendental over $K$, then 
\[\dim V\geq \mathrm{tr.deg.}_{K}K(\mathbf{z},j(\mathbf{z})) > \mathrm{tr.deg.}_{K_{1}}K_{1}(\mathbf{z},j(\mathbf{z})) = \dim V_{1},\]
so $(\mathbf{z},j(\mathbf{z}))$ is generic in $V$ over $K$.
\end{proof}

\subsection{\texorpdfstring{$j$}{j}-derivations}

As mentioned in \textsection\ref{subsec:keys}, an obstacle towards proving Theorem \ref{thm:generic} is that MSCD only gives a lower bound for the transcendence degree over $\mathbb{Q}$. 
Since we are looking for generic points over an arbitrary finitely generated field $K$, we would like to have an inequality for the transcendence degree over $K$. 
This is done in \textsection\ref{subsec:transineq}. 
Before that, we need to give a quick review of $j$-derivations (see \cite[\textsection 3.1]{aek-closureoperator} and \cite[\textsection 5]{sebae} for details). 

A function $\partial:\mathbb{C}\rightarrow\mathbb{C}$ is called a \emph{$j$-derivation} if it satisfies the following axioms:
\begin{enumerate}
    \item For all $a,b\in\mathbb{C}$, $\partial(a+b) = \partial(a) + \partial(b)$.
    \item For all $a,b\in\mathbb{C}$, $\partial(ab) = a\partial(b) + b\partial(a)$.
    \item For all $z\in\mathbb{H}$ and all $n\in\mathbb{N}$, $\partial\left(j^{(n)}(z)\right) = j^{(n+1)}(z)\partial(z)$, where $j^{(n)}$ denotes the $n$-th derivative of $j$.
\end{enumerate}
$\mathbb{C}$ has non-trivial $j$-derivations (in fact, it has continuum many $\mathbb{C}$-linearly independent $j$-derivations).  

\begin{defn}
Let $A\subseteq\mathbb{C}$ be any set. 
We define the set $j\mathrm{cl}(A)$ by the property: $x\in j\mathrm{cl}(A)$ if and only if $\partial(x) = 0$ for every $j$-derivation $\partial$ with $A\subseteq\ker\partial$. 
If $A = j\mathrm{cl}(A)$, then we say that $A$ is \emph{$j\mathrm{cl}$-closed}. 
\end{defn}

Every $j\mathrm{cl}$-closed subset of $\mathbb{C}$ is an algebraically closed subfield. 
Furthermore, $j\mathrm{cl}$ has a corresponding well-defined notion of dimension\footnote{In technical terms, $j\mathrm{cl}$ is a \emph{pregeometry}.}, which we denote $\dim^{j}$, defined in the following way. 
For any subsets $A, B\subseteq\mathbb{C}$, $\dim^{j}(A|B) \geq n$ if and only if there exist $a_{1},\ldots,a_{n}\in j\mathrm{cl}(A)$ and $j$-derivations $\partial_{1},\ldots,\partial_{n}$ such that $B\subseteq\ker\partial_{i}$ for $i=1,\ldots,n$ and 
\begin{equation*}
    \partial_{i}\left(a_{k}\right) = \left\{\begin{array}{cl}
         1 & \mbox{ if } i=k\\
         0 & \mbox{ otherwise}
    \end{array}\right.
\end{equation*}
for every $i,k = 1,\ldots,n$. 

We define $C_{j} := j\mathrm{cl}(\emptyset)$; this is a countable algebraically closed subfield of $\mathbb{C}$ which is $j\mathrm{cl}$-closed. 

\begin{remark}
\label{rem:dimj}
By definition, $C_{j}$ is contained in the kernel of every $j$-derivation, so for every finite set $A\subseteq\mathbb{C}$ we have that $\dim^{j}(A) = \dim^{j}(A|C_{j})$. 
\end{remark}

\begin{lem}
\label{lem:c}
Let $C\subseteq\mathbb{C}$ be $j\mathrm{cl}$-closed. 
For every $z\in\mathbb{H}$, the following statements hold:
\begin{enumerate}[(a)]
    \item $z\in C$ implies $j(z),j'(z),j''(z)\in C$.\footnote{Due to the differential equation of $j$ (\ref{eq:j}), it is enough to only consider the derivatives up to $j''$.}
    \item $j(z)\in C$ implies $z\in C$.
\end{enumerate}
\end{lem}
\begin{proof}
Follows from \cite[Proposition 5.8]{sebae} and the fact that $j'(z) = 0$ implies that $z$ is algebraic
\end{proof}

The Ax--Schanuel theorem for $j$ \cite[Theorem 1.3]{pila-tsimerman} has the following consequence:

\begin{prop}[{{\cite[Proposition 6.2]{sebae}}}]
\label{prop:as}
Let $C$ be a $j\mathrm{cl}$-closed subfield of $\mathbb{C}$, then for every $z_{1},\ldots,z_{n}\in\mathbb{H}$ we have that 
 \begin{equation*}
     \mathrm{tr.deg.}_{C}C\left(\mathbf{z},j(\mathbf{z}),j'(\mathbf{z}),j''(\mathbf{z})\right)\geq 3\dim_{G}\left(\mathbf{z}|C\right) + \dim^{j}\left(\mathbf{z}|C\right).
 \end{equation*}
\end{prop}

\subsection{Convenient tuples}
\label{subsec:transineq}

In this section we use the results of \cite{aek-closureoperator} to show that MSCD implies the existence of ``convenient generators'' for any finitely generated field. 
We will use $J$ to denote the triple of function $(j,j',j'')$, so that if $z_{1},\ldots,z_{n}$ are elements of $\mathbb{H}$, then:
\[J(\mathbf{z}) := (j(z_{1}),\ldots,j(z_{n}),j'(z_{1}),\ldots,j'(z_{n}),j''(z_{1}),\ldots,j''(z_{n})).\]

\begin{defn}
We will say that a tuple $\mathbf{t}=(t_{1},\ldots,t_{m})$ of elements of $\mathbb{H}\setminus\Sigma$ is \emph{convenient for $j$} if
\[\mathrm{tr.deg.}_{\mathbb{Q}}\mathbb{Q}(\mathbf{t},J(\mathbf{t})) = 3\dim_{G}(\mathbf{t}) + \dim^{j}(\mathbf{t}).\]
We remark that since the coordinates of $\mathbf{t}$ are not in $\Sigma$, then $\dim_{G}(\mathbf{t}) = \dim_{G}(\mathbf{t}|\Sigma)$.
\end{defn}

The convenience of such a tuple of elements is manifested in the following two lemmas, which show that we can obtain MSCD-type inequalities over fields generated by a convenient tuple. 

\begin{lem}
\label{lem:fgjmscas}
Assume MSCD holds. 
Suppose that $t_{1},\ldots,t_{m}\in\mathbb{H}\setminus\Sigma$ is a convenient tuple for $j$ and set $F = \mathbb{Q}\left(\mathbf{t},J\left(\mathbf{t}\right)\right)$. 
Then for any $z_{1},\ldots,z_{n}\in\mathbb{H}$ we have
\begin{align}
    \label{eq:mscasfgJ}
         \mathrm{tr.deg.}_{F}F\left(\mathbf{z},J(\mathbf{z})\right) 
          &\geq  3\dim_{G}\left(\mathbf{z}|\Sigma,\mathbf{t}\right) + \dim^{j}\left(\mathbf{z}|\mathbf{t}\right), \mbox{ and }\\
    \label{eq:mscasfgj}     \mathrm{tr.deg.}_{F}F\left(\mathbf{z},j(\mathbf{z})\right) 
          &\geq  \dim_{G}\left(\mathbf{z}|\Sigma,\mathbf{t}\right) + \dim^{j}\left(\mathbf{z}|\mathbf{t}\right).
\end{align}
\end{lem}
\begin{proof}
 Using the addition formula, we first get that:
\begin{equation}
\label{eq:1}
    \mathrm{tr.deg.}_{\mathbb{Q}}\mathbb{Q}\left(\mathbf{z},\mathbf{t},J(\mathbf{z}),J\left(\mathbf{t}\right)\right) = \mathrm{tr.deg.}_{\mathbb{Q}}F  + \mathrm{tr.deg.}_{F}F(\mathbf{z},J(\mathbf{z})).
\end{equation}
Similarly we get
\begin{equation}
\label{eq:2}
\dim_{G}(\mathbf{z},\mathbf{t}|\Sigma) = \dim_{G}(\mathbf{t}|\Sigma) + \dim_{G}(\mathbf{z}|\mathbf{t},\Sigma)
\end{equation}
and also
\begin{equation}
\label{eq:fgj3}
    \dim^{j}(\mathbf{z},\mathbf{t}) = \dim^{j}(\mathbf{t}) + \dim^{j}(\mathbf{z}|\mathbf{t}).
\end{equation}
Combining MSCD with (\ref{eq:1}), (\ref{eq:2}) and (\ref{eq:fgj3}), we obtain (\ref{eq:mscasfgJ}). 
Inequality (\ref{eq:mscasfgj}) now follows directly from (\ref{eq:mscasfgJ}).
\end{proof}

\begin{remark}
\label{rem:convenient}
Suppose that $t_{1},\ldots,t_{m}\in\mathbb{H}\setminus\Sigma$ are such that $j(t_{1}),\ldots,j(t_{m})\in\overline{\mathbb{Q}}$. 
Then under MSC we get $\mathrm{tr.deg.}_{\mathbb{Q}}\mathbb{Q}(\mathbf{t},j(\mathbf{t})) =\mathrm{tr.deg.}_{\mathbb{Q}}\mathbb{Q}(\mathbf{t}) = \dim_{G}(\mathbf{t}|\Sigma)$. 
So for any $z_{1},\ldots,z_{n}\in\mathbb{H}$ we can repeat the arguments in the proof of Lemma \ref{lem:fgjmscas} to get:
\begin{equation*}
\mathrm{tr.deg.}_{\mathbb{Q}(\mathbf{t})}\mathbb{Q}(\mathbf{z},\mathbf{t},j(\mathbf{z}),j(\mathbf{t})) \geq \dim_{G}(\mathbf{z}|\Sigma,\mathbf{t}).     
\end{equation*}
\end{remark}

\begin{lem}
\label{lem:msc+t}
Assume MSCD holds. 
Let $t_{1},\ldots,t_{m}\in\mathbb{H}\setminus\Sigma$ be a convenient tuple for $j$. 
Then for any $z_{1},\ldots,z_{n}\in\mathbb{H}$ we have
\begin{align*}
    \mathrm{tr.deg.}_{\mathbb{Q}}\mathbb{Q}(\mathbf{z},\mathbf{t},J(\mathbf{z}),J(\mathbf{t}))&\geq3\dim_{G}(\mathbf{z},\mathbf{t}|\Sigma) + \dim^{j}(\mathbf{t}), \mbox{ and }\\
    \mathrm{tr.deg.}_{\mathbb{Q}}\mathbb{Q}(\mathbf{z},\mathbf{t},j(\mathbf{z}),j(\mathbf{t}))&\geq\dim_{G}(\mathbf{z},\mathbf{t}|\Sigma) + \dim^{j}(\mathbf{t}).
\end{align*}
\end{lem}
\begin{proof}
Set $F = \mathbb{Q}\left(\mathbf{t},J\left(\mathbf{t}\right)\right)$. We proceed by contradiction, if $z_{1},\ldots,z_{n}\in\mathbb{H}$ are such
\[\mathrm{tr.deg.}_{\mathbb{Q}}\mathbb{Q}(\mathbf{z},\mathbf{t},j(\mathbf{z}),j(\mathbf{t}))<\dim_{G}(\mathbf{z},\mathbf{t}|\Sigma) + \dim^{j}(\mathbf{t}),\]
then we also get that
\[\mathrm{tr.deg.}_{\mathbb{Q}}\mathbb{Q}(\mathbf{z},\mathbf{t},J(\mathbf{z}),J(\mathbf{t}))<3\dim_{G}(\mathbf{z},\mathbf{t}|\Sigma) + \dim^{j}(\mathbf{t}).\]
Then using that $\mathbf{t}$ is convenient for $j$, Lemma \ref{lem:fgjmscas}, (\ref{eq:1}) and (\ref{eq:2}), we get:
\begin{align*}
3\dim_{G}(\mathbf{z}|\Sigma,\mathbf{t}) + \mathrm{tr.deg.}_{\mathbb{Q}}F &= 3\dim_{G}(\mathbf{z},\mathbf{t}|\Sigma) + \dim^{j}(\mathbf{t})\\
&> \mathrm{tr.deg.}_{\mathbb{Q}}\mathbb{Q}(\mathbf{z},\mathbf{t},J(\mathbf{z}),J(\mathbf{t}))\\
&= \mathrm{tr.deg.}_{\mathbb{Q}}F + \mathrm{tr.deg.}_{F}F(\mathbf{z},J(\mathbf{z}))\\
&\geq \mathrm{tr.deg.}_{\mathbb{Q}}F + 3\dim(\mathbf{z}|\Sigma,\mathbf{t}) + \dim^{j}(\mathbf{z}|\mathbf{t}).
\end{align*}
As $\dim^{j}(\mathbf{z}|\mathbf{t})\geq 0$, this gives a contradiction. 
\end{proof}

In the rest of the subsection we address the question of existence of convenient tuples, and furthermore finding a convenient tuple which contains the generators of a specific finitely generated field. 
We start by recalling the following result.

\begin{thm}{{\cite[Theorem 5.1]{aek-closureoperator}}}
\label{thm:mainclosureoperator}
Let $F$ be a subfield of $\mathbb{C}$ such that $\mathrm{tr.deg.}_{C_{j}}F$ is finite. 
Then there exist  $t_{1},\ldots,t_{m}\in\mathbb{H}\setminus C_{j}$ such that:
\begin{enumerate}
    \item[(A1):] $F\subseteq\overline{C_{j}\left(\mathbf{t},J\left(\mathbf{t}\right)\right)}$, and
    \item[(A2):] $\mathrm{tr.deg.}_{C_{j}}C_{j}\left(\mathbf{t},J\left(\mathbf{t}\right)\right) = 3\dim_{G}\left(\mathbf{t}|C_{j}\right) + \dim^{j}\left(\mathbf{t}\right)$.
\end{enumerate} 
\end{thm}

Assuming MSCD, we now refine this result to show that convenient tuples exist.

\begin{lem}
\label{lem:convenientgenerators}
Let $K\subset\mathbb{C}$ be a finitely generated subfield. 
Then MSCD implies that there exist $t_{1},\ldots,t_{m}\in\mathbb{H}\setminus\Sigma$ such that:
\begin{enumerate}
    \item[(c1):] $\overline{K}\subseteq\overline{\mathbb{Q}\left(\mathbf{t},J\left(\mathbf{t}\right)\right)}$,
    \item[(c2):] $\mathbf{t}$ is convenient for $j$.
\end{enumerate}
Furthermore, we can assume that $\dim_{G}(\mathbf{t}|\Sigma) = m$. 
\end{lem}
\begin{proof}
By Theorem \ref{thm:mainclosureoperator} there exist $\mathbf{t}_{1}:=(t_{1},\ldots,t_{k})\in(\mathbb{H}\setminus C_{j})^k$ such that 
\begin{enumerate}[(a)]
    \item $K\subseteq\overline{C_{j}\left(\mathbf{t}_{1},J\left(\mathbf{t}_{1}\right)\right)}$,
    \item $\mathrm{tr.deg.}_{C_{j}}C_{j}\left(\mathbf{t}_{1},J\left(\mathbf{t}_{1}\right)\right) = 3\dim_{G}\left(\mathbf{t}_{1}|C_{j}\right) + \dim^{j}\left(\mathbf{t}_{1}\right)$.
\end{enumerate}
As $\mathrm{tr.deg.}_{C_{j}}C_{j}\left(\mathbf{t}_{1},J\left(\mathbf{t}_{1}\right)\right)$ is finite, then there is a finitely generated field $F\subseteq C$ such that $\mathrm{tr.deg.}_{C_{j}}C_{j}\left(\mathbf{t}_{1},J\left(\mathbf{t}_{1}\right)\right) = \mathrm{tr.deg.}_{F}F\left(\mathbf{t}_{1},J\left(\mathbf{t}_{1}\right)\right)$. 
As $K$ is finitely generated, if $L$ denotes the compositum of $F$ and $K\cap C_{j}$, then $L$ has finite transcendence degree over $\mathbb{Q}$, so by \cite[Theorem 6.18]{aek-closureoperator}, MSCD implies that there exist $\mathbf{t}_{2}=(t_{k+1},\ldots,t_{m})\in(\mathbb{H}\cap C_{j})^{m-k}$ such that 
\begin{enumerate}[(i)]
    \item $L\subseteq\overline{\mathbb{Q}\left(\mathbf{t}_{2},J\left(\mathbf{t}_{2}\right)\right)}$,
    \item $\mathrm{tr.deg.}_{\mathbb{Q}}\mathbb{Q}\left(\mathbf{t}_{2},J\left(\mathbf{t}_{2}\right)\right) = 3\dim_{G}\left(\mathbf{t}_{2}\right|\Sigma)$.
\end{enumerate}
The paragraph immediately following \cite[Theorem 6.18]{aek-closureoperator} shows that the coordinates of $\mathbf{t}_2$ may be chosen outside of $\Sigma$. 
Let $\mathbf{t}=(\mathbf{t}_{1},\mathbf{t}_{2})$. 
By construction, the elements of $\mathbf{t}_{1}$ share no $G$-orbits with element of $\mathbf{t}_{2}$. 
Condition (c1) is satisfied by (a) and (i). 
As $F$ is contained in $L$, then condition (c2) is satisfied by (ii) and (b). 

The ``furthermore'' part follows from Schneider's theorem and the equivalence between (M1) and (M2) in \textsection\ref{sec:modular_polynomials}. 
\end{proof}

\subsection{Proof of Theorem \ref{thm:generic}}
\label{sec:generalcase}
We first set up some notation that will be kept through the rest of this subsection. 

Let $V\subset\mathbb{C}^{n}\times\mathrm{Y}(1)^{n}$ be a broad and free variety. 
Let $K\subset\mathbb{C}$ be a finitely generated subfield such that $V$ is defined over $K$ and let $t_{1},\ldots,t_{m}\in\mathbb{H}\setminus\Sigma$ be given by Lemma \ref{lem:convenientgenerators} (which assumes MSCD) applied to $K$ ($m$ may be zero, which happens when $K\subseteq\overline{\mathbb{Q}}$). 
We assume that $\dim_{G}(\mathbf{t}|\Sigma) = m$. 
Set $F:=\mathbb{Q}(\mathbf{t},j(\mathbf{t}))$. 

Choose $(\mathbf{x},\mathbf{y})\in V$ generic over $F$. 
By Lemma \ref{lem:genericpoint} and the freeness of $V$ we know that no coordinate of $(\mathbf{x},\mathbf{y})$ is in $\overline{F}$. 
Let $W\subseteq\mathbb{C}^{n+m}\times\mathrm{Y}(1)^{n+m}$ be the $\overline{\mathbb{Q}}$-Zariski closure of the point $(\mathbf{x},\mathbf{t},\mathbf{y},j(\mathbf{t}))$.

\begin{lem}
\label{lem:dimW}
$\dim W = 
\dim V+m+\dim^{j}(\mathbf{t})$.
\end{lem}
\begin{proof}
\begin{align*}
\dim W = \mathrm{tr.deg.}_{\mathbb{Q}}\mathbb{Q}(\mathbf{x},\mathbf{t},\mathbf{y},j(\mathbf{t}))
&= \mathrm{tr.deg.}_{\mathbb{Q}}\mathbb{Q}(\mathbf{t},j(\mathbf{t})) + \mathrm{tr.deg.}_{\mathbb{Q}(\mathbf{t},j(\mathbf{t}))}\mathbb{Q}(\mathbf{x},\mathbf{t},\mathbf{y},j(\mathbf{t}))\\
&= \dim_{G}(\mathbf{t}|\Sigma) + \dim^{j}(\mathbf{t}) + \dim V.\qedhere
\end{align*}
\end{proof}

We are now ready to prove the first main result.

\begin{proof}[Proof of Theorem \ref{thm:generic}]
By Proposition \ref{prop:secindimn}, we can reduce to the case $\dim V=n$. 
The case $n=1$ was proven in \cite[Theorem 1.2]{paper1}, so now we assume that $n>1$. 

Consider the parametric family  $(W_\mathbf{q})_{\mathbf{q}\in Q}$ of subvarieties of $W$ defined in Example \ref{ex:parametric family}. 
Let $\mathscr{S}$ be the finite collection of proper special subvarieties of $\mathrm{Y}(1)^{n}$ given by Corollary \ref{cor:horizontalUMZP} applied to the parametric family $\left(\pi_{\mathrm{Y}}(W_{\mathbf{q}})\right)_{\mathbf{q}\in Q}$. 
Let $N$ be the integer given by applying Proposition \ref{prop:horizontalboundedcomplexity} to the family $(W_{\mathbf{q}})_{\mathbf{q}\in Q}$. 
Without loss of generality we may assume that $\Delta_b(T)\leq N$ for all $T\in\mathscr{S}$.

Let $W_{1}\subseteq W$ be a Zariski open subset such that for all $(\mathbf{a},\mathbf{b})\in W_{1}$ we have that $\mathbf{b}\notin T_0$ for all $T_0\in\mathscr{S}$. 

Observe that we can choose $(\mathbf{z},j(\mathbf{z}))\in V$ such that $(j(\mathbf{z}),j(\mathbf{t}))\notin T_0$ for all $T_0\in\mathscr{S}$. 
This is because every equation defining $T_0$ either gives a modular dependence between two coordinates of $j(\mathbf{z})$, or it gives a modular dependence between a coordinate of $j(\mathbf{z})$ and a coordinate of $j(\mathbf{t})$. 
As $V$ is free, satisfies (EC), and $\mathscr{S}$ is finite, we can find the desired point. 
This way we get that $(\mathbf{z},\mathbf{t},j(\mathbf{z}),j(\mathbf{t}))\in W_1$.

We will show that $\dim_{G}(\mathbf{z},\mathbf{t}|\Sigma)=n+m$. 
For this we proceed by contradiction, so suppose that $\dim_{G}(\mathbf{z},\mathbf{t}|\Sigma)<n+m$. 
Let $T$ be the special closure of the point $(j(\mathbf{z}),j(\mathbf{t}))$. 
Observe that $\dim T = \dim_{G}(\mathbf{z},\mathbf{t}|\Sigma)$.

Since $\dim_{G}(\mathbf{z},\mathbf{t}|\Sigma)<n+m$ and $\dim_{G}(\mathbf{t}|\Sigma) = m$, then that at least one of the following happen:
\begin{enumerate}[(a)]
    \item There is $i\in\left\{1,\ldots,n\right\}$, $k\in\left\{1,\ldots,m\right\}$, and $g\in G$ such that $gz_i=t_k$.
    \item There are $i,k\in\left\{1,\ldots,n\right\}$ (possibly equal) and $g\in G$ such that $gz_i=z_k$. 
\end{enumerate}
Let $M\subset\mathbb{C}^{n+m}$ be a proper subvariety defined by M\"obius relations defined over $\mathbb{Q}$ and/or setting some coordinates to be a constant in $\Sigma$, satisfying $(\mathbf{z},\mathbf{t})\in M$. 
In other words, $M$ is a M\"obius variety witnessing the relations found in (a) and (b). 
Thus $W\cap (M\times\mathrm{Y}(1)^{n+m})$ is an element of the family $(W_{\mathbf{q}})_{\mathbf{q}\in Q}$, call it $W_{M}$. 
We remark that $W_{M}$ is defined over $\overline{\mathbb{Q}}$. 

We now claim that $\dim W_M < \dim W$. 
Indeed, if the dimensions were equal, then $W_M = W$ and so $V\subset\mathrm{Pr}_{\mathbf{i}}(W_M)$. 
An equation defining $M$ coming from either (a) or (b) would then immediately contradicts freeness of $V$.

Now let $X$ be the irreducible component of $W_{M}\cap(\mathbb{C}^{n+m}\times T)$ containing $(\mathbf{z},\mathbf{t}, j(\mathbf{z}),j(\mathbf{t}))$. 
Observe that $X$ is defined over $\overline{\mathbb{Q}}$. 
Then by Lemma \ref{lem:msc+t} (which assumes MSCD) we get
\begin{equation*}
    \dim X\geq \mathrm{tr.deg.}_{\mathbb{Q}}\mathbb{Q}(\mathbf{z},\mathbf{t}, j(\mathbf{z}),j(\mathbf{t})) \geq \dim_{G}(\mathbf{z},\mathbf{t}|\Sigma)+\dim^{j}(\mathbf{t}).
\end{equation*}
On the other hand, using that $\dim W_M<\dim W$ and Lemma \ref{lem:dimW} we get: 
\begin{align*}
\dim W_{M} + \dim\mathbb{C}^{n+m}\times T - \dim\mathbb{C}^{n+m}\times\mathrm{Y}(1)^{n+m} &< \dim W+\dim T-n\\
&= \dim T + \dim^j(\mathbf{t})\\
&\leq   \dim_{G}(\mathbf{z},\mathbf{t}|\Sigma)+ \dim^j(\mathbf{t})\\
&\leq \dim X.
\end{align*}
This shows that $X$ is  an atypical component of $W_{M}\cap(\mathbb{C}^{n+m}\times T)$. 
But then by Corollary \ref{cor:horizontalUMZP} there must be $T_0\in\mathscr{S}$ such that $X\subseteq\mathbb{C}^{n+m}\times T_0$, which is a contradiction since $(j(\mathbf{z}),j(\mathbf{t}))\notin T_0$.

We have thus shown that $\dim_{G}(\mathbf{z},\mathbf{t}|\Sigma)=n+m$, which proves the theorem by Lemma \ref{lem:fgjmscas}. 
\end{proof}

\begin{remark}
\label{rem:transcoord}
In the proof of Theorem \ref{thm:generic}, if we somehow knew beforehand that the coordinates of $j(\mathbf{z})$ are all transcendental over $\mathbb{Q}$, then we would not need to invoke MZP, because as $X$ is defined over $\overline{\mathbb{Q}}$, then any coordinate that is constant on $X$ must equal an element of $\overline{\mathbb{Q}}$. 
So if we know that the coordinates of $j(\mathbf{z})$ are all transcendental over $\mathbb{Q}$, then we would know that $X$ does not have constant coordinates, hence $X$ is strongly atypical and we can apply Theorem \ref{thm:horizontalweakzp}. 
\end{remark}

\subsection{Blurring}
\label{subsec:blurring}
We will now give an analogue of Theorem \ref{thm:generic} for the the so-called \emph{blurring} of $j$ by $G$ (see \cite{aslanyan-kirby} for more details).

\begin{thm}
\label{thm:genericblurred}
Let $V\subseteq\mathbb{C}^{2n}$ be a broad and free irreducible variety with. Then MSCD and MZP imply that for every finitely generated field $K\subset\mathbb{C}$, there are $g_{1},\ldots,g_{n}\in G$ such that $V$ has a point of the form 
\[(z_{1},\ldots,z_{n},j(g_{1}z_{1}),\ldots,j(g_{n}z_{n})),\] 
with $(z_1,\ldots,z_n)\in \mathbb{H}^n$, which is generic over $K$.\footnote{In fact, using \cite[Theorem 3.1]{aslanyan-kirby}, $g_{1},\ldots,g_{n}$ may be chosen to be upper-triangular. }
\end{thm}

To prove Theorem \ref{thm:genericblurred}, proceed in exactly the same way as in the proof of Theorem \ref{thm:generic} (recall that Theorem \ref{thm:blurredj} already ensures the existence of a Zariski dense set of solutions). 
The key (trivial) observation is that for all $z,w\in\mathbb{H}$ and all $g,h\in G$ we have that $j(z)$ and $j(w)$ are modularly dependent if and only if $j(gz)$ and $j(hw)$ are modularly dependent. 
This manifests in the different ingredients of the proof as follows.
\begin{enumerate}[(a)]
\item MSCD implies that for every $g_{1},\ldots,g_{n}\in G$ and every $z_{1},\ldots,z_{n}\in\mathbb{H}^{n}$ we have
\[\mathrm{tr.deg.}_{\mathbb{Q}}\mathbb{Q}(z_{1},\ldots,z_{n},j(g_{1}z_{1}),\ldots,j(g_{n}z_{n}))\geq\dim_{G}(z_{1},\ldots,z_{n}|\Sigma).\]
\item One can restate the inequalities of \textsection\ref{subsec:transineq} in a straightforward way.  
This is because for any field $F$, for all $z_{1},\ldots,z_{n}\in\mathbb{H}$ and all $g_{1},\ldots,g_{n}\in G$ we have that
\[\mathrm{tr.deg.}_{F}F(z_{1},\ldots,z_{n},j(z_{1}),\ldots,j(z_{n})) = \mathrm{tr.deg.}_{F}F(z_{1},\ldots,z_{n},j(g_{1}z_{1}),\ldots,j(g_{n}z_{n})).\]
So, for example, we can restate Lemma \ref{lem:msc+t} as: let $t_{1},\ldots,t_{m}\in\mathbb{H}\setminus\Sigma$ be a convenient tuple for $j$, then for all $z_{1},\ldots,z_{n}\in\mathbb{H}$ and all $g_{1},\ldots,g_{n}\in G$ we have that:
\[\mathrm{tr.deg.}_{\mathbb{Q}}\mathbb{Q}(\mathbf{z},\mathbf{t},j(g_{1}z_{1}),\ldots,j(g_{n}z_{n}),j(\mathbf{t}))\geq \dim_{G}(\mathbf{z},\mathbf{t}|\Sigma) + \dim^{j}(\mathbf{t}).\]
\item Corollary \ref{cor:horizontalUMZP} still applies as stated. 
\end{enumerate}

\subsection{Special solutions}
\label{subsec:secialsol}
In this subsection we show that for a variety $V$ as in Theorem \ref{thm:generic}, the set
\[V_{\Sigma}:=\left\{(\boldsymbol{\alpha},j(\boldsymbol{\beta}))\in V : j(\boldsymbol{\alpha}),j(\boldsymbol{\beta})\mbox{ are special points}\right\}\]
cannot be Zariski dense in $V$. 
We show this in Proposition \ref{prop:special}. 
The proof relies on the results of \cite{pila:andre-oort} (Pila's proof of the Andr\'e--Oort conjecture for powers of the modular curve).

Let $X$ denote a subvariety of $\mathrm{Y}(1)^{n}$ defined over a finitely generated field $F$. 
Let $\mathcal{K}_{X}$ denote the set of all subfields of $\overline{F}$ over which $X$ is definable. 
Define
\[\delta_{F}(X):=\min\left\{[K:F]  \mid K\in\mathcal{K}_{X}\right\}.\]

\begin{defn}
Given a special point $\tau\in\Sigma$, there is a unique quadratic polynomial $ax^{2}+bx+c$, with $a,b,c\in\mathbb{Z}$, $\mathrm{gcd}(a,b,c)=1$ and $a>0$, such that $a\tau^{2}+b\tau+c=0$. Let $D_{\tau} = b^{2}-4ac$. 
If $\boldsymbol{\tau}=(\tau_{1},\ldots,\tau_{n})\in\Sigma^{n}$, we define 
\[\mathrm{disc}(\boldsymbol{\tau}):=\max\left\{D_{\tau_{1}},\ldots,D_{\tau_{n}}\right\}.\]
For every $\gamma\in\mathrm{SL}_{2}(\mathbb{Z})$ we have that $\mathrm{disc}(\gamma\tau) = \mathrm{disc}(\tau)$, so it makes sense to define:
\[\mathrm{disc}(j(\boldsymbol{\tau})):=\mathrm{disc}(\boldsymbol{\tau}).\]
\end{defn}

\begin{defn}
A basic special subvariety $T\subset\mathrm{Y}(1)^{n}$ is defined by certain modular polynomials in the ring $\mathbb{C}[Y_{1},\ldots,Y_{n}]$. 
Suppose that $1\leq i_{1}<\cdots<i_{\ell}\leq n$ denote some indices such that, for every $s\in\left\{1,\ldots,s\right\}$, the variable $Y_{i_{s}}$ does not appear in any the polynomials defining $T$. 
In particular this implies that $\mathrm{pr}_{\mathbf{i}}(T) = \mathrm{Y}(1)^{\ell}$. 
Following \cite{pila:andre-oort}, given $\mathbf{y}\in\mathrm{Y}(1)^{\ell}$ we call $T\cap\mathrm{pr}^{-1}(\mathbf{y})$ the \emph{translate} of $T$ by $\mathbf{y}$, and denote it as $\mathrm{tr}(T,\mathbf{y})$. 

We remark that $\mathrm{tr}(T,\mathbf{y})$ is a weakly special subvariety, and if every coordinate of $\mathbf{y}$ is a special point, then $\mathrm{tr}(T,\mathbf{y})$ is a special subvariety.
\end{defn}

\begin{prop}
\label{prop:special}
Let $V\subseteq\mathbb{C}^{n}\times\mathrm{Y}(1)^{n}$ be an irreducible variety which is modularly free and defined over $\overline{\mathbb{Q}}$. Suppose that the dimension of the generic fibre of $\pi_{\mathbb{C}}\upharpoonright_{V}$ is less than $n$. 
Then $V_{\Sigma}$ is not Zariski dense in $V$.
\end{prop}
\begin{proof}
Given $\mathbf{a}\in\mathbb{C}^{n}$, let
\[V_{\mathbf{a}}:=\left\{\mathbf{b}\in\mathrm{Y}(1)^{n} : (\mathbf{a},\mathbf{b})\in V\right\}.\]
In this way we obtain a parametric family $(V_{\mathbf{a}})_{\mathbf{a}\in\pi_{\mathbb{C}}(V)}$ of subvarieties of $\mathrm{Y}(1)^{n}$. 

Let $F\subset\overline{\mathbb{Q}}$ be a finitely generated field such that $V$ is definable over $F$. 
For every $(\mathbf{a},\mathbf{b})\in V$, $V_{\mathbf{a}}$ is definable over the field $F(\mathbf{a})$. 
If $(\mathbf{a},\mathbf{b})\in V_{\Sigma}$, then since every coordinate of $\mathbf{a}$ defines a degree 2 extension of $\mathbb{Q}$, we get that $\delta_{\mathbb{Q}}(V_{\mathbf{a}})\leq 2^{n}[F:\mathbb{Q}]$.  

By \cite[Theorem 13.2]{pila:andre-oort} (uniform Andr\'e-Oort) there is a finite collection $\mathscr{S}$ of basic special subvarieties of $\mathrm{Y}(1)^{n}$ such that for every $T\in\mathscr{S}$ there is a constant $C>0$ (depending only on $n$, $V$ and $T$) such that, if $\mathbf{i}=(i_{1},\ldots,i_{\ell})$ denotes indices as in the definition above, then for every $\mathbf{a}\in \mathbb{C}^{n}$  and for every $\boldsymbol{\tau}\in\mathbb{H}^{\ell}$, if $\delta_{\mathbb{Q}}(V_{\mathbf{a}}) \leq 2^{n}[F:\mathbb{Q}]$ and $\mathrm{tr}(T,j(\boldsymbol{\tau}))$ is a maximal special subvariety of $V_{\mathbf{a}}$, then
\[\mathrm{disc}(\boldsymbol{\tau}) \leq C\quad\mbox{ and }\quad\delta_{\mathbb{Q}}(\mathrm{tr}(T,j(\boldsymbol{\tau})))\leq C.\]
We remark that by \cite[Proposition 13.1]{pila:andre-oort}, all the maximal special subvarieties of $V_{\mathbf{a}}$ are of the form $\mathrm{tr}(T,j(\boldsymbol{\tau}))$.

As explained in \cite[\textsection 5.6]{pila:andre-oort}, there are finitely many values of $j(\boldsymbol{\tau})$ subject to $\mathrm{disc}(\boldsymbol{\tau}) \leq C$. 
Therefore, there is a finite collection $\mathscr{S}^{\star}$ of special subvarieties of $\mathrm{Y}(1)^{n}$ such that for every $(\mathbf{a},\mathbf{b})\in V_{\Sigma}$, the maximal special subvarieties of $V_{\mathbf{a}}$ are elements of $\mathscr{S}^{\star}$. 

By hypothesis and the fibre-dimension theorem there is a Zariski open subset $V'\subset V$ such that for all $(\mathbf{a},\mathbf{b})\in V'$ we have
\[\dim V\cap\pi_{\mathbb{C}}^{-1}(\mathbf{a}) = \dim V - \dim\pi_{\mathbb{C}}(V) < \dim V = n.\]
This means that $(V_{\mathbf{a}})_{\mathbf{a}\in\pi_{\mathbb{C}}(V')}$ is a family of proper subvarieties of $\mathrm{Y}(1)^{n}$. 
Therefore the elements of $\mathscr{S}^{\star}$ corresponding to the family $(V_{\mathbf{a}})_{\mathbf{a}\in\pi_{\mathbb{C}}(V')}$ are proper special subvarieties of $\mathrm{Y}(1)^{n}$. 
As $V$ is free, the intersection
\[V\cap \bigcup_{T\in\mathscr{S}^{\star}}\mathbb{C}\times T\]
is contained in a proper subvariety of $V$. 
\end{proof}

\begin{remark}
If $V\subseteq\mathbb{C}^{n}\times\mathrm{Y}(1)^{n}$ is not definable over $\overline{\mathbb{Q}}$, then it is immediate by Galois-theoretic reasons that $V\left(\overline{\mathbb{Q}}\right)$ cannot be Zariski dense in $V$. Since the elements of $V_{\Sigma}$ are all elements of $V\left(\overline{\mathbb{Q}}\right)$, then $V_{\Sigma}$ is also not Zariski dense in $V$ in this case.
\end{remark}

\begin{ex}
The condition on the dimension of the fibres in Proposition \ref{prop:special} is necessary. 
For example, consider the variety $V\subseteq\mathbb{C}^{2}\times\mathrm{Y}(1)^{2}$ of dimension 3 defined by the single equation $X_{1} = X_{2}^2$. 
It is clear that $V$ is modularly free, $\dim\pi_{\mathbb{C}}(V)=1>0$, and $V$ is defined over $\overline{\mathbb{Q}}$.

Let $\tau$ be any element of $\Sigma$ with positive real part (of which there are infinitely many), and observe that then $\tau^2\in\mathbb{H}$. 
Since $\tau^2\notin\mathbb{R}$ and $\mathbb{Q}(\tau^2)\subseteq\mathbb{Q}(\tau)$, then $1 < [\mathbb{Q}(\tau^2):\mathbb{Q}] \leq [\mathbb{Q}(\tau):\mathbb{Q}] = 2$, thus showing that $\tau^2\in\Sigma$. 
Therefore, given any two singular moduli $b_{1},b_{2}\in\mathrm{Y}(1)$, we have that $(\tau^2,\tau,b_{1},b_{2})\in V_{\Sigma}$. 
Since special points are Zariski dense in $\mathrm{Y}(1)^2$ and $X_{1} = X_{2}^2$ defines a curve in $\mathbb{C}^2$, the Zariski closure of $V_{\Sigma}$ is equal to $V$. 
\end{ex}

\begin{lem}
\label{lem:t-special}
Let $V\subseteq\mathrm{Y}(1)^{n}\times Q$ be a parametric family of irreducible subvarieties of $\mathrm{Y}(1)^{n}$, and let $d$ be a positive integer. 
Let $\mathbf{t}\in(\mathbb{H}\setminus\Sigma)^{n}$ be a convenient tuple for $j$, and suppose that $V$ is definable over $F=\mathbb{Q}(\mathbf{t},J(\mathbf{t}))$. 
Then there is a finite collection of $j(\mathbf{t})$-special subvarieties $\mathscr{S}$ such that for every $\mathbf{q}\in Q$ with $[F(\mathbf{q}):F]\leq d$, if $V_\mathbf{q}$ is a $j(\mathbf{t})$-special subvariety, then $V_\mathbf{q}\in\mathscr{S}$.
\end{lem}
\begin{proof}
Using definability of dimensions as we did in the proof of Corollary \ref{cor:horizontalUMZP}, we may first restrict to the subfamily $Q'\subseteq Q$ such that for all $\mathbf{q}\in Q'$ we have that $V_\mathbf{q}$ is a proper subvariety of $\mathrm{Y}(1)^{n}$. 
So we will assume now that $V\subseteq\mathrm{Y}(1)^{n}\times Q$ is a parametric family of proper irreducible subvarieties of $\mathrm{Y}(1)^{n}$.

Let $T$ be a basic special subvariety of $\mathrm{Y}(1)^{n}$, and suppose that some translate $\mathrm{tr}(T,\mathbf{y})$ is equal to $V_\mathbf{q}$, for some $\mathbf{q}\in Q$. 
Then since $\mathrm{tr}(T,\mathbf{y})$ is definable over $\mathbb{Q}(\mathbf{y})$, then so is $V_\mathbf{q}$. 
On the other hand, $V_\mathbf{q}$ is definable over $F(\mathbf{q})$, and so by \cite[Lemma 7.1]{paper1} there is a positive integer $d'$ depending only on $d$ and $F$ such that if $[F(\mathbf{q}):F]\leq d$ and $\mathbf{q}$ is algebraic over $\mathbb{Q}(j(\mathbf{t}))$, then $[\mathbb{Q}(j(\mathbf{t}),\mathbf{q}):\mathbb{Q}(j(\mathbf{t}))]\leq d'$. 
So if $\mathrm{tr}(T,\mathbf{y}) = V_\mathbf{q}$ and $[F(\mathbf{q}):F]\leq d$, then $[\mathbb{Q}(j(\mathbf{t}),\mathbf{y}):\mathbb{Q}(j(\mathbf{t}))]\leq d'$. 

By Remark \ref{rem:heckeorbit} and \cite[Theorem 13.2]{pila:andre-oort} we have that the set
\[B=\left\{c\in\mathrm{He}(j(\mathbf{t}))\cup j(\Sigma) : [\mathbb{Q}(j(\mathbf{t}),c):\mathbb{Q}(j(\mathbf{t}))]\leq d'\right\}\]
is finite. 
Indeed, by \cite[Lemma 7.1]{paper1} we can find a positive integer $d''$ depending only on $d'$ and $j(\mathbf{t})$ such that if $c\in B\cap j(\Sigma)$, then $[\mathbb{Q}(c):\mathbb{Q}]\leq d''$, and so \cite[Theorem 13.2]{pila:andre-oort} applies. 
Otherwise $c\in \mathrm{He}_{d''}(j(\mathbf{t}))$, and so Remark \ref{rem:heckeorbit} applies directly.

This shows that, for a given basic special variety $T$, there are only finitely many tuples $\mathbf{y}$ which satisfy the following three conditions
\begin{enumerate}[(i)]
    \item $\mathrm{tr}(T,\mathbf{y})$ is a $j(\mathbf{t})$-special subvariety,
    \item $\mathrm{tr}(T,\mathbf{y}) = V_\mathbf{q}$ for some $\mathbf{q}\in Q$, and
    \item $[F(\mathbf{q}):F]\leq d$.
\end{enumerate} 

We also have that when $\mathrm{tr}(T,\mathbf{y})=V_\mathbf{q}$,
\[\dim V_\mathbf{q}\cap\mathrm{tr}(T,\mathbf{y}) > \dim V_\mathbf{q} + \mathrm{tr}(T,\mathbf{y})-n,\]
so the intersection is atypical. 
By \cite[Thoerem 4.2]{vahagn3} there is a finite collection $\mathscr{S}$ of basic special subvarieties such that for every $\mathbf{q}\in Q$, every maximal $j(\mathbf{t})$-atypical component of $V_\mathbf{q}$ is a translate of some $T\in\mathscr{S}$. 
This finishes the proof.
\end{proof}

A straightforward adaptation of the proof of Proposition \ref{prop:special} gives now the following. 

\begin{prop}
\label{prop:Fspecial}
Suppose that $t_{1},\ldots,t_{m}\in\mathbb{H}\setminus\Sigma$ define a convenient tuple for $j$, and set $F:=\mathbb{Q}(\mathbf{t},J(\mathbf{t}))$. 
Let $V\subseteq\mathbb{C}^{n}\times\mathrm{Y}(1)^{n}$ be an irreducible variety which is modularly free and defined over $\overline{F}$. 
Suppose that the dimension of the generic fibre of $\pi_{\mathbb{C}}\upharpoonright_{V}$ is less than $n$. 
Then the set
\[V_{\Sigma,F}:=\left\{(\boldsymbol{\alpha},j(\boldsymbol{\beta})) : \dim_{G}(\boldsymbol{\alpha}|\Sigma\cup\mathbf{t}) = \dim_{G}(\boldsymbol{\beta}|\Sigma\cup\mathbf{t})=0\right\}\]
is not Zariski dense in $V$.
\end{prop}

\subsection{Results which do not require MZP}

In this subsection we review some cases of Theorem \ref{thm:generic} which do not depend on MZP. 
We recall that the case $n=1$ (i.e.~plane curves) was proven in \cite[Theorem 1.2]{paper1}, and does not rely on MZP. 
We will show that the same is true for $n=2$. 
First we need the following lemma.

\begin{lem}
\label{lem:curvefalg}
Let $V\subseteq\mathbb{C}\times\mathrm{Y}(1)$ be a free irreducible curve. 
Let $t_{1},\ldots,t_{m}\in\mathbb{H}$ be a convenient tuple for $j$ and set $F=\mathbb{Q}(\mathbf{t},J(\mathbf{t}))$. 
If $V$ is definable over $\overline{F}$, then MSCD implies that the set
\[A(V):=\left\{(z,j(z))\in V : j(z)\in\overline{F}\right\}\]
is not Zariski dense in $V$.
\end{lem}
\begin{proof}
By freeness and irreducibility of $V$ we have that $\dim\pi_{\mathrm{Y}}(V) = 1$, and so the standard fibres of the restriction $\pi_{\mathrm{Y}}\upharpoonright_{V}$ have dimension 0, so there is a non-empty Zariski open subset $U\subset V$ such that if $(z,j(z))\in U\cap A(V)$, then $z\in\overline{F}$. 

Similarly we have that $\dim\pi_{\mathbb{C}}(V) = 1$. Let $d\in\mathbb{N}$ be the degree of the restriction $\pi_{\mathbb{C}}\upharpoonright_{V}$. 
Suppose $(z,j(z))\in U\cap A(V)$, then as $z\in\overline{F}$, Lemma \ref{lem:fgjmscas} implies that $z$ is in the same $G$-orbit of some element of $\Sigma\cup\mathbf{t}$. 
Using Proposition \ref{prop:special} we can shrink $U$ if necessary, and so without loss of generality we may assume that $z\notin\Sigma$, therefore it must be that $z$ is in the $G$-orbit of some element of $\mathbf{t}$. 
This implies that $z\in F$, and so we get that $[F(j(z)) : F]\leq d$. 
Since $j(z)$ is modularly dependent over $j(\mathbf{t})$, this means that there are only finitely many possible values for $j(z)$ by Remark \ref{rem:heckeorbit}. 
This finishes the proof. 
\end{proof}

\begin{thm}
Let $V\subseteq\mathbb{C}^{2}\times\mathrm{Y}(1)^2$ be a broad and free variety. 
If $V$ satisfies (EC), then MSCD implies that $V$ satisfies (SEC).
\end{thm}
\begin{proof}
By Proposition \ref{prop:secindimn} we may assume that $\dim V=2$. 
Let $t_{1},\ldots,t_{m}\in\mathbb{H}\setminus\Sigma$ be a convenient tuple for $j$, set $F=\mathbb{Q}(\mathbf{t},J(\mathbf{t}))$, and assume that $V$ is definable over $\overline{F}$. 

By the fibre-dimension theorem, Proposition \ref{prop:Fspecial} and Lemma \ref{lem:curvefalg} we can first find a non-empty Zariski open subset $U\subset V$ satisfying the following three conditions:
\begin{enumerate}[(a)]
    \item If $\mathbf{i}=(1)$ or $(2)$, then for all $(a,b)\in\mathrm{Pr}_{\mathbf{i}}(U)$ we have that $\dim\mathrm{Pr}_{\mathbf{i}}^{-1}((a,b))\cap V = \dim V - \dim\mathrm{Pr}_{\mathbf{i}}(V)$.
    \item If $(z_{1},z_{2},j(z_{1}),j(z_2))\in U\cap\mathrm{E}_{j}^{2}$, then there is $i\in\{1,2\}$ such that $z_{i}\notin\Sigma\cup G\mathbf{t}$.
    \item If $\mathbf{i}=(1)$ or $(2)$ and $\dim\mathrm{Pr}_{\mathbf{i}}(V)=1$, then for all $(z_{1},z_{2},j(z_{1}),j(z_2))\in U\cap\mathrm{E}_j^2$ we have that $\mathrm{pr}_{\mathbf{i}}(j(z_1),j(z_2))\notin\overline{F}$.
\end{enumerate}
\begin{claim}
The set
\[A(V):=\left\{(z_{1},z_{2},j(z_{1}),j(z_2))\in V\cap\mathrm{E}_j^2 : j(z_1)\in\overline{F}\mbox{ or } j(z_2)\in\overline{F}\right\}\]
is not Zariski dense in $V$.
\end{claim}
\begin{proof}
Proceed by contradiction, so suppose that $A(V)$ is Zariski dense in $V$. 
Then for some $i\in\{1,2\}$ we have that the subset  
\[A_i(V):=\{(z_{1},z_{2},j(z_{1}),j(z_2))\in A(V) : j(z_i)\in\overline{F}\}\]
is also Zariski dense in $V$.
Without loss of generality, we assume that $A_2(V)$ is Zariski dense in $V$.
Then we can choose $(z_{1},z_{2},j(z_{1}),j(z_2))\in U\cap A_2(V)$. 
Also, since $A_2(V)$ is Zariski dense in $V$, then $\mathrm{Pi}_{(2)}(A_2(V))$ is Zariski dense in $\mathrm{Pr}_{(2)}(V)$. 
By broadness of $V$ we know that $\dim\mathrm{Pr}_{(2)}(V)\geq 1$, and using condition (c) in the definition of $U$ we conclude that $\dim\mathrm{Pr}_{(2)}(V)=2$.
Thus generic fibres of the restriction $\mathrm{Pr}_{(2)}\upharpoonright_{V}$ have dimension 0. 

By condition (b) in the definition of $U$, we have that either $z_{1}\notin\Sigma\cup G\mathbf{t}$ or $z_{2}\notin\Sigma\cup G\mathbf{t}$ (or both). 
We now show that both cases imply that $z_2$ is transcendental over $F$. 
If $z_2\notin\Sigma\cup G\mathbf{t}$, then by Lemma \ref{lem:fgjmscas} we get that $\mathrm{tr.deg.}_{F}F(z_{2},j(z_{2})) \geq 1$, and since we are assuming that $j(z_2)\in \overline{F}$, we conclude that $z_2\notin\overline{F}$.
On the other hand, if $z_{1}\notin\Sigma\cup G\mathbf{t}$, then by Lemma \ref{lem:fgjmscas} (which assumes MSCD) we get that $\mathrm{tr.deg.}_{F}F(z_{1},j(z_{1})) \geq 1$. 
By condition (a) in the definition of $U$, we know that the fibre in $V$ above $(z_2,j(z_2))$ has dimension 0, and so $\mathrm{tr.deg.}_{F}F(z_{2},j(z_{2})) \neq 0$, which as before gives that $z_2\notin\overline{F}$. 

We will now show that we can reduce to the case where both $z_1$ and $z_2$ are transcendental over $F$. 
For this we will separate into cases depending on the value of $\dim\pi_{\mathbb{C}}(V)$ (which is at least 1, by freeness).
We first consider the case $\dim\pi_{\mathbb{C}}(V)=1$. 
By the freeness of $V$ there is a Zariski open subset $\mathcal{O}\subseteq\pi_{\mathbb{C}}(V)$ such that for every $(x_1,x_2)\in\mathcal{O}$ we have that if $x_2\notin\overline{F}$, then $x_1\notin\mathcal{F}$. 
Therefore, we may shrink $U$ if necessary to ensure that both $z_1,z_2\notin \overline{F}$. 

Now suppose that $\dim\pi_{\mathbb{C}}(V)=2 =\dim V$ (the following paragraph is just an adapation of an argument present in the proof of \cite[Proposition 7.4]{paper1}). 
We may then shrink $U$ so that for all $(x_1,x_2,y_1,y_2)\in U$ we have that (see e.g.~the definition of triangular varieties in \cite[\S 5]{paper1})
\[\mathrm{tr.deg.}_FF(x_1,x_2,y_1,y_2) = \mathrm{tr.deg.}_FF(x_1,x_2).\]
By Lemma \ref{lem:fgjmscas}  we get that
\[\mathrm{tr.deg.}_FF(z_1,z_2) = \mathrm{tr.deg.}_FF(z_1,z_2,j(z_1),j(z_2)) \geq \dim_G(z_1,z_2|\Sigma,\mathbf{t}).\]
From this we deduce that
\[\mathrm{tr.deg.}_FF(z_1,z_2) = \dim_G(z_1,z_2|\Sigma,\mathbf{t}).\]
If $z_1\in\overline{F}$, then this equality implies that $z_1\in \Sigma\cup G\mathbf{t}$. 
There is a positive integer $d_1$, which depends only on the equations defining $V$, such that for all $(x_1,x_2,y_1,y_2)\in U$ we have $[F(x_1,x_2,y_1):F(x_1,x_2)]\leq d_1$. 
Then in particular $[F(z_1,z_2,j(z_1)):F(z_1,z_2)]\leq d_1$, but since $z_2$ is transcendental over $F$, we get $[F(z_1,j(z_1)):F(z_1)]\leq d_1$. 
We can now find another positive integer $D_1$ depending only on $V$ and $F$ such that, if $z_1\in \Sigma$, then $[\mathbb{Q}(z_1,j(z_1)):\mathbb{Q}(z_1)]\leq D_1$ (observe that $F\cap\overline{\mathbb{Q}}$ is a number field).
On the other hand, if $z_1\in G\mathbf{t}$, let us say that $z \in Gt_i$, then $F\cap \overline{\mathbb{Q}(t_i,j(t_i))}$ is a finite extension of $\mathbb{Q}(t_i,j(t_i))$. 
So, similarly as to the case $z_1\in \Sigma$, we get a positive integer $D_2$ depending only on $V$ and $F$ such that,  $z_1\in G t_i$ for some $i\in\{1,,\ldots,m\}$, then $[\mathbb{Q}(z_1,t_i,j(z_1),j(t_i)):\mathbb{Q}(z_1,t_i,j(t_i))]\leq D_2$.
In either case, we may use Remark \ref{rem:heckeorbit} and \cite[Lemma 2.1]{paper1} to conclude that there are only finitely many possible values for $j(z_1)$. 
By shrinking $U$, we may avoid these values. 

So from now on, we assume that $z_1,z_2\notin\overline{F}$. 
Since $j(z_2)\in\overline{F}$ and $V$ is free, then $(z_1,z_2,j(z_1),j(z_2))$ is not generic in $V$ over $F$.
By Lemma \ref{lem:fgjmscas}  we have that
\[2=\dim V> \mathrm{tr.deg.}_{F}F(z_1,z_2,j(z_1),j(z_2)) \geq \dim_{G}(z_{1},z_2|\Sigma,\mathbf{t}),\]
so we must have that $z_1\in G\cdot z_2$.  
Let $d\in\mathbb{N}$ be the degree of the restriction $\mathrm{Pr}_{(2)}\upharpoonright_{V}$, then 
\[[F(z_{1},z_{2},j(z_{1}),j(z_{2})) : F(z_{2},j(z_{2}))]\leq d.\]
Since $z_1\in G\cdot z_2$ and $z_{2}$ is transcendental over $F$, then by \cite[Lemma 7.1]{paper1}
\begin{align*}
    [F(z_{1},z_{2},j(z_{1}),j(z_{2})) : F(z_{2},j(z_{2}))] &= [F(z_2,j(z_{1}),j(z_{2})) : F(j(z_{2}))]\\
    &= [F(j(z_{1}),j(z_{2})) : F(j(z_{2}))]\\
    &\leq d.
\end{align*}
By Remark \ref{rem:heckeorbit} we conclude that there are only finitely many modular polynomials $\Phi_{1},\ldots,\Phi_{N}$ such that if $(z_1,z_2,j(z_1),j(z_2))\in U\cap A_2(V)$, then $\prod_{k=1}^{N}\Phi_{k}(j(z_{1}),j(z_2))=0$. 
Since $V$ is free, we may shrink $U$ to avoid these finitely many modular polynomials, while still preserving Zariski openness. 
This contradicts the density of $A(V)$. 
\end{proof}

By the Claim and (EC) we can find $(z_1,z_2,j(z_1),j(z_2))\in U$ satisfying $j(z_1),j(z_2)\notin\overline{F}$  are Zariski dense in $V$. 
By Remark \ref{rem:transcoord}, this finishes the proof. 
\end{proof}

In higher dimensions, inspired by the results in \cite{weakSEC}, we get the following theorem.

\begin{thm}
\label{thm:domproj2}
Let $V\subseteq\mathbb{C}^{n}\times\mathrm{Y}(1)^n$ be variety such that $V$ projects dominantly both to $\mathbb{C}^{n}$ and $\mathrm{Y}(1)^n$. 
Then MSCD implies that $V$ has (SEC).
\end{thm}
\begin{proof}
The proof below is a small adaptation of the one given in \cite{weakSEC}. 
The domination conditions on $V$ imply that $V$ is free, broad and satisfies (EC) (by Theorem \ref{th:main1}). 
By (the proof of) Proposition \ref{prop:secindimn}, it suffices to consider the case where $\dim V=n$. 

We keep the notation used earlier, so $K$, $\mathbf{t}$ and $F$ are as in \textsection\ref{sec:generalcase}. 
Using the fibre dimension theorem, let $V^\star$ be the Zariski open subset of $V$ such that for all $(\mathbf{x},\mathbf{y})\in V$ we have that 
\[\mathrm{tr.deg.}_{F}F(\mathbf{x},\mathbf{y}) = \mathrm{tr.deg.}_{F}F(\mathbf{y})=\mathrm{tr.deg.}_{F}F(\mathbf{x}).\]
Take $(\mathbf{z},j(\mathbf{z}))\in V_0$ and assume that
\[\mathrm{tr.deg.}_{F}F(\mathbf{z},j(\mathbf{z})) < n.\]
Then by MSCD we know that $\dim_G(\mathbf{z}|\Sigma,\mathbf{t}) < n$, and so we can choose a M\"obius variety $M\subset\mathbb{C}^{n}$ defined over $\overline{\mathbb{Q}(\mathbf{t})}$, and a proper $j(\mathbf{t})$-special subvariety $T$ of $\mathrm{Y}(1)^{n}$ such that $(\mathbf{z},j(\mathbf{z}))\in M\times T$ and $\dim M = \dim T = \dim_G(\mathbf{z}|\Sigma,\mathbf{t})$. 

As usual, consider the definable family $(V_{q})_{q\in Q}$ of Example \ref{ex:parametric family} and let $V_M := V\cap(M\times\mathrm{Y}(1)^{n})$. 
Observe that by MSCD and the double domination assumption we get
\begin{equation}
    \label{eq:doubledom}
    \mathrm{tr.deg.}_{F}F(\mathbf{z},j(\mathbf{z})) = \mathrm{tr.deg.}_{F}F(j(\mathbf{z})) = \mathrm{tr.deg.}_{F}F(j(\mathbf{z})) = \dim_G(\mathbf{z}|\Sigma,\mathbf{t}).
\end{equation}
As $(\mathbf{z},j(\mathbf{z}))\in V_M$, then $\dim V_M\geq \dim M$.

By (\ref{eq:doubledom}), $T$ is in fact the $F$-Zariski closure of $j(\mathbf{z})$. 
So $T$ is contained in the irreducible component of Zariski closure of $\pi_{\mathrm{Y}}(V_M)$ containing $j(\mathbf{z})$.

Similarly, (\ref{eq:doubledom}) shows that $M$ is the $F$-Zariski closure of $\mathbf{z}$. 
So $M$ is contained in the irreducible component of the Zariski closure of $\pi_{\mathbb{C}}(V_M)$ containing $\mathbf{z}$. 
But by definition of $V_M$ we also have that $\pi_{\mathbb{C}}(V_M)$ must be contained in $M$, so $M$ is in fact the Zariski closure of $\pi_{\mathbb{C}}(V_M)$. 
Since $\mathbf{z}$ is generic in $M$ over $F$, it is also generic in $\pi_{\mathbb{C}}(V_M)$, so by the fibre dimension theorem the dimension of $V_{M}\cap\pi_{\mathbb{C}}^{-1}(\mathbf{z})$ must be equal to $\dim V_M - \dim \pi_{\mathbb{C}}^{-1}(V_M)$. 
By hypothesis $\dim \pi_{\mathbb{C}}^{-1}(\mathbf{z})=0$, so $\dim V_M = \pi_{\mathbb{C}}^{-1}(V_M)= \dim M$.

This shows that $\dim\pi_{\mathrm{Y}}(V_M)\leq\dim M= \dim T$, and as we already showed that $T$ is contained in $\pi_{\mathrm{Y}}(V_M)$, then $\dim \pi_{\mathbb{C}}(V_M) = \dim T$, and $T$ must be equal to the irreducible component of Zariski closure of $\pi_{\mathrm{Y}}(V_M)$ containing $j(\mathbf{z})$. 
As $V$ is free, then $\dim V_M <\dim V=n$, so
\[\dim \pi_{\mathrm{Y}}(V_M)\cap T = \dim T > \dim \pi_{\mathrm{Y}}(V_M) + \dim T - n.\]
This shows that $T$ is a $j(\mathbf{t})$-atypical component of $\pi_{\mathrm{Y}}(V_M)\cap T$. 
As $V$ is free and satisfies (EC), to complete the proof it suffices to show that there are only finitely many possible options for $T$.

Consider the parametric family $\left\{U_i\right\}_{i\in I}$ of the Zariski closures of irreducible components of the members of $\left(\pi_{\mathrm{Y}}(V_\mathbf{q})\right)_{\mathbf{q}\in Q}$. 
Let $d$ be a positive integer so that for every $\mathbf{q}\in Q$ which is defined over $\mathbb{Q}(j(\mathbf{t}))$, we have that the irreducible components of $\pi_{\mathrm{Y}}(V_\mathbf{q})$ are defined over a field $L$ satisfying $[L:F]\leq d$. 
Then there is an integer $d'$, which only depends on $d$ and the family $\left\{U_i\right\}_{i\in I}$  with the following property: for every $j(\mathbf{t})$-special subvariety $S$ of $\mathrm{Y}(1)^{n}$, if $S$ equals $U_i$, where $U_i$ is an irreducible component of some $\pi_{\mathrm{Y}}(V_\mathbf{q})$ with $\mathbf{q}\in Q$ definable over $\mathbb{Q}(j(\mathbf{t}))$, then for any value $c$ of a constant coordinate of $S$ we have that
\[[\mathbb{Q}(j(\mathbf{t}),j(\Sigma),c):\mathbb{Q}(j(\mathbf{t}),j(\Sigma))]\leq d'.\]
So by Lemma \ref{lem:t-special} this shows that there are only finitely many possible values that can appear as constant coordinates in $S$. 
Therefore there are only finitely many options for $T$.
\end{proof}

\begin{remark}
We observe that one can also get a notion of ``convenient tuples for $\exp$'' in analogy to our notion of convenient tuples for $j$ (see \cite[\textsection 5.2]{aek-closureoperator}). 
Using this, one can then strengthen the main result of \cite{weakSEC} by removing the requirement of $V$ being definable over $\overline{\mathbb{Q}}$, and instead allowing any finitely generated field of definition like we have done in Theorem \ref{thm:domproj2}.
\end{remark}

In \cite{paper1} it was shown that despite the fact that the composition of $j$ with itself is not defined on all of $\mathbb{H}$, we are still able to find solutions to certain equations involving the iterates of $j$. 
More precisely, given a positive integer $n$ we define inductively 
\[j_{1}:= j\quad \mbox{ and }\quad j_{n+1}:=j\circ j_{n}.\]
The domains of these iterates are defined as
\[\mathbb{H}_{1}:=\mathbb{H}\quad\mbox{ and }\quad \mathbb{H}_{n+1}:=\left\{z\in\mathbb{H}_{n} : j(z)\in\mathbb{H}\right\},\]
so that the natural domain of $j_{n}$ is $\mathbb{H}_{n}$. 

If perhaps unnatural at first, the motivation behind the results of \cite{paper1} concerning iterates of $j$ was to show that there is some ground to stand on if one wants to consider a dynamical system using $j$. 
While there is a lot of work on the dynamical aspects of meromorphic functions on $\mathbb{C}$, the same is not true about holomorphic functions whose natural domain is (under the Riemann mapping theorem) the unit disc. 

The following result is a generalisation of \cite[Theorem 1.4]{paper1}, and in particular, it shows that (under MSCD) for every positive integer $n$ there is $z\in\mathbb{H}_{n}$ satisfying $z = j_n(z)$ and
\[\mathrm{tr.deg.}_{\mathbb{Q}}\mathbb{Q}(z,j(z),j_2(z),\ldots,j_{n-1}(z)) = n.\]

\begin{cor}
\label{cor:iterates1}
Let $Z\subset\mathbb{C}^{n+1}$ be an irreducible hypersurface defined by an irreducible polynomial $p\in\mathbb{C}[X,Y_{1},\ldots,Y_{n}]$ satisfying $\frac{\partial p}{\partial X},\frac{\partial p}{\partial Y_{n}}\neq 0$. 
Then MSCD implies that for every every finitely generated field $K$ over which $Z$ can be defined, there is $z\in\mathbb{H}_{n}$ such that $(z,j(z),j_{2}(z),\ldots,j_{n}(z))$ is generic in $Z$ over $K$. 
\end{cor}
\begin{proof}
As explained in \cite[\textsection 8]{paper1}, finding a point of the form $(z,j(z),j_{2}(z),\ldots,j_{n}(z))$ in $Z$ can be done by finding points in $V\cap \mathrm{E}_{j}^{n}$, where $V\subset\mathbb{C}^{n}\times\mathrm{Y}(1)^{n}$ is defined as
\[V:=\left\{\begin{array}{rcc}
    X_{2} & = & Y_{1} \\
    X_{3} & = & Y_{2}\\
    &\vdots&\\
    X_{n} &=& Y_{n-1}\\
    p(X_{1},\ldots,X_{n},Y_{n}) &=& 0
\end{array}\right\}.\]
Under the conditions $\frac{\partial p}{\partial Y_{n}}\neq 0$ and $\frac{\partial p}{\partial X}\neq 0$ we have that $\dim\pi_{\mathbb{C}}(V) = \dim\pi_{\mathrm{Y}}(V)=n$. 
So by Theorem \ref{thm:domproj2} we get that $V$ has a point $(\mathbf{z},j(\mathbf{z}))$ which is generic over $K$. 
Since $z_{2} = j(z_{1}),\ldots,z_{n}=j(z_{n-1})$, then we also obtain a point in $Z$ of the desired form which is generic over $K$. 
\end{proof}

\section{Unconditional Results}
\label{subsec:unconditional}
One would like to find varieties $V\subseteq\mathbb{C}^{n}\times\mathrm{Y}(1)^{n}$ having generic points in the graph of the $j$-function without having to rely on MSCD or MZP. 
As we mentioned in the introduction, a few unconditional results have already been obtained: see \cite[Theorem 1.1 and \textsection 6.2]{aek-closureoperator}. 

\begin{defn}
A broad algebraic variety $V\subseteq\mathbb{C}^{n}\times\mathrm{Y}(1)^{n}$ is said to have \emph{no $C_j$-factors} if for every choice of indices $1\leq i_1 < \ldots < i_k \leq n$ we have that either $\dim\mathrm{Pr}_{\mathbf{i}} (V) > k$, or $\dim\mathrm{Pr}_{\mathbf{i}} (V)=k$ and $\mathrm{Pr}_{\mathbf{i}} (V)$ is not definable over $C_j$. 

In particular, if $V\subseteq\mathbb{C}^{n}\times\mathrm{Y}(1)^{n}$ is an irreducible variety of dimension $n$ with no $C_j$-factors, then $V$ is not definable over $C_j$.
\end{defn}

The condition of having no $C_{j}$-factors aims at giving a precise notion of what it would mean for a variety to be sufficiently generic with respect to the $j$-function. 
In particular, if $(V_q)_{q\in Q}$ is an algebraic family of free and broad subvarieties of $\mathbb{C}^{n}\times\mathrm{Y}(1)^{n}$ such that for every choice of indices $1\leq i_1 < \ldots < i_k \leq n$ we have that the family of projections $(\mathrm{Pr}_{\mathbf{i}} (V_q))_{q\in Q}$ is either generically of dimension larger than $k$ or non-constant in $q$, then this family contains varieties with no $C_j$-factors (because $C_j$ is countable). 

It should be observed that the condition of having no $C_{j}$-factors is stronger than simply saying that $V$ is not defined over $C_{j}$. 
Indeed, we cannot expect to prove Theorem \ref{thm:unconditional} unconditionally under this weaker assumption. 
For example, suppose that $V_{1},V_{2}\subset\mathbb{C}\times\mathrm{Y}(1)$ are two different free plane curves, with $V_{1}$ defined over $C_{j}$ and $V_{2}$ not definable over $C_{j}$. 
Then the variety $V = V_{1}\times V_{2} \subset\mathbb{C}^{2}\times\mathrm{Y}(1)^{2}$ can be easily checked to be free, broad, and not definable over $C_{j}$. 
If somehow we could prove Theorem \ref{thm:unconditional} for this $V$, then in particular we would have proven Theorem \ref{thm:generic} for $V_{1}$ unconditionally, thus eliminating the need for MSCD.

\subsection{Proof of Theorem \ref{thm:unconditional}}
The proof is inspired by \cite[Proposition 11.5]{bays-kirby}. 
We will set up very similar notation to the one used in \textsection\ref{sec:generalcase}. 
Let $V\subset\mathbb{C}^{n}\times\mathrm{Y}(1)^{n}$ be a broad and free variety. 
Let $K\subset\mathbb{C}$ be a finitely generated subfield such that $V$ is defined over $K$. 
Let $t_{1},\ldots,t_{m}\in\mathbb{H}\setminus C_{j}$ be given by Theorem \ref{thm:mainclosureoperator} applied to $K$. 
We assume that $\dim_{G}(\mathbf{t}|C_{j}) = m$.

Choose $(\mathbf{x},\mathbf{y})\in V$ generic over $C_{j}(\mathbf{t},j(\mathbf{t}))$ (which is possible since $C_{j}(\mathbf{t},j(\mathbf{t}))$ is a countable field). 
By the freeness of $V$ and Lemma \ref{lem:genericpoint} we know that no coordinate of $(\mathbf{x},\mathbf{y})$ is in $C_{j}(\mathbf{t},j(\mathbf{t}))$. 
Let $W\subseteq\mathbb{C}^{n+m}\times\mathrm{Y}(1)^{n+m}$ be the $C_{j}$-Zariski closure of the point $(\mathbf{x},\mathbf{t},\mathbf{y},j(\mathbf{t}))$. 

\begin{lem}
\label{lem:strongWbroadnfree}
$W$ is broad, free and $\dim W = \dim V+m+\dim^{j}(\mathbf{t})$. 
Furthermore, if $\dim^{j}(\mathbf{t})>0$, then $W$ is strongly broad.
\end{lem}
\begin{proof}
The calculation of the dimension of $W$ is done in the same way as in Lemma \ref{lem:dimW}.

As $V$ is free, then the coordinates of $\mathbf{x}$ are all in distinct $G$-orbits and the coordinates of $\mathbf{y}$ are modularly independent. 
On the other hand, as $(\mathbf{x},\mathbf{y})$ is generic over $C_j(\mathbf{t},j(\mathbf{t}))$, then the coordinates of $\mathbf{x}$ define different $G$-orbits than the coordinates of $\mathbf{t}$. 
Similarly, every coordinates of $\mathbf{y}$ is modularly independent from every coordinate of $j(\mathbf{t})$. 
Since $W$ is defined over $C_j$, then $W$ cannot have constant coordinates as every coordinate of $(\mathbf{x},\mathbf{t},\mathbf{y},j(\mathbf{t}))$ is transcendental over $C_j$. 
By construction of $\mathbf{t}$, the coordinates of $\mathbf{t}$ are all in distinct $G$-orbits which implies that the coordinates of $j(\mathbf{t})$ are modularly independent. 
Therefore $W$ is free.

Choose $1\leq i_{1}<\cdots < i_{\ell} \leq n < k_{1}<\cdots <k_{s}\leq n+m$ and set 
\[(\mathbf{i},\mathbf{k})=(i_{1},\ldots,i_{\ell},k_{1},\ldots,k_{s}).\] 
Write
\[\mathrm{Pr}_{(\mathbf{i},\mathbf{k})}(\mathbf{x},\mathbf{t},\mathbf{y},j(\mathbf{t})) = (\mathbf{x}_{\mathbf{i}},\mathbf{t}_{\mathbf{k}},\mathbf{y}_{\mathbf{i}},j(\mathbf{t}_{\mathbf{k}})).\]
Observe that $\mathrm{Pr}_{(\mathbf{i},\mathbf{k})}(W)$ is defined over $C_j$, and since $(\mathbf{x},\mathbf{t},\mathbf{y},j(\mathbf{t}))$ is (by construction) generic in $W$ over $C_j$, then by Proposition \ref{prop:as} we get
\begin{align*}
\dim \mathrm{Pr}_{\mathbf{i}}(W) &= \mathrm{tr.deg.}_{C_j}C_j(\mathbf{x}_{\mathbf{i}},\mathbf{t}_{\mathbf{k}},\mathbf{y}_{\mathbf{i}},j(\mathbf{t}_{\mathbf{k}}))\\
&= \mathrm{tr.deg.}_{C_j}C_j(\mathbf{t}_{\mathbf{k}},j(\mathbf{t}_{\mathbf{k}})) + \mathrm{tr.deg.}_{C_j(\mathbf{t}_{\mathbf{k}},j(\mathbf{t}_{\mathbf{k}}))}C_j(\mathbf{x}_{\mathbf{i}},\mathbf{t}_{\mathbf{k}},\mathbf{y}_{\mathbf{i}},j(\mathbf{t}_{\mathbf{k}}))\\
&\geq \dim_{G}(\mathbf{t}_{\mathbf{k}}|C_j) + \dim^{j}(\mathbf{t}) + \dim \mathrm{Pr}_{\mathbf{i}}(V)\\
&\geq s+\ell.
\end{align*}
Therefore $W$ is broad. 
From the last inequality we also get that if $\dim^{j}(\mathbf{t})>0$, then $W$ is strongly broad.
\end{proof}

\begin{proof}[Proof of Theorem \ref{thm:unconditional}]
By Proposition \ref{prop:secindimn} we will assume that $\dim V=n$. 
We then know that $V$ is not definable over $C_{j}$, so $V(C_{j})$ is contained in a proper subvariety of $V$. 
In particular, there is a Zariski open subset $V_{0}\subseteq V$ such that if $(\mathbf{z},j(\mathbf{z}))\in V_{0}\cap \mathrm{E}_{j}^{n}$, then some of the coordinates of $j(\mathbf{z})$ are not in $C_{j}$. 
Also, $K\not\subset C_{j}$, so $\dim^{j}(\mathbf{t})\geq 1$.

The case $n=1$ was proven in \cite[Theorem 1.1]{aek-closureoperator}, so now we assume that $n>1$. 
Consider the parametric family of subvarieties $(W_\mathbf{q})_{\mathbf{q}\in Q}$ of $W$ defined in Example \ref{ex:parametric family} and let $\mathscr{S}$ be the finite collection  of special subvarieties of $\mathrm{Y}(1)^{n+m}$ given by Theorem \ref{thm:horizontalweakzp} applied to this family. 
Let $N$ be the integer given by Proposition \ref{prop:horizontalboundedcomplexity} applied to  $(W_{\mathbf{q}})_{\mathbf{q}\in Q}$. 

Let $W_{0}\subseteq W$ be a Zariski open subset defined over $C_j$ such that the points $(\mathbf{a},\mathbf{b})$ of $W_{0}$ satisfy all of the following conditions:
\begin{enumerate}[(a)]
    \item The point $\mathbf{b}$ does not lie in any $T\in\mathscr{S}$. 
    As $W$ is free, this condition defines a Zariski open subset of $W$.
    \item The coordinate of $\mathbf{b}$ do not satisfy any of the modular relations $\Phi_{1},\ldots,\Phi_{N}$. 
    As $W$ is free, this condition defines a Zariski open subset of $W$.
    \item For every $1\leq i_{1}<\cdots<i_{\ell}\leq n$, and letting $\mathbf{i}=(i_{1},\ldots,i_{\ell})$, we have that
    \[\dim \left(W\cap \mathrm{Pr}_{\mathbf{i}}^{-1}((\mathbf{a},\mathbf{b}))\right) = \dim W - \dim \mathrm{Pr}_{\mathbf{i}}(W).\]
    By the fibre-dimension theorem and the fact there are only finitely many tuples $\mathbf{i}$ to consider, this defines a Zariski open subset of $W$. 
\end{enumerate}
By the construction of $W$, there is a Zariski opens subset $V_{1}\subseteq V$ such that if $(\mathbf{x},\mathbf{y})$ is any point of $V_{1}$, then $(\mathbf{x},\mathbf{t},\mathbf{y},j(\mathbf{t}))$ is a point in $W_{0}$. 

Choose $(\mathbf{z},j(\mathbf{z}))\in V_{0}\cap V_{1}$, so that $(\mathbf{z},\mathbf{t},j(\mathbf{z}),j(\mathbf{t}))\in W_{0}$. 
We will show that $\dim_{G}(\mathbf{z},\mathbf{t}|C_{j})=n+m$. 
For this we proceed by contradiction, so suppose that $\dim_{G}(\mathbf{z},\mathbf{t}|C_{j})<n+m$. 
Let $T$ be the weakly special subvariety of $\mathrm{Y}(1)^{n+m}$ of minimal dimension defined over $C_{j}$ for which $(j(\mathbf{z}),j(\mathbf{t}))\in T$. 
Observe that $\dim(T) = \dim_{G}(\mathbf{z},\mathbf{t}|C_{j})$. 

Let $M\subset\mathbb{C}^{n+m}$ be a subvariety of minimal dimension defined by M\"obius relations defined over $\mathbb{Q}$ and/or setting some coordinates to be a constant in $C_{j}$ satisfying $(\mathbf{z},\mathbf{t})\in M$. 
Then $W\cap (M\times\mathrm{Y}(1)^{n+m})$ is an element of the family $(W_{\mathbf{q}})_{\mathbf{q}\in Q}$, call it $W_{M}$. 
We remark that $W_{M}$ is defined over $C_{j}$. 

Now let $X$ be the irreducible component of $W_{M}\cap(\mathbb{C}^{n+m}\times T)$ containing $(\mathbf{z},\mathbf{t}, j(\mathbf{z}),j(\mathbf{t}))$. 
Observe that $X$ is defined over $C_{j}$. 
Then by Proposition \ref{prop:as} we get
\begin{equation}
\label{eq:dimXunconditional}
\dim X \geq \mathrm{tr.deg.}_{C_{j}}C_{j}(\mathbf{z},\mathbf{t}, j(\mathbf{z}),j(\mathbf{t})) \geq \dim_{G}(\mathbf{z}
,\mathbf{t}|C_{j}) + \dim ^{j}(\mathbf{z}
\cup\mathbf{t}).
\end{equation}
On the other hand, as $W$ is free, then $\dim W_{M} <\dim W$, so using Lemma \ref{lem:strongWbroadnfree} we have
\begin{equation}
\label{eq:atypicalXunconditional}
\begin{array}{ccl}
\dim W_{M} + \dim\mathbb{C}^{n+m}\times T - \dim\mathbb{C}^{n+m}\times\mathrm{Y}(1)^{n+m} &<& \dim W+\dim T-n-m\\
&=& \dim^{j}(\mathbf{t})+\dim T\\
&=& \dim^{j}(\mathbf{t}) + \dim_{G}(\mathbf{z},\mathbf{t}|C_{j})\\
&\leq& \dim X.
\end{array}    
\end{equation}
This shows that $X$ is  an atypical component of $W_{M}\cap(\mathbb{C}^{n+m}\times T)$ in $\mathbb{C}^{n+m}\times\mathrm{Y}(1)^{n+m}$. 
If $\pi_{\mathrm{Y}}(X)$ has no constant coordinates, then there exists $T_{0}\in\mathscr{S}$ such that $X\subset W_{M}\cap(\mathbb{C}^{n+m}\times T_{0})$. 
However, this would contradict condition (a) in the definition of $W_{0}$.

So $\pi_{\mathrm{Y}}(X)$ has some constant coordinates. 
Then, as $X$ is defined over $C_{j}$, those constant coordinates must be given by elements of $C_{j}$. 
Since no element of $j(\mathbf{t})$ is in $C_{j}$, the constant coordinates of $\pi_{\mathrm{Y}}(X)$ must be found among the coordinates of $j(\mathbf{z})$. 
Let $1\leq i_{1}<\cdots<i_{\ell}\leq n$ denote all the coordinates of $j(\mathbf{z})$ which are in $C_{j}$. 
Since $(\mathbf{z},j(\mathbf{z}))\in V_{0}$, then $\ell<n$. 

By Proposition \ref{prop:horizontalboundedcomplexity} we know that there is a weakly special subvariety $T_{0}
\subset\mathrm{Y}(1)^{n+m}$ such that $\Delta_b(T_{0})\leq N$, $X\subseteq\mathbb{C}^{n+m}\times T_{0}$, and
\begin{equation}
\label{eq:typicalXunconditional}
    \dim X\leq \dim W_{M}\cap(\mathbb{C}^{n+m}\times T_0) + \dim T\cap T_0 - \dim T_0.
\end{equation}
By condition (b) in the definition of $W_0$ we know that $\Delta(T_0)=0$, which means that $T_0$ is completely defined by setting certain coordinates to be constant. 
As the constant coordinates of $\pi_Y(X)$ are in $C_j$, then $T_0$ is defined over $C_j$, and $\dim T_0\geq n+m-\ell$. 
But $T$ is, by definition, the smallest weakly special subvariety of $\mathrm{Y}(1)^{n+m}$ which is defined over $C_j$ and contains the point $(j(\mathbf{z}),j(\mathbf{t}))$. 
So $T\cap T_0 = T$. 
Combining (\ref{eq:dimXunconditional}) and (\ref{eq:typicalXunconditional}) we get
\begin{equation}
\label{eq:fibredimlowunconditional}
    \dim^{j}(\mathbf{t}) + \dim T_0 \leq \dim W_{M}\cap(\mathbb{C}^{n+m}\times T_0).
\end{equation}

Set $\Theta:=W\cap \mathrm{Pr}_{\mathbf{i}}^{-1}\left(\mathrm{Pr}_{\mathbf{i}}(\mathbf{z},\mathbf{t},j(\mathbf{z},\mathbf{t}))\right)$. 
By the fibre dimension theorem, condition (c) of the definition of $W_{0}$ and the fact that $W$ is strongly broad (Lemma \ref{lem:strongWbroadnfree}) we know that 
\begin{equation}
\label{eq:fibredimupunconditional}
    \dim \Theta\leq  \dim W - \dim \mathrm{Pr}_{\mathbf{i}}(W) < n+m+\dim^{j}(\mathbf{t})-\ell.
\end{equation}
Observe that $W_{M}\cap(\mathbb{C}^{n+m}\times T_0) = W
\cap(M\times T_0)\subseteq\Theta$, so combining (\ref{eq:fibredimlowunconditional}) and (\ref{eq:fibredimupunconditional}) gives
\[\dim T_0 < n+m-\ell\]
which is a contradiction.

We deduce from this that $\dim_{G}(\mathbf{z},\mathbf{t}|C_{j})=n+m$. 
By \cite[Lemma 5.2]{aek-closureoperator} this implies that $(\mathbf{z},j(\mathbf{z}))$ is generic in $V$ over $C_{j}(\mathbf{t},j(\mathbf{t}))$. 
\end{proof}

With this we can get the following ``generic'' version of Corollary \ref{cor:iterates1}.

\begin{cor}
\label{cor:iterates2}
Let $Z\subset\mathbb{C}^{n+1}$ be an irreducible hypersurface defined by an irreducible polynomial $p\in\mathbb{C}[X,Y_{1},\ldots,Y_{n}]$ satisfying $\frac{\partial p}{\partial X},\frac{\partial p}{\partial Y_{n}}\neq 0$. Suppose that $Z$ is not definable over $C_{j}$. 
Then for every every finitely generated field $K$ over which $Z$ can be defined, there is $z\in\mathbb{H}_{n}$ such that $(z,j(z),j_{2}(z),\ldots,j_{n}(z))$ is generic in $Z$ over $K$. 
\end{cor}
\begin{proof}
Proceed just like in the proof of Corollary \ref{cor:iterates1}. 
As we are assuming that $Z$ is not definable over $C_{j}$, this will imply that the corresponding variety $V$ has no $C_{j}$-factors, and so Theorem \ref{thm:unconditional} applies. 
\end{proof}

\subsection{Proof of Theorem \ref{thm:blurred}}
\label{subsec:thmblurred}
The proof of Theorem \ref{thm:blurred} is done by a straightforward repetition of the proof of Theorem \ref{thm:unconditional}, and taking into consideration the comments in \textsection\ref{subsec:blurring}, which manifest in the proof of Theorem \ref{thm:blurred} as follows.
\begin{enumerate}[(a)]
\item Proposition \ref{prop:as} takes the following form: for every $g_{1},\ldots,g_{n}\in G$ and every $z_{1},\ldots,z_{n}\in\mathbb{H}^{n}$ we have
\[\mathrm{tr.deg.}_{C_j}C_j(z_{1},\ldots,z_{n},j(g_{1}z_{1}),\ldots,j(g_{n}z_{n}))\geq\dim_{G}(z_{1},\ldots,z_{n}|C_j) + \dim^j(z_1,\ldots,z_n).\]
\item We can use \cite[Corollary 5.4]{aek-closureoperator} to obtain the following: if $t_{1},\ldots,t_{m}\in\mathbb{H}\setminus C_j$ satisfies 
\[\mathrm{tr.deg.}_{C_j}C_j(\mathbf{t},J(\mathbf{t})) = 3\dim_G(\mathbf{t}|C_j) + \dim^j(\mathbf{t})\]
(the existence of such tuples is guaranteed by Theorem \ref{thm:mainclosureoperator}), then setting $F:=C_j(\mathbf{t},j(\mathbf{t}))$ we have that for all $z_{1},\ldots,z_{n}\in\mathbb{H}$ and all $g_{1},\ldots,g_{n}\in G$:
\[\mathrm{tr.deg.}_{F}F(z_1,\ldots,z_n,j(g_{1}z_{1}),\ldots,j(g_{n}z_{n}))\geq \dim_{G}(\mathbf{z}|C_j\cup\mathbf{t}) + \dim^{j}(\mathbf{z}|\mathbf{t}).\]
\item Theorem \ref{thm:horizontalweakzp} still applies as stated. 
\end{enumerate}

\section{Results With Derivatives}
\label{sec:derivatives}
As explained in \cite{aslanyan-kirby} and \cite{vahagn} (among other sources), in conjunction with the EC problem for $j$, one should also consider the EC problem for $j$ and its derivatives. 
In this section we will explain how the methods we have used can be adapted to study the strong EC problem for $j$ and its derivatives.

\subsection{Definitions}

We start by setting up some notation. 
Define $\mathrm{Y}_{2}(1) := \mathrm{Y}(1)\times\mathbb{C}^{2}$. 
Let $\mathrm{E}_{J}^{n}:=\left\{(\mathbf{z},J(\mathbf{z})) : \mathbf{z}\in\mathbb{H}^{n}\right\}\subseteq\mathbb{C}^{n}\times\mathrm{Y}_{2}(1)^{n}$. 

Let $n,\ell$ be positive integers with $\ell \leq n$ and $\mathbf{i}=(i_1,\ldots,i_\ell)$ in $\mathbb{N}^{\ell}$ with $1\leq i_1 < \ldots < i_\ell \leq n$. 
Define $\mathrm{PR}_{\mathbf{i}}:\mathbb{C}^{n}\times\mathrm{Y}_{2}(1)^{n}\rightarrow \mathbb{C}^{\ell}\times\mathrm{Y}_{2}(1)^{\ell}$ by
\[\mathrm{PR}_{\mathbf{i}}:(\mathbf{x},\mathbf{y}_{0},\mathbf{y}_{1},\mathbf{y}_{2})\mapsto (\mathrm{pr}_{\mathbf{i}}(\mathbf{x}),\mathrm{pr}_{\mathbf{i}}(\mathbf{y}_{0}), \mathrm{pr}_{\mathbf{i}}(\mathbf{y}_{1}),\mathrm{pr}_{\mathbf{i}}(\mathbf{y}_{2})).\]
We will abuse slightly some notation we have already introduced define the maps $\pi_{\mathbb{C}}:\mathbb{C}^{n}\times\mathrm{Y}_{2}(1)^{n}\to\mathbb{C}^{n}$ and $\pi_{\mathrm{Y}}:\mathbb{C}^{n}\times\mathrm{Y}_{2}(1)^{n}\to\mathrm{Y}(1)^{n}$ as the coordinate projections. 
Notice that $\pi_{\mathrm{Y}}$ still maps onto $\mathrm{Y}(1)^{n}$, not to $\mathrm{Y}_2(1)^{n}$.
These projections will be used in a very similar way as to how $\pi_\mathbb{C}$ and $\pi_\mathrm{Y}$ have been used in the previous sections, which is why we have decided to keep the names. 

\begin{defn}
An algebraic set $V \subseteq \mathbb{C}^{n}\times\mathrm{Y}_{2}(1)^{n}$ is said to be $J$-\emph{broad} if for any $\mathbf{i}=(i_1,\ldots,i_{\ell})$ in $\mathbb{N}^{\ell}$ with $1\leq i_1 < \ldots < i_{\ell} \leq n$ we have $\dim \mathrm{PR}_{\mathbf{i}} (V) \geq 3\ell$. 
In particular, if $V$ is $J$-broad then $\dim V\geq 3n$. 

We say $V$ is \emph{strongly $J$-broad} if the strict inequality $\dim \mathrm{PR}_{\mathbf{i}} (V) > 3\ell$ holds for every $\mathbf{i}$.
\end{defn}

\begin{defn}
A subvariety $T\subseteq\mathrm{Y}_{2}(1)^{n}$ is called a \emph{special subvariety} of $\mathrm{Y}_{2}(1)^{n}$ if there is a M\"obius subvariety $M\subseteq\mathbb{C}^{n}$ defined over $\mathbb{Q}$ such that $T$ is the Zariski closure over $\overline{\mathbb{Q}}$ of the set $J(M\cap\mathbb{H}^{n})$. 
We will say that $T$ is \emph{weakly special} if there is a M\"obius subvariety $M\subseteq\mathbb{C}^{n}$ such that $T$ is the Zariski closure over $\mathbb{C}$ of the set $J(M\cap\mathbb{H}^{n})$.
\end{defn} 

\begin{defn}
We will say that an irreducible constructible  set $V\subseteq \mathbb{C}^{n}\times\mathrm{Y}_{2}(1)^{n}$ is $J$-\emph{free} if no coordinate of $V$ is constant, and $V$ is not contained in any subvariety of the form $M\times\mathrm{Y}_{2}(1)^{n}$ or $\mathbb{C}^{n}\times T$, where $M\subset\mathbb{C}^{n}$ is a proper M\"obius subvariety defined over $\mathbb{Q}$, and $T\subset\mathrm{Y}(1)^{n}$ is a proper special subvariety. 

A constructible subset $V\subseteq \mathbb{C}^{n}\times\mathrm{Y}(1)^{n}$ is $J$-\emph{free} if every irreducible component of $V$ is free.
\end{defn}

\begin{defn}
We say that an algebraic variety $V\subseteq \mathbb{C}^{n}\times\mathrm{Y}_{2}(1)^{n}$ satisfies the \emph{Existential Closedness condition for $J$}, or $\mathrm{(EC)}_{J}$, if the set $V\cap \mathrm{E}_{J}^{n}$ is Zariski dense in $V$. 

We say that $V$ satisfies the \emph{Strong Existential Closedness condition for $J$}, or $\mathrm{(SEC)}_{J}$, if for every finitely generated field $K\subset\mathbb{C}$ over which $V$ can be defined, there exists $(\mathbf{z},J(\mathbf{z}))\in V$ such that $(\mathbf{z},J(\mathbf{z}))$ is generic in $V$ over $K$. 
\end{defn}

\begin{conj}
\label{conj:ecwithderivatives}
For every positive integer $n$, every algebraic variety $V\subseteq \mathbb{C}^{n}\times\mathrm{Y}_{2}(1)^{n}$ which is $J$-broad and $J$-free, if $V\cap\left(\mathbb{H}^n\times\mathrm{Y}_2(1)^n\right)$ is Zariski dense in $V$, then $V$ satisfies $\mathrm{(EC)}_{J}$. 
Even more, such $V$ satisfy $\mathrm{(SEC)}_{J}$.
\end{conj}

As explained in \cite[\textsection 5.2]{vahagn2}, Conjecture \ref{conj:ecwithderivatives} can be reduced to the case where $\dim V=3n$ by doing an obvious adaptation of Proposition \ref{prop:secindimn}. 

Before continuing to the results with derivatives, in the next few sections we will go over the key ingredients that we need.

\subsection{Convenient tuples for \texorpdfstring{$J$}{J}} 

We could start this section by giving a natural definition of \emph{convenient generators for $J$}, following what we did in \textsection\ref{subsec:transineq}. 
However this definition would be exactly the same as the definition of convenient generators for $j$. 
To see this we recall that \cite[Theorem 1.2]{aek-differentialEC} and the results in \cite[\textsection 5]{aek-closureoperator} already include the derivatives of $j$. 
This is manifested in the fact that the various inequalities for $j$ we showed in \textsection\ref{subsec:transineq} were all proven by first proving the statement with derivatives. 
So we already have all the transcendence inequalities we need.

\subsection{Weak Zilber--Pink with derivatives} 
Here we recall one of the main results of \cite{vahagn}. 

\begin{defn}
Given a special subvariety $S$ of $\mathrm{Y}_{2}(1)^{n}$ and a special subvariety $T$ of $\mathrm{Y}(1)^{n}$, we say that $S$ \emph{is associated with} $T$ if $\pi_{\mathrm{Y}}(S) = T$.
\end{defn}

\begin{defn}
Let $V$ be an algebraic subvariety of $\mathrm{Y}_{2}(1)^{n}$. 
An \emph{atypical component} of $V$ is an irreducible component $X$ of the intersection between $V$ and a special subvariety $T$ of $\mathrm{Y}_{2}(1)^{n}$ such that
\[\dim X > \dim V + \dim T - 3n.\]
Furthermore, we say that $X$ is a \emph{strongly atypical component} of $V$ if $X$ is an atypical component of $V$ and no coordinate is constant on $\pi_{\mathrm{Y}}(X)$. 
\end{defn}

For the definition of \emph{upper triangular $D$-special subvariety} used in the following theorem, see \cite[\textsection 6.1]{vahagn}. 
In particular, we can choose $S = \mathrm{Y}_{2}(1)^{n}$.

\begin{thm}[Uniform weak MZP with derivatives, see {{\cite[Theorem 7.9]{vahagn}}}]
\label{thm:weakMZPD}
Let $S$ be an upper-triangular $D$-special subvariety of $\mathrm{Y}_{2}(1)^{n}$. Given a parametric family $(V_{\mathbf{q}})_{\mathbf{q}\in Q}$ of algebraic subvarieties of $\mathrm{Y}_{2}(1)^{n}$, there is a finite collection $\mathscr{S}$ of proper special subvarieties of $\mathrm{Y}(1)^{n}$ such that for every $\mathbf{q}\in Q$ we have that for every strongly atypical component $X$ of $V$ there is a special subvariety $S$ of $\mathrm{Y}_{2}(1)^{n}$ such that $X\subseteq S$ and $\pi_{\mathrm{Y}}(S)= T$ for some $T\in\mathscr{S}$. 
\end{thm}

\begin{thm}
\label{thm:horizontalweakMZPD}
Given a parametric family $(U_{\mathbf{q}})_{\mathbf{q}\in Q}$ of constructible subsets of $\mathbb{C}^{n}\times\mathrm{Y}_{2}(1)^{n}$, there is a finite collection $\mathscr{S}$ of of proper special subvarieties of $\mathrm{Y}(1)^{n}$ such that for every $\mathbf{q}\in Q$ we have that for every strongly atypical component $X$ of $V_\mathbf{q}$ there is a special subvariety $S$ of $\mathrm{Y}_{2}(1)^{n}$ such that $X\subseteq\mathbb{C}^{n}\times S$ and $\pi_{\mathrm{Y}}(S)= T$ for some $T\in\mathscr{S}$. 
\end{thm}

We can now go through the same sequence of steps as in \textsection\ref{subsec:weakmzp} to obtain the following analogue of Proposition \ref{prop:horizontalboundedcomplexity}. 

\begin{prop}
\label{prop:Jhorizontalboundedcomplexity}
Given a parametric family $(U_{\mathbf{q}})_{\mathbf{q}\in Q}$ of constructible subsets of $\mathbb{C}^{n}\times\mathrm{Y}_{2}(1)^{n}$, there is a positive integer $N$ such that for every $\mathbf{q}\in Q$, for every weakly special subvariety $S\subset\mathrm{Y}_{2}(1)^{n}$ and for every atypical component $X$ of $U_{\mathbf{q}}\cap (\mathbb{C}^{n}\times S)$, there is a proper weakly special subvariety $S_{0}\subset\mathrm{Y}_{2}(1)^{n}$ with $\Delta_b(\pi_{\mathrm{Y}}(S_{0}))\leq N$ such that $X\subseteq \mathbb{C}^{n}\times S_{0}$. 
\end{prop}

\subsection{Main Result}

Here we present and prove an analogue of Theorem \ref{thm:unconditional} which includes derivatives. 

\begin{defn}
Let $L$ be an algebraically closed subfield of $\mathbb{C}$. A $J$-broad algebraic variety $V\subseteq\mathbb{C}^{n}\times\mathrm{Y}_{2}(1)^{n}$ is said to have \emph{no $L$-factors} if for every choice of indices $1\leq i_1 < \ldots < i_k \leq n$ we have that either $\dim\mathrm{PR}_{\mathbf{i}} (V) > 3k$, or $\dim\mathrm{PR}_{\mathbf{i}} (V)=3k$ and $\mathrm{PR}_{\mathbf{i}} (V)$ is not definable over $L$. 
\end{defn}

\begin{thm}
\label{thm:unconditionalwithderivatives}
Let $V\subseteq\mathbb{C}^{n}\times\mathrm{Y}_{2}(1)^{n}$ be a $J$-broad and $J$-free variety with no $C_{j}$-factors satisfying $\mathrm{(EC)}_{J}$. 
Then for every finitely generated field $K$ over which $V$ can be defined, there exists $\mathbf{z}\in\mathbb{H}^{n}$ such that $(\mathbf{z},J(\mathbf{z}))\in V$ is generic over $K$. 
\end{thm}

We will set up very similar notation to the one used in \textsection\ref{sec:generalcase}. 
Let $V\subset\mathbb{C}^{n}\times\mathrm{Y}_{2}(1)^{n}$ be a $J$-broad and $J$-free variety with no $C_{j}$-factors. 
Let $K\subset\mathbb{C}$ be a finitely generated subfield such that $V$ is defined over $K$. 
Let $t_{1},\ldots,t_{m}\in\mathbb{H}\setminus C_{j}$ be given by Theorem \ref{thm:mainclosureoperator} applied to $K$. 
By Lemma \ref{lem:c}, we also assume that $\dim_{G}(\mathbf{t}|C_j) = m$. 

Choose $(\mathbf{x},\mathbf{y}_{0},\mathbf{y}_{1},\mathbf{y}_{2})\in V$ generic over $C_{j}(\mathbf{t},J(\mathbf{t}))$. 
Let $W\subseteq\mathbb{C}^{n+m}\times\mathrm{Y}_{2}(1)^{n+m}$ be the $C_{j}$-Zariski closure of the point $(\mathbf{x},\mathbf{t},\mathbf{y}_{0},\mathbf{y}_{1},\mathbf{y}_{2},J(\mathbf{t}))$. 

\begin{lem}
\label{lem:JstrongWbroadnfree}
$W$ is $J$-broad, $J$-free and $\dim W = 3\dim V+3m+\dim^{j}(\mathbf{t})$. 
Furthermore, if $\dim^{j}(\mathbf{t})>0$, then $W$ is strongly $J$-broad.
\end{lem}
\begin{proof}
Repeat the proof of Lemma \ref{lem:strongWbroadnfree}. 
\end{proof}

\begin{proof}[Proof of Theorem \ref{thm:unconditionalwithderivatives}]
We assume that $\dim V=3n$. 
We then know that $V$ is not definable over $C_{j}$, so $V(C_{j})$ is contained in a proper subvariety of $V$. 
In particular, there is a Zariski open subset $V_{0}\subseteq V$ such that if $(\mathbf{z},j(\mathbf{z}))\in V_{0}$, then some of the coordinates of $j(\mathbf{z})$ are not in $C_{j}$. 
Also, in this case $K\not\subset C_{j}$ so $\dim^{j}(\mathbf{t})\geq 1$.

Consider the parametric family of subvarieties $(W_\mathbf{q})_{\mathbf{q}\in Q}$ of $W$ such that for every $\mathbf{q}\in Q$ there is a M\"obius subvariety $M_\mathbf{q}\subseteq\mathbb{C}^{n}$ such that $W_\mathbf{q}:=W\cap (M_\mathbf{q}\times\mathrm{Y}_2(1)^{2})$. 
Let $\mathscr{S}$ be the finite collection  of special subvarieties of $\mathrm{Y}_2(1)^{n+m}$ given by Theorem \ref{thm:horizontalweakMZPD}. 
Let $N$ be given by Proposition \ref{prop:Jhorizontalboundedcomplexity} applied to  $(W_{\mathbf{q}})_{\mathbf{q}\in Q}$. 

Let $W_{0}\subseteq W$ be a Zariski open subset defined over $C_j$ such that the points $(\mathbf{a},\mathbf{b}_0,\mathbf{b}_1,\mathbf{b}_2)$ of $W_{0}$ satisfy all of the following conditions:
\begin{enumerate}[(a)]
    \item The point $\mathbf{b}_0$ does not lie in any $T\in\mathscr{S}$. As $W$ is free, this condition defines a Zariski open subset of $W$.
    \item The coordinate of $\mathbf{b}_0$ do not satisfy any of the modular relations $\Phi_{1},\ldots,\Phi_{N}$. 
    As $W$ is free, this condition defines a Zariski open subset of $W$.
    \item For every $1\leq i_{1}<\cdots<i_{\ell}\leq n$, and letting $\mathbf{i}=(i_{1},\ldots,i_{\ell})$, we have that
    \[\dim \left(W\cap \mathrm{PR}_{\mathbf{i}}^{-1}((\mathbf{a},\mathbf{b}_0,\mathbf{b}_1,\mathbf{b}_2))\right) = \dim W - \dim \mathrm{PR}_{\mathbf{i}}(W).\]
    By the fibre-dimension theorem and the fact there are only finitely many tuples $\mathbf{i}$ to consider, this defines a Zariski open subset of $W$. 
\end{enumerate}
By the construction of $W$, there is a Zariski open subset $V_{1}\subseteq V$ such that if $(\mathbf{x},\mathbf{y}_{0},\mathbf{y}_{1},\mathbf{y}_{2})$ is any point of $V_{1}$, then $(\mathbf{x},\mathbf{t},\mathbf{y}_{0},\mathbf{y}_{1},\mathbf{y}_{2},J(\mathbf{t}))$ is a point in $W_{0}$. 

Choose $(\mathbf{z},J(\mathbf{z}))\in V_{0}\cap V_{1}$, so that $(\mathbf{z},\mathbf{t},J(\mathbf{z}),J(\mathbf{t}))\in W_{0}$. 
We will show that $\dim_{G}(\mathbf{z},\mathbf{t}|C_{j})=n+m$. 
For this we proceed by contradiction, so suppose that $\dim_{G}(\mathbf{z},\mathbf{t}|C_{j})<n+m$. 
Let $S$ be the weakly special subvariety of $\mathrm{Y}_{2}(1)^{n+m}$ of minimal dimension defined over $C_{j}$ for which $(J(\mathbf{z}),J(\mathbf{t}))\in T$. 
Observe that $\dim(T) = 3\dim_{G}(\mathbf{z},\mathbf{t}|C_{j})$. 

Let $M\subset\mathbb{C}^{n+m}$ be the M\"obius subvariety of smallest dimension defined by M\"obius relation over $\mathbb{Q}$ and/or setting some coordinates to be a constant in $C_{j}$, which satisfies $(\mathbf{z},\mathbf{t})\in M$. 
Then $W\cap (M\times\mathrm{Y}_{2}(1)^{n+m})$ is an element of the family $(W_{\mathbf{q}})_{\mathbf{q}\in Q}$, call it $W_{M}$. 
We remark that $W_{M}$ is defined over $C_{j}$. 

Now let $X$ be the irreducible component of $W_{M}\cap(\mathbb{C}^{n+m}\times S)$ containing $(\mathbf{z},\mathbf{t}, J(\mathbf{z}),J(\mathbf{t}))$. 
Observe that $X$ is defined over $C_{j}$. 
Then by Proposition \ref{prop:as} we get
\begin{equation}
\label{eq:JdimXunconditional}
\dim X \geq \mathrm{tr.deg.}_{C_{j}}C_{j}(\mathbf{z},\mathbf{t}, J(\mathbf{z}),J(\mathbf{t})) \geq 3\dim_{G}(\mathbf{z},\mathbf{t}|C_{j}) + \dim ^{j}(\mathbf{z},\mathbf{t}).
\end{equation}
On the other hand, as $W$ is free, then $\dim W_{M} <\dim W$, so using Lemma \ref{lem:strongWbroadnfree} we have
\begin{equation}
\label{eq:JatypicalXunconditional}
\begin{array}{ccl}
\dim W_{M} + \dim\mathbb{C}^{n+m}\times S - \dim\mathbb{C}^{n+m}\times\mathrm{Y}_{2}(1)^{n+m} &<& \dim W+\dim T-n-m\\
&=& \dim^{j}(\mathbf{t})+\dim S\\
&=& \dim^{j}(\mathbf{t}) + 3\dim_{G}(\mathbf{z},\mathbf{t}|C_{j})\\
&\leq& \dim X.
\end{array}    
\end{equation}
This shows that $X$ is  an atypical component of $W_{M}\cap(\mathbb{C}^{n+m}\times S)$ in $\mathbb{C}^{n+m}\times\mathrm{Y}_{2}(1)^{n+m}$. 
If $\pi_{\mathrm{Y}}(X)$ has no constant coordinates, then by \ref{thm:horizontalweakMZPD} there exists a special subvariety $S\subseteq\mathrm{Y}_2(1)^{n+m}$ and $T_{0}\in\mathscr{S}$ such that $X\subset W_{M}\cap(\mathbb{C}^{n+m}\times S)$ and $\pi_{\mathrm{Y}}(S)= T_0$. 
However, this would contradict condition (a) in the definition of $W_{0}$.

So $\pi_{\mathrm{Y}}(X)$ has some constant coordinates. Then, as $X$ is defined over $C_{j}$, those constant coordinates must be given by elements of $C_{j}$. 
Since no element of $J(\mathbf{t})$ is in $C_{j}$ (by Lemma \ref{lem:c}), the constant coordinates of $\pi_{\mathrm{Y}}(X)$ must be found among the coordinates of $J(\mathbf{z})$. 
Let $1\leq i_{1}<\cdots<i_{\ell}\leq n$ denote all the coordinates of $j(\mathbf{z})$ which are in $C_{j}$. 
Recall that by Lemma \ref{lem:c}, if some coordinate of $(z,j(z),j'(z),j''(z))$ is in $C_j$, then they all are. 
Since $(\mathbf{z},J(\mathbf{z}))\in V_{0}$, then $\ell<n$. 
We also remark that at this point we have already proven the theorem for the case $n=1$.

By Proposition \ref{prop:horizontalboundedcomplexity} we know that there is a weakly special subvariety $S_{0}
\subset\mathrm{Y}_2(1)^{n+m}$ such that $\Delta_b(\pi_{\mathrm{Y}}(S_{0}))\leq N$, $X\subseteq\mathbb{C}^{n+m}\times S_{0}$, and
\begin{equation}
\label{eq:JtypicalXunconditional}
    \dim X\leq \dim W_{M}\cap(\mathbb{C}^{n+m}\times S_0) + \dim T\cap S_0 - \dim S_0.
\end{equation}
By condition (b) in the definition of $W_0$ we know that $\Delta_b(\pi_{\mathrm{Y}}(S_0))=0$, which means that $\pi_{\mathrm{Y}}(S_0)$ is completely defined by setting certain coordinates to be constant. 
As the constant coordinates of $\pi_\mathrm{Y}(X)$ are in $C_j$, then $S_0$ is defined over $C_j$, and $\dim S_0\geq 3(n+m-\ell)$ (again by Lemma \ref{lem:c}). 
But $S$ is, by definition, the smallest weakly special subvariety of $\mathrm{Y}_{2}(1)^{n+m}$ which is defined over $C_j$ and contains the point $(J(\mathbf{z}),J(\mathbf{t}))$. 
So $S\cap S_0 = S$. 
Combining (\ref{eq:JdimXunconditional}) and (\ref{eq:JtypicalXunconditional}) we get
\begin{equation}
\label{eq:Jfibredimlowunconditional}
    \dim^{j}(\mathbf{t}) + \dim S_0 \leq \dim W_{M}\cap(\mathbb{C}^{n+m}\times S_0).
\end{equation}

Set $\Theta:=W\cap \mathrm{PR}_{\mathbf{i}}^{-1}\left(\mathrm{PR}_{\mathbf{i}}(\mathbf{z},\mathbf{t},j(\mathbf{z},\mathbf{t}))\right)$. 
By the fibre dimension theorem, condition (c) of the definition of $W_{0}$ and the fact that $W$ is strongly broad (Lemma \ref{lem:JstrongWbroadnfree}) we know that
\begin{equation}
\label{eq:Jfibredimupunconditional}
    \dim \Theta\leq  \dim W - \dim \mathrm{PR}_{\mathbf{i}}(W) < 3(n+m-\ell).
\end{equation}
Observe that $W_{M}\cap(\mathbb{C}^{n+m}\times S_0) = W
\cap(M\times S_0)\subseteq\Theta$, so combining (\ref{eq:Jfibredimlowunconditional}) and (\ref{eq:Jfibredimupunconditional}) gives
\[\dim S_0 < 3(n+m-\ell)\]
which is a contradiction.

We deduce from this that $\dim_{G}(\mathbf{z},\mathbf{t}|C_{j})=3(n+m)$. 
By \cite[Lemma 5.2]{aek-closureoperator} this implies that $(\mathbf{z},J(\mathbf{z}))$ is generic in $V$ over $C_{j}(\mathbf{t},J(\mathbf{t}))$. 
\end{proof}

We have not added here an analogue of Theorem \ref{thm:blurred} including derivatives because \cite[Theorem 1.8]{aslanyan-kirby}, which is the available result on EC for the blurring of $J$, requires one to use a group larger than $G$ to define the blurring. 
So the result we can obtain would not be as close an approximation to Conjecture \ref{conj:ecwithderivatives} as Theorem \ref{thm:blurred} is to Conjecture \ref{conj:sec}.  

\subsubsection*{Availability of data and material} Not applicable.

\subsubsection*{Declaration} The author states that there is no conflict of interest.

\bibliographystyle{alpha}
\bibliography{gensol}{}

\end{document}